\documentclass{amsart}
\usepackage[utf8]{inputenc}
 \usepackage[toc,page]{appendix}
\usepackage{courier}
\usepackage{tikz-cd}
 \usepackage{mathtools}
 \usepackage{mathrsfs}
\usepackage{amsthm}
\usepackage{ dsfont }
\usepackage{listings}
\usepackage{color} 
\definecolor{mygreen}{RGB}{28,172,0} 
\definecolor{mylilas}{RGB}{170,55,241}

\theoremstyle{section}
\newtheorem{condition}{Condition}

\newtheorem{definition}{Definition}
\newtheorem{theorem}{Theorem}
\newtheorem{assumption}{Assumption}
\newtheorem{corollary}{Corollary}[theorem]
\newtheorem{lemma}[theorem]{Lemma}
\newtheorem{proposition}{Proposition}

\newtheorem{remark}{Remark}
\usepackage[british]{babel}
\usepackage[a4paper,top=3cm,bottom=3cm,left=3cm,right=3cm,headsep=11pt]{geometry}
\usepackage{amsmath,amssymb,amsthm,mathtools} 
\usepackage{color,graphicx} 
\DeclareGraphicsExtensions{.pdf,.png,.jpg,.svg} 
\usepackage{hyperref} 
\usepackage{enumerate} 
\allowdisplaybreaks

\DeclarePairedDelimiter\abs{\lvert}{\rvert}%
\DeclarePairedDelimiter\norm{\lVert}{\rVert}%

\makeatletter
\let\oldabs\abs
\def\abs{\@ifstar{\oldabs}{\oldabs*}}
\let\oldnorm\norm
\def\norm{\@ifstar{\oldnorm}{\oldnorm*}}
\makeatother

\newcommand{\mcI}{\mathcal{I}}

\newcommand{\mcG}{\mathcal{G}}

\newcommand{\mcR}{\mathcal{R}}

\newcommand{\mcD}{\mathcal{D}}

\newcommand{\mcM}{\mathcal{M}}

\newcommand{\HH}{\mathbb{H}}

\newcommand{\R}{\mathbb{R}}
\newcommand{\C}{\mathbb{C}}
\newcommand{\D}{\mathbb{D}}

\newcommand{\N}{\mathbb{N}}
\newcommand{\Z}{\mathbb{Z}}
\newcommand{\K}{\mathbb{K}}
\renewcommand{\P}{\mathbb{P}}
\newcommand{\E}{\mathbb{E}}
\newcommand{\rarrow}{\rightarrow}

\newcommand{\sig}{\sigma}
\newcommand{\eps}{\varepsilon}

\newcommand{\ind}{\mathds{1}}
\newcommand{\var}{\text{Var}}
\newcommand\floor[1]{\lfloor#1\rfloor}

\begin{document}
\title[Airy process at a thin rough region between frozen and smooth]{Airy process at a thin rough region between frozen and smooth}
\author{Kurt Johansson\textsuperscript{*} and Scott Mason\textsuperscript{$\dagger$}}
\thanks{\textsuperscript{*}Department of Mathematics, KTH Royal Institute of Technology, kurtj@kth.se}
\thanks{\textsuperscript{$\dagger$}Department of Mathematics, KTH Royal Institute of Technology, scottm@kth.se}
\thanks{Supported by the grant KAW 2015.0270 from the Knut and Alice Wallenberg Foundation.}
\maketitle
\begin{abstract}
We show there is a last path at the rough smooth boundary of the two-periodic Aztec diamond with parameter $a\in (0,1)$ that, suitably rescaled, converges to the Airy process, under the condition that $a$ tends to zero as the size of the Aztec diamond tends to infinity at a certain rate. This condition causes the rough region to have a thin, mesoscopic width. We also show that the dimers are described by a discrete Bessel kernel when the width is only of microscopic size.
\end{abstract}

\section{Introduction}We consider a dimer, or random domino tiling model, called the two-periodic Aztec diamond. This is a dimer model on the Aztec diamond graph of size $n$ with a two-periodic weight structure. This model was introduced in \cite{CY} and further studied in \cite{C/J} where a double contour integral formula for the inverse Kasteleyn matrix was given. The model has
two parameters $a$ and $b$ which describe the two-periodicity, see below for the precise definitions. The model is interesting since we can see all three types of possible phases, frozen, rough and smooth simultaneously in different regions. 

We consider a case where the parameters depend on $n$ the size of the Aztec diamond. Specifically, we take $a=n^{-1+\gamma}$, where $\gamma\in(0,1/2)$, and $b=1$. In this case, the width $an=n^{\gamma}$ of the rough region separating the frozen and smooth regions is mesoscopic in size. In the macroscopic limit, when we take the mesh size to be of order $1/n$, we see a frozen region meeting a smooth region. In the surface picture, we have two facets with different slopes meeting together. It is not obvious how to define the microscopic interface between the rough region and the smooth region in the two-periodic Aztec diamond, a problem that is addressed in \cite{BCJ2}. In that paper a path is defined that consists of a sequence of dimers of weight $a$, which we call $a$-\emph{dimers}. This path, the last path, starts at the bottom boundary of the Aztec diamond and ends along the left boundary. If one crosses this path, an associated
height function changes between $n-1$ and $n$. Heuristically, this path is a certain long (of typical length proportional to $n$) connected subset of a level set of the height function. Similarly there are other long paths of $a$-dimers describing height changes. In \cite{BCJ2} a so-called \emph{corridor} height function was defined and these long paths describe where this
height function changes. 

The main result of \cite{BCJ2}  was that the height changes of the corridor height function close to the rough-smooth boundary converges to the counting function in
the Airy kernel point process in a certain weak sense. We prove that when $a=n^{-1+\gamma}$ the positions of the $a$-dimers themselves close to the rough-smooth boundary
converge to the Airy kernel point process. Moreover, it is conjectured in \cite{BCJ2} that the last path converges to the Airy process for any fixed value $a\in (0,1)$. In this paper we prove this conjecture
for $a=n^{-1+\gamma}$, $\gamma\in(0,1/2)$. Although all possible dimer configurations are still possible with this choice of parameters we have a simpler geometric situation with virtually no "loops" in the smooth region (with high probability), i.e. very few $a$-dimers are seen in the smooth region. The smooth region still has some randomness though. In the case $\gamma=0$, where the width of the rough region $an\to 4\nu>0$ remains finite (i.e. has microscopic width in the limit), we see that the point process defined by the positions of the
$a$-dimers is described by the discrete Bessel point process.

The structure of this article is as follows. In section \ref{Defresult}, we recall the definition of the model, the Airy process, the $a$-height function and specify the last path which converges to the Airy process. We then state the main theorem of the paper, Theorem \ref{propmaxpathconvfdd} i.e. convergence of the last path to the Airy process. We then define and give results on several events of high probability which allow us to gain control over the whereabouts of the path. In section \ref{formresults}, we recall formulas from \cite{C/J}, including a double contour integral formula for the correlation kernel of the dimer point process. We then state the convergence of this kernel to the extended-Airy kernel. In section \ref{asympBeps1eps2}, we perform an asymptotic analysis on one of the double contour integrals which appear in the formula for the correlation kernel, showing that it converges to the Airy part of the extended-Airy kernel. At the end of the section we have a remark where we prove convergence to the discrete Bessel kernel mentioned above as a byproduct of the analysis. In section \ref{asympk11inv}, we perform an asymptotic analysis on a single integral which appears in the formula for the correlation kernel, showing that it limits to the Gaussian part of the extended-Airy kernel. In section \ref{backtracksection}, we show that the probability of seeing a backtracking dimer along certain lines goes to zero. In section \ref{aheightsection}, we show that with probability tending to one, the $a$-height function value at a finite collection of points at the top of the scaling window the last path is equal to $n$. We show this via Chebyshev's/Markov's inequality, that is, we compute the limit of the expectation and variance of the height at these points. In section \ref{Misclemmas}, we include some miscellaneous lemmas regarding formula simplification and simple bounds.
\section{Definitions and results}

 \label{Defresult}
 
\subsection{Definition of the model} Consider the subset $V\subset \Z^2$  consisting of black vertices $B$ and white vertices $W$, with $V = W\cup  B$, and
\begin{align*}
W=\{(i,j): i \ \text{mod } 2=1, \ j \ \text{mod } 2=0, 1
\leq i\leq2n-1, 0\leq j\leq 2n\}
\end{align*}
and
\begin{align*}
B=\{(i,j): i \ \text{mod } 2=0, \ j \ \text{mod } 2=1, 1
\leq i\leq2n, 0\leq j \leq 2n-1\}.
\end{align*}
We define the vertex set $V$ as the vertex set of the \emph{Aztec Diamond graph} $AD$ of size $n$ with corresponding edge set given by all $(b,w)$ such that $b-w=\pm \vec{e}_1, \pm\vec{e}_2$ for all $b\in B, w\in W$, where $\vec{e}_1=(1,1)$, $\vec{e}_2=(-1,1)$. We will also sometimes consider edges as the lines they represent in the plane, i.e. an edge is the straight line with endpoints given by its corresponding vertices. For an Aztec Diamond of size $n=4m, m\in \N_{>0}$ define the weight as a function $w$ from the edge set into $\R_{> 0}$ such that the edges contained in the smallest cycle surrounding the point $(i,j)$ where $(i+j)$ mod $4=2$, have weight $a\in(0,\infty)$ and the edges contained in the smallest cycle surrounding the point $(i,j)$ where $(i+j)$ mod $4=0$ have weight $b\in (0,\infty)$. Each of these cycles is the boundary of a face of $AD$ and we call each of these faces an $a$ face ($b$ face) if the edges on its boundary each have weight $a$ ($b$).  The graph $AD$ comes with an orientation on edges which we define as being oriented from white to black vertices.
We divide the white and black vertices into two different types. For $i\in\{0,1\}$,
\begin{align*}
&B_i=\{(x_1,x_2)\in B : x_1+x_2 \text{ mod } 4 = 2i+1\},
\end{align*}
and
\begin{align*}
&W_i= \{(x_1,x_2)\in W : x_1+x_2 \text{ mod } 4=2i+1\}.
\end{align*}\\
\begin{figure}[h]
\centering
\includegraphics[width = 0.5\textwidth]{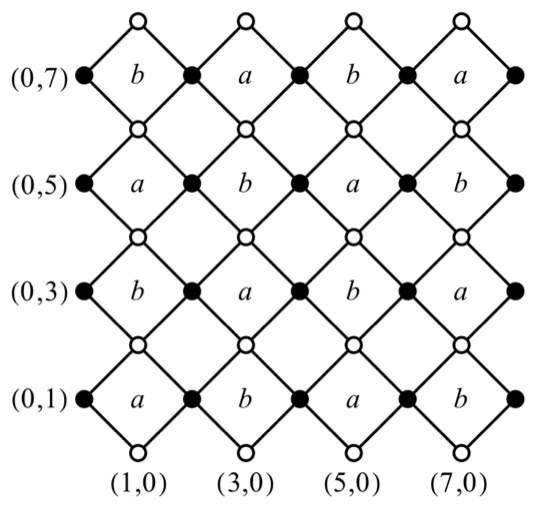}
\caption[Aztec Diamond graph]
	{The two-periodic Aztec Diamond graph for $n=4$ with the $a$ and $b$ faces labelled.}
	\label{AztecDiamondgraphn=4}
\end{figure}

Recall that a dimer configuration is a subset of edges such that every vertex belongs to \emph{exactly} one edge.
Define a probability measure $\P_{Az}$ on the finite set of all dimer configurations $\mcM(AD)$ of $AD$. For a dimer configuration $\omega\in\mcM(AD)$,
\begin{align*}
\P_{Az} (\omega) = \frac{1}{Z}\prod_{e\in \omega}w(e), \qquad \text{where}  \ \ Z=\sum_{\omega\in \mcM(AD)}\prod_{e\in\omega } w(e)
\end{align*}
is the partition function and the product is over all edges in $\omega$.
We call the probability space corresponding to $\P_{Az}$ and $\mcM(AD)$ the \emph{two-periodic Aztec Diamond}, and note that the setup here is the same as in \cite{C/J}.\\
\subsection{Kasteleyn's approach and dimer statistics} The classical approach to analyse the statistical behaviour of random dimer configurations of large bipartite graphs $G$ is to follow an idea introduced by Kasteleyn. In this approach, one puts signs $+1,-1$ (called a Kasteleyn orientation) into a submatrix of the weighted adjacency matrix indexed by $B'\times W'$ for the black and white vertices, $B'$ and $W'$, of $G$. The resulting matrix $K$ is called the \emph{Kasteleyn matrix} and has the property that the partition function of the dimer model is equal to the absolute value of the determinant of $K$. There are many good introductions to the utility of the Kasteleyn matrix, see for example \cite{Tonin}. In general, the Kasteleyn orientation is not unique, and its values need not be restricted to $1,-1$. Here we introduce the Kasteleyn matrix $K_{a,b}$ that we use for the two-periodic Aztec Diamond model of size $n=4m$. Define
\begin{align}
K_{a,b}(x,y)=\begin{cases} a(1-j)+bj & \text{if } y=x+\vec{e}_1, x\in B_{j}\\
i(aj+b(1-j)) & \text{if } y=x+\vec{e}_2, x\in B_{j}\\
aj+b(1-j) & \text{if } y=x-\vec{e}_1, x\in B_{j}\\
i(a(1-j)+bj) & \text{if }  y=x-\vec{e}_2,x\in B_{j}\\
0 & \text{otherwise}
\end{cases}
\end{align}
 where $i=\sqrt{-1}$. For dimers $e_1,...,e_n$, define the $n$-point correlation function
 \begin{align}
 \rho_n(e_1,...,e_n)=\P_{Az}(\omega\in \mcM(AD): e_1,...,e_n\in \omega).
  \end{align}
One can use the determinantal expression of the partition function to show that collections of dimers form a determinantal point process. Indeed, a theorem from Kenyon \cite{K} gives that for $e_i=(b_i,w_i), \ i=1,...,n$, the $n$-point correlation functions are 
\begin{align}
\rho_n(e_1,...,e_n)=\det(L(e_i,e_j))_{i,j=1}^n\label{corrfuncs}
\end{align}
with correlation kernel
\begin{align}
L(e_i,e_j)=K_{a,b}(b_i,w_i)K_{a,b}^{-1}(w_j,b_i).
\label{corrkernel}
\end{align}
In the above, $K_{a,b}^{-1}(w_j,b_i)$ is the inverse of the Kasteleyn matrix $K_{a,b}$ evaluated at $(w_j,b_i)$.
\subsection{The Airy Process}
Let $\Gamma_+=\{te^{i\pi /6}: t\geq 0\}\cup\{te^{5i\pi /6}:t\geq 0\}$ be a curve oriented from right to left and define $\Gamma_-$ to be its reflection in the real line and oriented from right to left.
Define the \emph{extended-Airy kernel} for $\tau,\tau',\xi,\xi'\in \R$ by
\begin{align}
A(\tau,\xi;\tau',\xi')=\tilde{A}(\tau,\xi;\tau',\xi')-\Psi(\tau,\xi;\tau',\xi')\ind_{\tau<\tau'}
\end{align}
where
\begin{align}
\tilde{A}(\tau,\xi;\tau',\xi')=\frac{e^{\frac{1}{3}(\tau'^3-\tau^3)+\tau \xi-\tau'\xi'}}{i(2\pi i)^2}\int_{\Gamma_-}dw\int_{\Gamma_+}dz\frac{e^{\frac{i}{3}z^3-\tau'z^2+iz(\xi'-\tau'^2)}}{e^{\frac{i}{3}w^3-\tau w^2+iw(\xi-\tau^2)}}\frac{1}{z-w}\label{airypart}
\end{align}
and
\begin{align}
\Psi(\tau,\xi;\tau',\xi')=\frac{1}{\sqrt{4\pi(\tau'-\tau)}}\exp\Big( -\frac{(\xi-\xi')^2}{4(\tau'-\tau)}-\frac{1}{2}(\tau'-\tau)(\xi+\xi')+\frac{1}{12}(\tau'-\tau)^3\Big).\label{gaussianpart}
\end{align}
The Airy process $t\rarrow A(t)$ is a stationary process on $\R$ which can be specified by its finite dimensional distributions. For $\xi_i\in \R$, $1\leq i\leq m$ and distinct times $\tau_1,...,\tau_m$ in $\R$, define $f$ on $\{\tau_1,...,\tau_m\}\times \R$ by $f(\tau_j,x)=-\ind_{(\xi_j,\infty)}(x)$ and then
\begin{align}
\P[A(\tau_1)\leq \xi_1,...,A(\tau_m)\leq \xi_m)]=\text{det}(\ind +f^{1/2}Af^{1/2})_{L^2(\{\tau_1,...,\tau_m\}\times \R)}\label{temp12}
\end{align}
where ${L^2(\{\tau_1,...,\tau_m\}\times \R)}$ has reference measure $\lambda \otimes \mu$, where $\lambda$ is the counting measure on $\{\tau_1,...,\tau_m\}$ and $\mu$ is the Lebesgue measure on $\R$. Since $f^{1/2}Af^{1/2}$ is a trace class operator, \eqref{temp12} is a Fredholm determinant of a trace class operator.
\subsection{Geometric definitions}
The definition of the path we consider was put forward in \cite{BCJ2}, we give a summary of relevant definitions in that article.
In \cite{BCJ2}, a squishing procedure is outlined on the two-periodic Aztec Diamond graph where one contracts (or squishes) the $b$ faces whilst expanding the $a$ faces. The resulting graph, which we label $\widetilde{AD}$, consists only of $a$ edges. Due the dimer constraint, \cite{BCJ2} show that $a$-dimers in $\widetilde{AD}$ are either part of double edges, or oriented \emph{loops} or \emph{paths} which we will recall the definitions of.
Dimer models have an associated \emph{height function} interpretation. In our context, a height function is defined on faces of the graph $AD$ and the height function corresponding to a dimer configuration is specified by letting its value to be one at the face (0,0) (the outside face of $AD$) and then defining the following height changes between adjacent faces in $AD$;\\\\
a height change of $+ 3$ $ (-3)$ when traversing across an edge covered by a dimer with the white vertex on the right (left),\\\\
a height change of $+ 1$ $ (-1)$ when traversing across an edge \emph{not} covered by a dimer with the white vertex on the left (right).\\\\

One typically thinks of the presence of a dimer in a configuration as giving a steeper height change (modulus $3$ instead of $1$) of the height function between the two faces incident to the dimer, from this one then sees that the map that sends dimer configurations to their height functions is injective. Next, define the $a$-\emph{height function} to be the restriction of the height function to the $a$ faces. One can check that the height differences of the $a$-height function are always a multiple of $4$, see figure 5 in \cite{BCJ2}.

We define double edges as those edges in the squished graph which are the result of two $a$-dimers contracting to the same edge. \\We define a \emph{loop} of length $k\geq 4$ to be a sequence of distinct $a$-dimers $e_{2j-1}, 1\leq j\leq k$ which are not part of a double edge, and with the following properties:
\begin{enumerate}
\item There are distinct $b$ edges (not covered by dimers) $e_{2i}, 1\leq i \leq k$ such that the $b$ edge $e_{2i}$ shares one endpoint with the $a$-dimer $e_{2i-1}$ and the other endpoint with the $a$-dimer $e_{2i+1}$, $1\leq i \leq k$.
\item The sequence forms a loop in the sense that $e_{-1}=e_{2k-1}$, $e_{2k+1}=e_1$. 
\end{enumerate}
The collection of $e_{j}, 1\leq j\leq 2k$ form a sequence of adjacent edges and alternate between being $a$-dimers and $b$ edges, and they visually form a loop. After the squishing procedure the $b$ edges are contracted and one obtains a loop of $k$ $a$-dimers. The orientation along any $a$-dimer in a loop is given by taking the orientation of the dimer as from its white vertex to its black vertex in the pre-squished graph.
A path is defined the same as a loop but instead of the loop condition $(2)$, the paths start and end somewhere along the boundary of $AD$. Specifically, define a \emph{path} of length $k\geq 1$ to be a sequence of distinct $a$-dimers $e_{2j-1}, 1\leq j\leq k$ which are not part of a double edge, and with the following properties:
\begin{enumerate}
\item There are distinct $b$ edges (not covered by dimers) $e_{2i}, 1\leq i \leq k$ such that the $b$ edge $e_{2i}$ shares one endpoint with the $a$-dimer $e_{2i-1}$ and the other endpoint with the $a$-dimer $e_{2i+1}$, $1\leq i \leq k$.
\item $e_1$ and $e_{2k-1}$ are incident to the outside face of $AD$. 
\end{enumerate}
 
 Note there is an ambiguity in the definition of distinct loops and paths when two or more loops or paths intersect, this is resolved by picking a "mirror" at an arbitrary intersection which distinguishes the loops, see \cite{BCJ2}. We assume the same convention, however this issue will not concern the analysis here. The orientation on paths is given by the same as that on loops (from white to black in the pre-squished graph). When stepping into a counterclockwise loop the a-height function decreases by $4$, when stepping into a clockwise loop the a-height function decreases by $4$. From this it is reasonable to define the contribution of the loops to the $a$-height function $h^{\ell}(f)$ at a face $f$ as 4 times the number of clockwise loops surrounding $f$ minus 4 times the number of counterclockwise loops surrounding $f$. Similarly, when stepping over a path the $a$-height function changes by $\pm4$ depending on the orientation of the path. 

Define $\Gamma_{i,m}$ to be the oriented paths which separate the $a$-height $4i$ and $4i+4$, $0\leq i \leq 2m-1$, this depends on the dimer configuration. Since the height function is monotonic along the boundary faces of $AD$ (decreases from left to right on the top boundary, increases from left to right on the bottom boundary etc.), the paths $\Gamma_{i,m}$ partition the faces of the squished graph $\widetilde{AD}$ into collections of faces called \emph{corridors}. The $i^{th}$ corridor $C_i$ is defined as all $a$ faces in $\widetilde{AD}$ that are bounded between $\Gamma_{i-1,m}$ and $\Gamma_{i,m}$, $0<i\leq 2m-1$, $C_0$ is all faces in $\widetilde{AD}$ bounded between $\Gamma_{0,m}$ and the boundary of $\widetilde{AD}$ and  $C_{2m}$ is all faces in $\widetilde{AD}$ bounded between $\Gamma_{2m-1,m}$ and the boundary of $\widetilde{AD}$.
For a face  $f$ in $C_i$, define the contribution of the paths (the "corridor height" in \cite{BCJ2}) to the $a$-height function $h^c(f)$ as equal to $4i$. Since every $a$-dimer is a double edge, loop or path, one can argue that the $a$-height function, $h^a$, is just the sum of the contribution from the loops and paths i.e. $h^a(f)=h^\ell(f)+h^c(f)$. 

The paths $\Gamma_{i,m}$ split into two parts $\Gamma_{i,m}^t, \Gamma_{i,m}^b$ which start at the top and bottom boundaries respectively, the path $\Gamma_{i,m}^b$ is the path mentioned in the introduction, and is the path we consider in this article. In \cite{BCJ2}, the authors expect that for all fixed $a\in(0,1)$ there is an $i$ close to $m$ such that $\Gamma_{i,m}^b$ converges to an Airy process in a suitable rescaling. In this article, we will show that precisely the $i=m-1$ (suitably rescaled) path $\Gamma_{m-1,m}^b$ converges in fdd to the Airy process when we take the parameter $a=a(n)$ to zero with $n$ at certain rates.
\subsection{Dimer coordinates and a piecewise continuous function constructed out of $\Gamma_{m,m-1}^b$} From here we assume $n$ is large enough that a number of coordinate and parameter definitions below are well-defined.
Fix $\gamma\in(0,1/3)$ and, unless otherwise stated, we let 
\begin{align}
a\sim n^{-1+\gamma},&& b=1\label{definitionofa}
\end{align}
for the remainder of the article. 
 To make use of previous results in \cite{C/J} we also use the notation $c=a/(1+a^2)\in(0,1/2)$. This particular weighting on $a$ and $b$ has the effect of shrinking the rough region in the mesh limit, see (Figure). We will investigate the asymptotics of the correlation kernel restricted to the $a$-dimers in the following region. As in \cite{C/J} we consider the bottom left quadrant of $AD$. Let $-1<\xi<0$ so that $(n(1+\xi),n(1+\xi))$ is the coordinate varying over the diagonal of the bottom left quadrant of the Aztec Diamond. 
As in \cite{C/J}, we will use $(n(1+\xi_c), n(1+\xi_c))$ as a reference point at the rough-smooth boundary. We want it to have integer coordinates, i.e. we assume that 
 \begin{align*}
n(1+\xi_c) =\max\{2\N_{\geq 0}\cap [n(1-\frac{1}{2}\sqrt{1+2c},n(1-\frac{1}{2}\sqrt{1-2c})]\}.
\end{align*}
From \cite{C/J}, the rough-frozen boundary has as a reference point $n(1+\xi_f)$  where $\xi_f=-\frac{1}{2}\sqrt{1+2c}$, see the discussion following Theorem 2.6 in \cite{C/J}. Hence we see the width of the rough region is $n(1+\xi_c)-n(1+\xi_f)\sim an$.

We will see that since we consider the weight $a$ tending to zero, the Airy-type fluctuations \cite{C/J} of $n^{1/3}$, $n^{2/3}$ in the vertical and horizontal directions respectively are altered, in particular, the fluctuations become of the order
\begin{align}
 p_n=(an)^{1/3} , && q_n=(n^2/a)^{1/3}.\label{definitionofpnqn}
 \end{align}
 Observe that for fixed $a$, $p_n \sim n^{1/3}, q_n\sim n^{2/3}$.
 
 We note that for $a\sim n^{-1}$, we see that the a-dimers have \emph{Bessel}-type fluctuations, see remark \ref{BesselKernel}.
 \begin{definition} Let $S$ be the interior of the box with corners placed at the four points
\begin{align}
&(n(1+\xi_c)+1)\vec{e}_1\pm\alpha_{\pm}  p_n\vec{e}_1\pm\beta q_n\vec{e}_2,\quad(n(1+\xi_c)+1)\vec{e}_1\mp\alpha_{\mp}  p_n\vec{e}_1\pm\beta q_n\vec{e}_2\nonumber\\
&\text{where }(\alpha_{\pm} p_n,\beta q_n)\in (2\mathbb{Z})^2,\quad \text{$\alpha_-=\beta>>1$ is large but fixed, $\alpha_+ =\beta\log(n)$} \nonumber.
\end{align}
\end{definition}
Set $\alpha:=\alpha_+=\beta\log(n)$. One sees that the corners are placed at the centre of $a$ faces hence $S$ contains the same $a$ edges of $AD$ before and after the squishing procedure. We take the condition  $\alpha=\beta\log(n)$, since we will prove and use the fact that the variance of the height function along the top boundary of the box $S$ (where the "top" is furthest boundary in the $\vec{e}_1$ direction) goes to zero. This would not happen if we took $\alpha $ fixed. We also single out the "top" boundary line of $S$, define $\partial S^T$ to be the straight line with endpoints 
\begin{align}
(n(1+\xi_c)+1)\vec{e}_1+\alpha  p_n\vec{e}_1\pm\beta q_n\vec{e}_2, \nonumber
\end{align}
the line $\partial S^T$ intersects $a$ faces in the squished graph, it does not intersect $b$ faces.
 \begin{condition}
Let $j$ lie in a finite set $J$. For any $a$-dimer $e_j=(x,y)\in W_{\eps_1}\times B_{\eps_2}$ in the pre-squished graph $AD$, with $x=(x_1,x_2)$ and $y=(y_1,y_2)$ let
 \begin{align}
 x=(n(1+\xi_c)+1+\alpha_jp_n)\vec{e}_1+\beta_j q_n\vec{e}_2+(0,2\eps_1-1)\label{xlocs},\\
 y=(n(1+\xi_c)+1+\alpha_jp_n)\vec{e}_1+\beta_j q_n\vec{e}_2+(2\eps_2-1,0)\label{ylocs}
 \end{align}
 where 
\begin{align}
(\alpha_j p_n,\beta_j q_n)\in (2\Z)^2, &&\text{and } \quad\quad|\beta_j|\leq \beta,\quad -\beta\leq \alpha_j\leq \beta\log(n).
\label{alphabetarestric}
\end{align}
\label{cond-a-dimers}
 \end{condition}
 We think of $\Gamma_{m-1,m}^b$ as a path or a sequence of $a$-dimers in the squished graph. 
 Consider $\Gamma_{m-1,m}^b$ in the squished graph intersected with a line $t\vec{e}_2+\R \vec{e}_1$, $t\in \Z$. 
  If we look at a configuration such that $\Gamma_{m-1,m}^b$, viewed as a path, leaves a point along this line and moves back around to the same line we see it may intersect a line many times. If we want to think of a statistic of $\Gamma_{m-1,m}^b$ which is a random function converging to the Airy process, a reasonable definition would be to just look at the points furtherest along given lines of the form $t\vec{e}_2+\R \vec{e}_1$ and then make it piecewise constant. However, first we try to look at the furtherest points along $t\vec{e}_2+\R \vec{e}_1$ that are still \emph{below} $\partial S^T$. We use this as a statistic and then show that in fact these points are really the only points of $\Gamma_{m-1,m}^b$ on the lines $t\vec{e}_2+\R \vec{e}_1$ (up to an event with probability going to zero). 
This all motivates the following definition of a piecewise continuous function constructed out of each realisation of $\Gamma_{m,m-1}^b$.
Observe that for $x\in \R\setminus(2\Z+1)$,
\begin{align} 2\floor{\frac{x-1}{2}}+2\in 2\Z\label{templbiald}
\end{align} is the nearest even number to $x$. For $x\in 2\Z+1$, \eqref{templbiald} is the nearest even number greater than $x$.
\begin{definition}
For $t$ such that $tq_n\in[-\beta q_n,\beta q_n]\cap 2\Z$, let $\Gamma(t)$ be the largest $i\in (-\infty,\alpha ]$ such that $\Gamma_{m-1,m}^b$ intersects $(n(1+\xi_c)+ip_n) \vec{e}_1+t q_n\vec{e}_2$ in the \emph{squished} graph $\widetilde{AD}$. If there is no such $i$, define $\Gamma(t)$ to be the largest $i\in 2\Z/p_n$ such that $\Gamma_{m-1,m}^b$ intersects $(n(1+\xi_c)+ip_n) \vec{e}_1+t q_n\vec{e}_2$ in the \emph{squished} graph $\widetilde{AD}$. Extend $\Gamma(t)$ to the rest of the values of $t\in [-\beta ,\beta]\setminus 2\Z/q_n$ by making it piecewise constant on each interval $tq_n\in[-\beta q_n,\beta q_n]\cap [j-1,j+1)$, $j\in 2\Z$ i.e. 
\begin{align}
\Gamma(t)=\Gamma((2\floor{\frac{tq_n-1}{2}}+2)/q_n)
\end{align} for $-\beta <t<\beta $.
\end{definition} 
We now state the main theorem of this article.
\begin{theorem}
For a collection of fixed coordinates $(t_1,\xi_1),...,(t_j,\xi_j)\in \R^2$ such that $t_1<...<t_j$, then
\begin{align}
\lim_n \P(\Gamma(t_1)\leq \xi_1,...,\Gamma(t_j)\leq \xi_j)= \P(A(t_1)\leq \xi_1+t_1^2,...,A(t_j)\leq \xi_j+t_j^2).
\label{maxpathconvfdd}
\end{align}
where $A(t)$ is the Airy process. 
\label{propmaxpathconvfdd}
\end{theorem}
\subsection{Gaps of $a$-dimers and events of high probability}
In order to prove Theorem \ref{propmaxpathconvfdd}  we control the cylinder sets appearing on the left hand side of \eqref{maxpathconvfdd} by rewriting them in terms of gap events of $a$-dimers and other events whose complements have probability zero in the limit $n\rarrow \infty$.
In particular, for a given $(t_i,\xi_i)$, $1\leq i \leq j$, define the gap event of seeing no $a$-dimers along specific lines,
\begin{align}
&\{\text{no $a$-dimer above $(t_i,\xi_i)$}\}\nonumber\\&:=\{\text{no $a$-dimer in the squished graph intersects the lines}
\nonumber\\&\qquad\qquad \{(t_iq_n \vec{e}_2+(n(1+\xi_c)+(\xi_ip_n,\infty)\vec{e}_1))\cap S, 1\leq i \leq j\}\}.\nonumber
\end{align}
For a set of $j$ points $(t_1,...,t_j)\in \R^j$, define $\partial S^T_{\{t\}_j}$ to be the subset of faces along $\partial S^T$ given by 
\begin{align}
\{(n(1+\xi_c)+1+\alpha p_n)\vec{e}_1+t_iq_n\vec{e}_2: i\in \{1,...,j\}\}.
\end{align}
We also define the event
\begin{align}
&\{\text{$a$-height along $\partial S^T_{\{t\}_j}$ $=4m$}\}\nonumber\\&:=\{\text{the $a$-height function values at faces $\partial S^T_{\{t\}_j}$ are all equal to $4m$}\}.\nonumber
\end{align}

Heuristically, since $\Gamma_{m-1,m}^b$ is related to the level curves of the $a$-height function, on configurations where the above event holds, we expect the path $\Gamma_{m-1,m}^b$ will be encountered when travelling from $\partial S^T$ to the bottom left corner of the Aztec diamond (we will need to discount the possibility of encountering $\Gamma_{m-1,m}^t$). Finally, we give two events regarding the appearance of loops and double edges in the squished graph. Consider the set of loops $\mcD_\ell$ in $AD$ corresponding to a given dimer configuration, a loop intersects a set of $a$ edges $S'$ if it has an $a$-dimer in common with $S'$. Denote $\ell(\gamma)$ for the length of a loop $\gamma\in \mcD_\ell$. We have the following lemma from \cite{BCJ2},
\begin{lemma}\label{loopsbound}
Let $S'$ be a set of $a$-edges in $AD$, assume $a\in(0,1/3)$. Then
\begin{align}
\P_{Az}(\exists \gamma\in \mcD_\ell \text{ that intersects $S'$ and has length $\ell(\gamma)\geq d$})\leq \frac{|S'|(3a)^d}{(1-3a)}
\end{align}
where $|S'|$ is the cardinality of $S'$.\end{lemma}
A similar lemma holds for double edges, which we now recall.
Let $\mcD_e$ be the set of all sequences of distinct edges $\gamma=(e_1,...,e_{2k})$ such that
\begin{enumerate}
\item $e_i$ shares endpoints with $e_{i-1}$ and $e_{i+1}$ for $0\leq i\leq 2k$ with $e_0=e_{2k}$ and $e_{2k+1}=e_1$,
\item $e_{2i+1}$ are $a$ edges while $e_{2i+2}$ are $b$ edges for $0\leq i \leq k-1$,
\item the pairs $(e_{2i+1},e_{2k-(2i+1)})$ form double edges after the squishing procedure for all $0\leq i \leq k-1$,
\item $\gamma$ is not incident to any other double edges.
\end{enumerate}
Let $\ell_e(\gamma)$ be the number of $a$-dimers in a given $\gamma\in \mcD_e$.
\begin{lemma}
Let $S'$ be a set of $a$-edges in $AD$, assume $a\in(0,1/3)$. Then
\begin{align}
\P_{Az}(\exists \gamma\in \mcD_e \text{ that intersects $S'$ and has length $\ell(\gamma)\geq d$})\leq \frac{|S'|(3a)^d}{(1-3a)}
\end{align}
where $|S'|$ is the cardinality of $S'$.
\end{lemma} The previous two lemmas are proved via a Peierls-type argument. We will use them to control the appearance of double edges and loops along a finite collection of straight lines in the Aztec diamond. 
From now on, let $S'$ be the collection of $a$ edges intersecting the lines $(-\infty,n(1+\xi_c)+\alpha p_n] \vec{e}_1+t_i q_n\vec{e}_2$, $1\leq i \leq j$.
Define
\begin{align}
\{\text{no loops intersect $S'$}\}=\{\exists \gamma\in \mcD_\ell \text{ that intersects $S'$ and has length $\ell(\gamma)\geq 4$}\}^c.
\end{align}
and
\begin{definition}
A backtracking dimer $e\in B\times W$ is an $\emph{a}$-dimer which has both endpoints in either:\\
 The lower left quadrant ($[0,n]\times [0,n]$) and which is an element of $B_1\times W_1$, or\\
 the upper left quadrant ($[0,n]\times [n,2n]$) and which is an element of $B_1\times W_0$, or\\
 the upper right quadrant ($[n,2n]\times [n,2n]$)  and which is an element of $B_1\times W_1$, or\\
 the bottom right quadrant ($[n,2n]\times [0,n]$)  and which is an element of $B_0\times W_1$,
 or in the box with sides of length $\log(n)q_n$ centred at the point $(n,n)$.
 \end{definition}
Define two paths 
\begin{align}
\gamma_1^-&=(2n-2,0)+t\vec{e}_2, t\in[0,n-1], \label{antidiagpaths}\\
\gamma_1^+&=(2n-2,0)+t\vec{e}_2, t\in[n-1,2n-2]
\end{align} so that $\gamma_1^-\cup \gamma_1^+$ is a straight line across the anti-diagonal of the Aztec diamond (shifted by $-\vec{e}_1$ so that it passes through $a$-faces). Observe that for large $n$ enough that $\log(n)q_n>4$,
\begin{align}
\{\text{$\nexists$ backtracking dimers along }\gamma_1^-\}=\{\text{a-height is monotonically decreasing along }\gamma_1^-\},\\
\{\text{$\nexists$ backtracking dimers along }\gamma_1^+\}=\{\text{a-height is monotonically increasing along }\gamma_1^+\}.
\end{align}

Note that in the above statements, we used the condition that backtracking dimers are all $a$-dimers in the box with sides of length 4 centred at the point $(n,n)$. This is because the anti-diagonal paths cross into the lower left quadrant. 
Define $\gamma$ to be the union of $\gamma_1^-$ and $\gamma_1^+$.

For fixed $t\in \R$, we also define the path $\gamma_t'$ to be the path $\{s\vec{e}_1+t'q_n\vec{e}_2: s\}$ starting at a boundary face in the lower left quadrant and ending at the line $\gamma$, where $t'=(2\floor{\frac{tq_n-1}{2}}+2)/q_n$. 
\begin{lemma}\label{controloverGamma}
Fix $(t_1,..,t_j)\in \R^j$, $(\xi_1,...,\xi_j)\in \R^j$. For $n$ large, 
\begin{align}
&\{\text{no $a$-dimer above $(t_i,\xi_i)$}\}\cap \{\text{$a$-height along $\partial S^T_{\{t\}_j}=4m$}\}\\&\cap \{\text{no loops intersect $S'$}\}\cap\{\text{$\nexists$ backtracking dimers along }\gamma\cup_i\gamma_{t_i}'\}\nonumber\\&\subset \{\Gamma(t_1)\leq \xi_1,...,\Gamma(t_j)\leq \xi_j\}.\nonumber
\end{align}
\begin{proof}
Consider the case $j=1$. Since there is no $a$-dimer intersecting the line $(n(1+\xi_c)+(\xi_i,\alpha ]p_n)\vec{e}_1+t_iq_n\vec{e}_2$, the $a$ height change is zero along this line. Hence the $a$ height is equal to $4m$ at the face $(n(1+\xi_c)+1+\xi_ip_n)\vec{e}_1+t_iq_n\vec{e}_2$. Recall the boundary values of the height function are fixed. The line $X\subset S'$ extending from this face down to the boundary of $AD$ has height $\sim t_iq_n<<m$, hence there must be a height change from $4m-4$ to $4m$ along $X$. The height changes can only come from paths since no loops intersect $X$, and since the height change is from $4m-4$ to $4m$, either  $\Gamma_{m-1,m}^b$ or $\Gamma_{m-1,m}^t$ passes through $X$.  We now need to discount the latter possibility.  

Recall paths have an orientation, whereby travelling from one vertex to another in the squished graph, the lower height is on the left and the higher height is on the right.
Recall the path $\Gamma_{m-1,m}^t$ starts at the \emph{top} boundary (and ends at either the left or right boundary).

Assume that an $a$-dimer $e$ of $\Gamma_{m-1,m}^t$ intersects $X(\subset \gamma_{t_1}')$. The orientation of the path whilst crossing $X$ must be in the direction $\vec{e}_2$ (i.e. travelling from white to black vertices), as there are no backtracking dimers on $X$. The white vertex of $e$ is then the endpoint of the subsection $\Gamma_e$ of $\Gamma_{m-1,m}^t$ starting at the top boundary. The path $\gamma_{t_1}'\cup\gamma$ partitions $AD$ into  three components, label the component containing the white vertex of $e$ by $R$. By connectedness, the path $\Gamma_e$ must pass through $\gamma_{t_1}'$ or $\gamma$, as the path begins in $R^c$. Moreover, it must pass \emph{into} $R$. This is a contradiction, as it easy to check that the no backtracking condition only allows paths to leave $R$ (when $\log(n)> |t_1|$). Hence $\Gamma_{m-1,m}^b$ passes through $X$. The result follows by the definition of $\Gamma(t_1)$. The extension to arbitrary $j$ follows similarly.
\end{proof}
\end{lemma}
\begin{proposition}\label{expectheightprop}
\begin{align}
\P_{Az}(\text{$a$-height along $\partial S^T_{\{t\}_j}=4m$})\rarrow 1
\end{align} 
in the limit $n\rarrow \infty$.
\end{proposition}
The proof of the above Proposition can be found at the end of section \ref{aheightsection}. 
\begin{proposition} \label{backtrackpropmain} Fix $(t_1,..,t_j)\in \R^j$,
\begin{align}
\P_{Az}(\text{$\exists$ backtracking dimers along }\gamma\cup_i\gamma_{t_i}')\rarrow 0
\end{align}
as $n\rarrow \infty$.
\end{proposition}
The proof of the above Proposition can be found at the end of section \ref{backtracksection}.
\begin{proposition}Fix $(t_1,\xi_1),...,(t_j,\xi_j)\in \R^{2}$,
\begin{align}
\P_{Az}(\Gamma(t_1)\leq \xi_1,...,\Gamma(t_j)\leq \xi_j)=\P_{Az}(\text{no $a$-dimer above $(t_i,\xi_i)$})+o(1).
\end{align}
in the limit $n\rarrow \infty$.\label{pathsprobablitygaps}
\begin{proof}
Let 
\begin{align}
B_n=& \ \{\text{$a$-height along $\partial S^T_{\{t\}_j}=4m$}\}\cap \{\text{no loops intersect $S'$}\}\\ &\cap\{\text{$\nexists$ backtracking dimers along }\gamma\cup_i\gamma_{t_i}'\}.\nonumber
\end{align}
By Lemma \ref{controloverGamma},
\begin{align}
\{\text{no $a$-dimer above $(t_i,\xi_i)$ }\}\cap B_n\subset \{\Gamma(t_1)\leq \xi_1,...,\Gamma(t_j)\leq \xi_j\}\cap B_n.
\end{align}
So if we show that
\begin{align}
\{\Gamma(t_1)\leq \xi_1,...,\Gamma(t_j)\leq \xi_j\}\cap B_n\subset \{\text{no $a$-dimer above $(t_i,\xi_i)$ }\}\label{temp23ggdfgd}
\end{align}
then we are done, since $\P(B_n)\rarrow 1$ by Lemma \ref{loopsbound} and Propositions \ref{expectheightprop} and \ref{backtrackpropmain}. We show \eqref{temp23ggdfgd} for $j=1$. 

If $\Gamma(t_1)\leq \xi_1$, then $a$-dimers intersecting the line
$X'=\{(t_1q_n \vec{e}_2+(n(1+\xi_c)+(\xi_1p_n,\infty)\vec{e}_1))\cap S$ can be part of either; a loop, double edge, or a path other than $\Gamma_{m-1,m}^b$. If the $a$-height at $\partial S^T_{\{t\}_1}$ is $n$ and there are no backtracking dimers along $\gamma_{t_1}'\supset X'$, then the $a$-height is equal to $n$ along $X'$. This means no loops or paths intersect the line $X'$ (otherwise the $a$-height would change). Moreover, a double edge intersecting $X'$ must contain a backtracking edge, hence double edges can not intersect $X'$. Therefore the only possibility is that there are no $a$-dimers intersecting the line $X'$. This proves \eqref{temp23ggdfgd} for $j=1$, the extension to arbitrary $j$ is the same argument applied to each line.
\end{proof}
\end{proposition}
Since all type $W_1\times B_1$ $a$-dimers in the gap event $ \{\text{no $a$-dimer above $(t_i,\xi_i)$}\}$ are backtracking dimers, they are easily discountable. We define: for a given $(t_i,\xi_i)$, $1\leq i \leq j$,
\begin{align}
&\{\text{no $a$-dimer in $W_0\times B_0$ above $(t_i,\xi_i)$}\}\nonumber\\&:=\{\text{no $a$-dimer of type $W_0\times B_0$ in the squished graph intersects the lines}
\nonumber\\&\qquad\qquad \{(t_iq_n \vec{e}_2+(n(1+\xi_c)+(\xi_ip_n,\infty)\vec{e}_1))\cap S, 1\leq i \leq j\}\}.\nonumber
\end{align}
Hence, the inclusions
\begin{align}
 &\{\text{no $a$-dimer above $(t_i,\xi_i)$}\}\subset  \{\text{no $a$-dimer in $W_0\times B_0$ above $(t_i,\xi_i)$}\},\nonumber\\
 &\{\text{no $a$-dimer in $W_0\times B_0$ above $(t_i,\xi_i)$}\}\cap \{\text{$\nexists$ backtracking dimers along }\cup_i\gamma_{t_i}'\}\nonumber\\&\quad \quad\subset  \{\text{no $a$-dimer above $(t_i,\xi_i)$}\}\nonumber
\end{align}
give us
\begin{corollary}
Fix $(t_1,\xi_1),...,(t_j,\xi_j)\in \R^{2}$,
\begin{align}
\P_{Az}(\Gamma(t_1)\leq \xi_1,...,\Gamma(t_j)\leq \xi_j)=\P_{Az}(\text{no $a$-dimer in $W_0\times B_0$ above $(t_i,\xi_i)$})+o(1).
\end{align}
in the limit $n\rarrow \infty$.\label{pathsprobablitygapsw0b0}
\end{corollary}
We now write a determinantal formula for the gap process. Define $S_i$ as the collection of $W_0\times B_0$ $a$ edges in $S$ that intersect the line $t_iq_n\vec{e}_2+(n(1+\xi_c)+(\xi_ip_n,\infty))\vec{e}_1$ and $\tilde{g}$ as a function on $a$ edges in $S$ such that $g(e)=-1$ if $e\in \cup_iS_i$ and zero otherwise. We have
\begin{align}
\label{gapprob}&\P_{Az}(\text{no $a$-dimer in $W_0\times B_0$ above $(t_i,\xi_i)$})\\&\nonumber=\E_{Az}[\ind_{(\text{no $a$-dimer  in $W_0\times B_0$ above $(t_i,\xi_i)$}}]\\
&= \E_{Az}[\prod_{i=1}^j\prod_{e\in S_i}(1+\tilde{g}(e))]\nonumber\\
&=1+\sum_{k=1}^\infty \frac{1}{k!}\sum_{(e_1,...,e_k )}\tilde{g}(e_1)...\tilde{g}(e_k)\det(L(e_i,e_j))_{i,j=1}^k\nonumber\\
&=\det(1+\tilde{g}L)_{\cup_iS_i\times \cup_iS_i}\nonumber
\end{align}
where the second sum is over $(e_1,...,e_k )\in\{a \text{ edges in $\cup_iS_i$}\}^k$ and the matrices in the determinant are indexed by $a$ edges in $\cup_iS_i$. We can rewrite this determinant as a Fredholm determinant of an operator on a particular $L_2$ space, similar to \eqref{temp12}, by viewing the appearance of edges along lines as the appearance of particles in a point process on $\{t_1,...,t_j\}\times 2\Z/p_n\subset \{t_1,...,t_j\}\times\R$. That is, an edge $e\in \cup_iS_i$ intersects the line $t_iq_n\vec{e}_2+(n(1+\xi_c)+(\xi_ip_n,\infty))\vec{e}_1)\cap S$ at a point $t_iq_n\vec{e}_2+(n(1+\xi_c)+x p_n)\vec{e}_1$ where $x\in 2\Z/p_n$, so we identify $e$ and $x,t_i$, we denote this identification $e=\Phi(t_i,x)$. Now the Fredholm expansion in \eqref{gapprob} becomes
\begin{align}
&1+\sum_{k=1}^\infty\frac{1}{k!} \int_{(\{t_1,...,t_j\}\times \R)^k}\prod_{i=1}^k\tilde{g}(\Phi(t_i,x_i))\det(L(\Phi(t_i,x_i),\Phi(t_j,x_j)))_{i,j}^kd(\lambda\otimes \mu_{p_n})^k(t,x)\nonumber\\
&=\det(1+\tilde{g} L\circ \Phi)_{\ell^2(\{t_1,...,t_j\}\times 2\Z/p_n)}\nonumber
\end{align}
where $(t,x):=((t_1,x_1),...,(t_k,x_k))$, $\mu_{p_n}$ is the counting measure on $2\Z/p_n$ and the reference measure for $\ell^2(\{t_1,...,t_j\}\times 2\Z/p_n)$ is $\lambda\otimes \mu_{p_n}$. Observe that $\tilde{g}(\Phi(t_i,x))=-\ind_{(\xi_i,\alpha ]}(x)$.
Finally, we can rewrite this as a Fredholm determinant on $L^2(\{t_1,...,t_n\}\times\R)$ by setting
\begin{align}
\tilde{\Phi}(t,x)=\Phi(t,(2\floor{\frac{p_nx-1}{2}}+2)/p_n),
\end{align}
so then \eqref{gapprob} becomes
\begin{align}
&1+\sum_{k=1}^\infty\frac{1}{k!} \int_{(t_1,...,t_j\}\times\R)^k}\prod_{i=1}^k\tilde{g}(\tilde{\Phi}(t_i,x_i))\det(L(\tilde{\Phi}(t_i,x_i),\tilde{\Phi}(t_j,x_j)))_{i,j}^k\Big(\frac{p_n}{2}\Big)^kd(\lambda \otimes\mu)^k(t,x)\label{fredbigell2}\\
&=\det(1+\frac{p_n}{2}\tilde{g} L\circ \tilde{\Phi})_{L^2(\{t_1,...,t_j\}\times \R)}.\nonumber
\end{align}
\begin{proposition}
Fix $(t_1,\xi_1),...,(t_j,\xi_j)\in \R^2$, where $t_1,...,t_j$ are distinct, then
\begin{align}
\P_{Az}(\text{no $a$-dimer in $W_0\times B_0$ above $(t_i,\xi_i)$})= \P(A(t_1)\leq \xi_1+t_1^2,...,A(t_j)\leq \xi_j+t_j^2)+o(1)
\end{align}
in the limit $n\rarrow \infty$, where $A(t)$ is the Airy process. \label{agapsairy}
\begin{proof}
See the end of section \ref{formresults}.
\end{proof}
\end{proposition}
Together Corollary  \ref{pathsprobablitygapsw0b0} and Proposition \ref{agapsairy} prove Theorem \ref{propmaxpathconvfdd}.
\section{Formulas and Results}\label{formresults}
In this section we present results and formulas required to prove Propositions \ref{expectheightprop} through \ref{agapsairy}.
We will use the formula for $K_{a,1}^{-1}$ derived in \cite{C/J}. Before we can give the formula we first need to define the objects that come into it.
For $\eps_1,\eps_2\in\{0,1\}$, we write
\begin{equation}
h(\eps_1,\eps_2)=\eps_1(1-\eps_2)+\eps_2(1-\eps_1).\label{HHH}
\end{equation}
In many cases the coordinates of the vertices giving the position of the dimers will enter into formulas via the following expressions. Let $(x_1,x_2)\in W_{\eps_1}$, $(y_1,y_2)\in B_{\eps_2}$ and define
\begin{align}
&k_1=\frac{x_2-y_2-1}{2}+h(\eps_1,\eps_2), && \ell_1=\frac{y_1-x_1-1}{2},\label{klh1}\\
& k_2=k_1+1-2h(\eps_1,\eps_2) , && \ell_2=\ell_1+1.
\label{klh2}
\end{align}
Let $\D^*\subset \C$ denote the punctured open unit disc centred at the origin. A basic role is played by the analytic function
\begin{align}\label{Gfunction}
G: \C \setminus i[-\sqrt{2c},\sqrt{2c}] \rarrow \D^* ;  \ w\mapsto \frac{1}{\sqrt{2c}}(w-\sqrt{w^2+2c})
\end{align}
with $\sqrt{w^2 +2c}=\exp(\frac{1}{2}\log(w+i\sqrt{2c})+\frac{1}{2}\log(w-i\sqrt{2c}))$ and where the logarithm takes its argument in $(-\pi/2,3\pi/2)$. Let $\sqrt{1/w^2 +2c}$  denote the previous square root evaluated at $1/w$. Note that $G$ is the inverse of the analytic bijection $\D^*\rarrow \C \setminus i[-\sqrt{2c},\sqrt{2c}] ; \ u\mapsto \sqrt{c/2}(u-1/u)$, which is related to the Joukovski map, see chapter 6 in \cite{Mar}. We also remark that $|G|=1$ along its cut, and $G(w)\rarrow0$ as $|w|\rarrow \infty$. 
From the definition, we have the symmetries
\begin{align}
\overline{\sqrt{w^2+2c}}=\sqrt{\overline{w}^2+2c}, && -\sqrt{w^2+2c}=\sqrt{(-w)^2+2c}
\label{squarerootsymmetries}
\end{align}
which give
\begin{align}
&\overline{G(w)}=G(\overline{w}),&& -G(w)=G(-w).
\label{Gsymmetries}
\end{align}

Let $k,\ell$ be non-zero integers. For $k,\ell\geq  0$ define
\begin{align}
E_{k,\ell}=\frac{i^{-k-\ell}}{2(1+a^2)2\pi i}\int_{\Gamma_1} \frac{dw}{w}\frac{G(w)^\ell G(1/w)^k}{\sqrt{w^2+2c}\sqrt{1/w^2+2c}},
\label{EKL1}
\end{align}
where $\Gamma_r$ is the unit circle of radius $r$, and then for all $k,\ell$ define
\begin{align}
E_{k,\ell}&=E_{|k|,|\ell|}.\label{ADPAD}
\end{align}
For vertices $(x_1,x_2)\in W_{\eps_1}$, $(y_1,y_2)\in B_{\eps_2}$, we define
\begin{align}
\K_{1,1}^{-1}(x_1,x_2,y_1,y_2)=-i^{1+h(\eps_1,\eps_2)}(a^{\eps_2}E_{k_1,\ell_1}+a^{1-\eps_2}E_{k_2,\ell_2}),
\label{K_11}
\end{align}
where we used (\ref{klh1}) and (\ref{klh2}). Define for $0\leq x_1,x_2\leq 2n$
\begin{align}
H_{x_1,x_2}(w)=\frac{w^{n/2}G(w)^{\frac{n-x_1}{2}}}{G(w^{-1})^{\frac{n-x_2}{2}}}.
\label{Hfunc}
\end{align}
Next we define a function $V_{\eps_1,\eps_2}(w_1,w_2)$ which is analytic in the set $\big(\C\setminus (i(-\infty,-1/\sqrt{2c}]\cup i[-\sqrt{2c},\sqrt{2c}]\cup i[1/\sqrt{2c},\infty))\big)^2$.
Define 
\begin{align}
V_{\eps_1,\eps_2}(w_1,w_2)=\sum_{\gamma_1,\gamma_2}^1 Q_{\gamma_1,\gamma_2}^{\eps_1,\eps_2}(w_1,w_2)\label{Vours}
\end{align}
where we define $Q$ as follows,
\begin{align}\label{Qours}
Q_{\gamma_1,\gamma_2}^{\eps_1,\eps_2}(w_1,w_2)=& \ (-1)^{\eps_1+\eps_2+\eps_1\eps_2}\frac{G(w_1)^{3\eps_1-1}G(w_2^{-1})^{3\eps_2-1}}{4(1+a^2)^2\prod_{j=1}^2\sqrt{w_j^2+2c}\sqrt{w_j^{-2}+2c}}\\
& \ (-1)^{\gamma_1(1+\eps_2)+\gamma_2(1+\eps_1)}s(w_1)^{\gamma_1}s(w_2^{-1})^{\gamma_2}\tilde{y}_{\gamma_1,\gamma_2}^{\eps_1,\eps_2}(G(w_1),G(w_2^{-1})),\nonumber
\end{align}
where \begin{align}
s(w)=w\sqrt{w^{-2}+2c},
\end{align}
and the terms 
\begin{align}
\tilde{y}_{\gamma_1,\gamma_2}^{\eps_1,\eps_2}(u,v):=\tilde{y}_{\gamma_1,\gamma_2}^{\eps_1,\eps_2}(a,1,u,v)\label{ytildes}
\end{align} are specified by the following \emph{Laurent polynomials}, for $a,b>0$,
\begin{align}
&\tilde{y}_{0,0}^{0,0}(a,b,u,v)=\frac{a}{4(a^2+b^2)^2}(2a^6u^2v^2-a^4b^2(1+u^4+u^2v^2-u^4v^2+v^4-u^2v^4)\nonumber\\
&\quad\qquad-a^2b^4(1+3u^2+3v^2+2u^2v^2+u^4v^2+u^2v^4-u^4v^4)-b^6(1+v^2+u^2+3u^2v^2)),\nonumber\\
&\tilde{y}_{0,1}^{0,0}(a,b,u,v)=\frac{a}{4(a^2+b^2)}(b^2+a^2u^2)(2a^2v^2+b^2(1+v^2-u^2+u^2v^2)),\nonumber\\
&\tilde{y}_{1,0}^{0,0}(a,b,u,v)=\frac{a}{4(a^2+b^2)}(b^2+a^2v^2)(2a^2u^2+b^2(1-v^2+u^2+u^2v^2)),\nonumber\\
&\tilde{y}_{1,1}^{0,0}(a,b,u,v)=\frac{a}{4}(2a^2u^2v^2+b^2(-1+v^2+u^2+u^2v^2)),\nonumber
\end{align}
and
\begin{align}
\tilde{y}_{i,j}^{0,1}(a,b,u,v)&:=\tilde{y}_{i,j}^{0,0}(b,a,u,v^{-1}),\nonumber\\
\tilde{y}_{i,j}^{1,0}(a,b,u,v)&:=\tilde{y}_{i,j}^{0,0}(b,a,u^{-1},v),\nonumber\\
\tilde{y}_{i,j}^{1,1}(a,b,u,v)&:=\tilde{y}_{i,j}^{0,0}(a,b,u^{-1},v^{-1})\nonumber.
\end{align}
The formula we have introduced for $V_{\eps_1,\eps_2}(w_1,w_2)$ is equal to the one given in \cite{C/J}, this fact is proved in Section \ref{Misclemmas}, in particular see Lemma \ref{tempVequivlemma}.
We also have equation $3.11$ in \cite{C/J} given by
\begin{align}
B_{\eps_1,\eps_2}(a,x_1,x_2,y_1,y_2)=\frac{i^{(x_2-x_1+y_1-y_2)/2}}{(2\pi i)^2}\int_{\Gamma_r}\frac{dw_1'}{w_1'}\int_{\Gamma_{1/r}}dw_2'\frac{V_{\eps_1,\eps_2}(w_1',w_2')}{w_2'-w_1'}\frac{H_{x_1+1,x_2}(w_1')}{H_{y_1,y_2+1}(w_2')}
\label{Beps1eps2}
\end{align}
where $\sqrt{2c}<r<1$, $(x_1,x_2)\in W_{\eps_1}, (y_1,y_2)\in B_{\eps_2}$. 
 A formula for the elements of the inverse Kasteleyn matrix was given in \cite{C/J} which we bring forward in the following theorem.
\begin{theorem}\label{ThmInvKast}
For $n=4m$, $m\in\N_{>0}$, $(x_1,x_2)\in W_{\eps_1}$, $(y_1,y_2)\in B_{\eps_2}$ with $\eps_1,\eps_2\in \{0,1\}$ then 
\begin{align}
K_{a,1}^{-1}((x_1,x_2),(y_1,y_2))&=\K_{1,1}^{-1}((x_1,x_2),(y_1,y_2))- B_{\eps_1,\eps_2}(a,(x_1,x_2),(y_1,y_2))\label{inverseKastelements2}\\
&\qquad\qquad+B^*_{\eps_1,\eps_2}(a,(x_1,x_2),(y_1,y_2)).\nonumber
\end{align}
\end{theorem}
The expression for $B^*_{\eps_1,\eps_2}$ consists of three integrals related to $B_{\eps_1,\eps_2}((x_1,x_2),(y_1,y_2))$ by symmetry, these are given in section \eqref{Misclemmas}. In Lemma \ref{B*smallmaindiag}, we show that $B^*$ is negligible in the setting of condition \ref{cond-a-dimers}.

In order to analyse the asymptotics of the correlation kernel \eqref{corrkernel} we wil take two $a$-dimers $e_i=(w_i,b_i)$, $e_j=(w_j,b_j)$ as in condition \ref{cond-a-dimers}, we will then perform steepest descent analysis on the correlation kernel afforded by Theorem \ref{ThmInvKast}.
Define
\begin{align}
\mathcal{G}=\sqrt{\frac{a}{2}}\exp(a/2)\label{Gcurly}
\end{align}
and 
\begin{align}
g_{\eps_1,\eps_2}=-2iV_{\eps_1,\eps_2}(i,i).
\end{align}
\begin{theorem}\label{extendedAiryKernelLimit}
Let $e_i=(\tilde{x}_i,\tilde{y}_i)$, $e_j=(\tilde{x}_j,\tilde{y}_j)$ be two $a$-dimers as in condition \ref{cond-a-dimers} and let $(x_1,x_2)=\tilde{x}_j\in W_{\eps_1}$, $(y_1,y_2)=\tilde{y}_i\in B_{\eps_2}$, then
\begin{align}
i^{x_1-y_1-1}&\mathcal{G}^{\frac{2+x_1-x_2-y_1+y_2}{2}}g_{\eps_1,\eps_2}^{-1}(an)^{1/3}K_{a,1}^{-1}(\tilde{x}_j,\tilde{y}_i)\\&\rarrow -e^{\beta_j\alpha_j-\beta_i\alpha_i+\frac{2}{3}(\beta_j^3-\beta_i^3)} A(-\beta_j,\alpha_j+\beta_j^2;-\beta_i,\alpha_i+\beta_i^2)\nonumber
\end{align}
as $n$ tends to infinity, uniformly in $\alpha_i,\alpha_j,\beta_i,\beta_j$.
\end{theorem}
This is proved at the end of section \ref{asympk11inv}
\begin{proof}[proof of Proposition \ref{agapsairy}]
First we establish that
\begin{align}
\det\Big(\frac{p_n}{2}L(\tilde{\Phi}(t_i,x_i),\tilde{\Phi}(t_j,x_j))\Big)_{i,j=1}^k\rarrow \det\Big(A(-t_i,x_i+t_i^2,-t_j,x_j+t_j^2)\Big)_{i,j=1}^k
\end{align}
in the limit $n\rarrow \infty$, pointwise $a.e.$ for $t_i,t_j\in \R$, $x_i,x_j>-\beta$. We have $\tilde{\Phi}(t_j,x_j)=e_j=(\tilde{x}_j,\tilde{y}_j)$ where 
\begin{align}
\tilde{x}_j=(n(1+\xi_c)+\alpha_jp_n)\vec{e}_1+t_jq_n\vec{e}_2+(1,0)\\
\tilde{y}_j=(n(1+\xi_c)+\alpha_jp_n)\vec{e}_1+t_jq_n\vec{e}_2+(0,1)
\end{align}
are from condition  \ref{cond-a-dimers} with $|t_j-\beta_j|\leq2/q_n$ and $|\alpha_j-x_j|\leq 2/p_n$.
The correlation kernel is given by
\begin{align}
\frac{p_n}{2}L(e_i,e_j)=\frac{p_n}{2}aiK^{-1}_{a,1}(\tilde{x}_j,\tilde{y}_i).
\end{align}
By substitution,
\begin{align}
i^{x_1-y_1-1}&=i^{(\alpha_j-\alpha_i)p_n+(\beta_i-\beta_j)q_n}\nonumber\\
\mcG^{\frac{2+x_1-x_2-y_1+y_2}{2}}&=\mcG^{2+\beta_iq_n-\beta_jq_n}\nonumber\\
g^{-1}_{0,0}&=-i(1+O(a))\nonumber
\end{align}
where in the last equality we used lemma \ref{Veps1eps2formula}. Hence,
\begin{align}
\det\Big(\frac{p_n}{2}L(e_i,e_j)\Big)_{i,j}&=e^{-a/k}\det\Big(\mcG^2(-i)p_nK_{a,1}^{-1}(\tilde{x}_j,\tilde{y}_i)\Big)_{i,j}\label{temprfvh}\\
&=e^{-a/k}\det\Big(-i^{x_1-y_1-1}\mcG^{\frac{2+x_1-x_2-y_1+y_2}{2}}g_{0,0}^{-1}(an)^{1/3}K_{a,1}^{-1}(\tilde{x}_j,\tilde{y}_i)\Big)_{i,j}\nonumber\\
&\rarrow \det\Big(A(-t_i,x_i+t_i^2,-t_j,x_j+t_j^2)\Big)_{i,j}\nonumber
\end{align}
where we note the conjugate factors and in the last line we used $\alpha_j\rarrow x_j, \beta_j\rarrow t_j$. The last line in \eqref{temprfvh}, we see the time reversal $t_j\mapsto -t_j$ of the Airy process. This does matter, since the Airy process itself is reversible. All that remains is to justify switching the limits in \eqref{fredbigell2}, which holds by Hadamard's inequality and estimates on the kernel of the type used to prove \eqref{temp607oy}, see Lemma 12.30 in \cite{BDS}.
\end{proof}
\section{Asymptotics of $B_{\eps_1,\eps_2}$}\label{asympBeps1eps2}
In this section, we investigate the asymptotics of $B_{\eps_1,\eps_2}$.
 First we define the saddle-point function of $B_{\eps_1,\eps_2}$ relevant for our regime of the weight space.
Observe the following basic facts: $\eta\rarrow 1-1/\eta^2$ is analytic and non-zero on $\C\setminus [-1,1]$, and maps $\C\setminus[-1,1]$ to $\C\setminus (-\infty,0]$. Hence
 \begin{align}
 \eta\rarrow \sqrt{1-1/\eta^2}
 \end{align}
 is analytic on $\C\setminus [-1,1]$ (with principal branch square root). Now note that for $x>1$, $\sqrt{x-1}\sqrt{x+1}=x\sqrt{1-x^2}$, so by uniqueness of analytic functions
 \begin{align}
 \sqrt{\eta-1}\sqrt{\eta+1}=\eta\sqrt{1-1/\eta^2}.
 \end{align}
 We set $\eta=-iw/\sqrt{2c}$ in the above identity to get 
 \begin{align}
 \sqrt{w^2+2c}=w\sqrt{1+\frac{2c}{w^2}}\label{squarerootrewrite}
 \end{align}
 where the square root on the right-hand side is the principal branch square root. We can now rewrite $G$ as
 \begin{align}
 G(w)=\frac{w}{\sqrt{2c}}(1-\sqrt{1+\frac{2c}{w^2}})
 \label{temp24dzb}
 \end{align}
 
  Observe that $w\rarrow G(1/w)$ is analytic on $\C\setminus(\{0\}\cup i[1/\sqrt{2c},\infty)\cup i(-\infty,-1/\sqrt{2c}])$ and has an analytic extension to $\C\setminus i[1/\sqrt{2c},\infty)\cup i(-\infty,-1/\sqrt{2c}]$ by defining its value to be zero at $w=0$. Fix $R>1$, let $s\in (|w|,1/\sqrt{2c})$, and define a function $R(w,a)$ for $|w|\leq R$ and $a>0$ small enough that $1/(2\sqrt{2c})>R$, by
  \begin{align}
  (\sqrt{c}w)^5R(w,a):&=\frac{(\sqrt{2c}w)^5}{2\pi i}\int_{\Gamma_{\sqrt{2c}s}}\frac{\frac{1}{\zeta}(1-\sqrt{1+\zeta^2})}{\zeta^5(\zeta-\sqrt{2c}w)}d\zeta\quad
  \Big(= \frac{w^5}{2\pi i}\int_{\Gamma_s}\frac{G(1/\zeta)}{\zeta^5(\zeta-w)}d\zeta\Big).
  \end{align}
  From the definition of $R(w,a)$ we see that picking $s=3/(4\sqrt{2c})$, the function $w\rarrow R(w,a)$ is bounded by some constant $C>0$ independent of $a$ for $|w|\leq R$. Observe also that $w\rarrow R(w,a)$ is analytic for $|w|\leq 1/(2\sqrt{2c})$.
Hence, we have
 \begin{align}\label{Gexpans}
 G(1/w)=-\sqrt{\frac{c}{2}}w+\frac{c^{3/2}w^3}{2\sqrt{2}}+(\sqrt{c}w)^5R(w,a).
 \end{align}

An application of Taylor's theorem yields a bounded function $\tilde{R}(w,a)$
\begin{align}
\log(G(1/w))=\log(-\sqrt{\frac{a}{2}}w)-\frac{aw^2}{2}+a^2\tilde{R}(w,a)\label{logGtaylor}
\end{align}
where $\tilde{R}$ is bounded for all $|w|\leq R$, $a>0$ sufficiently small (depending on $R$).
\begin{assumption}\label{assumptionairynotationcoords}
For the remainder of this section, we let $e_i=(\tilde{x}_i,\tilde{y}_i)$, $e_j=(\tilde{x}_j,\tilde{y}_j)$ be two $a$-dimers as in condition \ref{cond-a-dimers} and let $(x_1,x_2)=\tilde{x}_j\in W_{\eps_1}$, $(y_1,y_2)=\tilde{y}_i\in B_{\eps_2}$.
\end{assumption}

We have

\begin{align}
H_{x_1+1,x_2}(w_1)&=\frac{w_1^{n/2}G(w_1)^{(n-x_1-1)/2}}{G(1/w_1)^{(n-x_2)/2}}\\
\nonumber&=\frac{w_1^{n/2}G(w_1)^{(-n\xi_c+\beta_jq_n-\alpha_jp_n-2)/2}}{G(1/w_1)^{(-n\xi_c-\beta_jq_n-\alpha_jp_n-2\eps_1)/2}}\\&=\exp\frac{n}{2}\big(\log(w_1)-\xi_1^{\tilde{x}_j}\log G(w_1)+\xi_2^{\tilde{x}_j}\log G(1/w_1)\big)\nonumber
\end{align}
and
\begin{align}
H_{y_1,y_2+1}(w_2)=\exp\frac{n}{2}\big(\log(w_2)-\xi_1^{\tilde{y}_i}\log G(w_2)+\xi_2^{\tilde{y}_i}\log G(1/w_2)\big)
\end{align}
where
\begin{align}
\xi_1^{\tilde{x}_j}=\xi_c-\beta_jq_n/n+\alpha_jp_n/n+2/n,\quad \xi_2^{\tilde{x}_j}=\xi_c+\beta_jq_n/n+\alpha_jp_n/n+2\eps_1/n,\\
\xi_1^{\tilde{y}_i}=\xi_c-\beta_iq_n/n+\alpha_ip_n/n+2\eps_2/n,\quad
\xi_2^{\tilde{y}_i}=\xi_c+\beta_iq_n/n+\alpha_ip_n/n+2/n.
\end{align} 

We note that $\beta_jq_n/n\rarrow 0$, that
\begin{align}
\frac{q_n}{n}>>\frac{a}{2}>>\frac{p_n}{n}>>\frac{1}{n}>>a^2
\end{align}
and that $a>>\alpha_jp_n/n$ uniformly in $\alpha_j$. 

Hence by \eqref{logGtaylor}, we have a bounded function $R(w_1):=R(w_1,\beta_j,\alpha_j,a,\eps_1)$ (for $1/R\leq |w_1|\leq R$) such that
\begin{align}
H_{x_1+1,x_2}(w_1)&=\exp\Big(\frac{n}{2}\big(\log(w_1)-\xi_1^{\tilde{x}_j}(\log(-\sqrt{\frac{a}{2}}\frac{1}{w_1})-a/(2w_1^2))+\xi_2^{\tilde{x}_j}(\log(-\sqrt{\frac{a}{2}}w_1)-aw_1^2/2)\big)+a^2nR(w_1)\Big)\\
&=w_1^{n/2}\big(-\sqrt{\frac{a}{2}}\frac{1}{w_1}\big)^{-n\xi_1^{\tilde{x}_j}/2}\big(-\sqrt{\frac{a}{2}}w_1\big)^{n\xi_2^{\tilde{x}_j}/2}\exp\Big(\frac{na}{4}\xi_1^{\tilde{x}_j}\frac{1}{w_1^2}-\frac{na}{4}\xi_2^{\tilde{x}_j}w_1^2+na^2R(w_1)\Big)\nonumber\\
&=\Big(-\sqrt{\frac{2}{a}}\Big)^{-\beta_jq_n+1-\eps_1}w_1^{n(\xi_c+1/2)+\alpha_jp_n+\eps_1+1}\exp\Big(\frac{na}{4}\xi_1^{\tilde{x}_j}\frac{1}{w_1^2}-\frac{na}{4}\xi_2^{\tilde{x}_j}w_1^2+na^2R(w_1)\Big).\nonumber
\end{align}
Observe that
\begin{align}
\frac{na}{4}\xi_1^{\tilde{x}_j}=-\frac{an}{8}-\beta_j\frac{aq_n}{4}+O(a^2n),&&\frac{na}{4}\xi_2^{\tilde{x}_j}=-\frac{an}{8}+\beta_j\frac{aq_n}{4}+O(a^2n),\nonumber
\end{align}
where the $O(a^2n)$ is uniform in $\alpha_j$, and
\begin{align}
n(\xi_c+1/2)=\frac{an}{2}+O(a^2n).
\end{align}
Finally, we arrive at
\begin{align}
H_{x_1+1,x_2}(w_1)=&\ \Big(-\sqrt{\frac{2}{a}}\Big)^{-\beta_jq_n+1-\eps_1}w_1^{\eps_1+1}\\&\times\exp\Big(\frac{an}{2}f(w_1)-\beta_j\frac{aq_n}{4}(w_1^2+\frac{1}{w_1^2})+\alpha_jp_n\log(w_1)+a^2nR_1(w_1)\Big)\nonumber
\end{align}
with a saddle-point function
\begin{align}
f(w)=\log(w)+\frac{1}{4}(w^2-\frac{1}{w^2}).
\end{align}
We call $f$ a saddle-point function because
\begin{align}
aq_n=(an)^{2/3}, && p_n=(an)^{1/3}
\end{align}
are lower order than $an$ as $n\rarrow \infty$.
Similarly,
\begin{align}
H_{y_1,y_2+1}(w_2)=&\ \Big(-\sqrt{\frac{2}{a}}\Big)^{-\beta_iq_n+\eps_2-1}w_2^{\eps_2+1}\\&\times\exp\Big(\frac{an}{2}f(w_2)-\beta_i\frac{aq_n}{4}(w_2^2+\frac{1}{w_2^2})+\alpha_ip_n\log(w_2)+a^2nR_2(w_2)\Big)\nonumber.
\end{align}
Clearly, the remainder functions $R_1,R_2$ are bounded for $1/R<|w_1|,|w_2|<R$, $-\beta<\beta_i,\beta_j<\beta$, $-\beta<\alpha_i,\alpha_j<\beta\log(n)$ for fixed $\beta,R>>1$ and all $a>0$ small (depending on $\beta,R$).
We compute
\begin{align}
f'(w)=\frac{(1+w^2)^2}{2w^3},&&f''(w)=\frac{(w^2-3)(w^2+1)}{2w^4}, && f'''(w)=\frac{6+2w^2}{w^5}.
\end{align}
We see $f'(i)=f'(-i)=f''(i)=f''(-i)=0$, so $f$ has double critical points at $i$ and $-i$. The paths of steepest ascent and descent satisfy
\begin{align}
\mcI[f(w)]=\mcI[f(i)]=\pi/2.\label{ascentdescentcondition}
\end{align} Observe that
\begin{align}
\mcR[f(w)]=\mcR[f(-w)]=\mcR[f(\overline{w})],
\end{align}
so we only need to consider the upper-right quadrant $\HH^+$ in finding the paths of steepest ascent and descent. Let $w=re^{i\theta}$ for $\theta\in(0,\pi/2)$, $r>0$. The condition \eqref{ascentdescentcondition} is equivalent to
\begin{align}
\theta-\pi/2+\frac{1}{4}(r^2+1/r^2)\sin(2\theta)=0,
\end{align}
this has two solutions for $r^2$,
\begin{align}
r^2_{\pm}(\theta)=\frac{2(\pi/2-\theta)}{\sin(2\theta)}\pm \sqrt{\Big(\frac{2(\pi/2-\theta)}{\sin(2\theta)}\Big)^2-1}.
\end{align}
Observe the following basic facts; $\theta\mapsto 2(\pi/2-\theta)/\sin(2\theta)$ is a decreasing function on $\theta\in (0,\pi/2)$ and
\begin{align}
\lim_{\theta\rarrow0^+}2\frac{\pi/2-\theta}{\sin(2\theta)}=\infty&&\lim_{\theta\rarrow\pi/2^-}2\frac{\pi/2-\theta}{\sin(2\theta)}=1.
\end{align}
Also, for $z\in (1,\infty)$, $z+\sqrt{z^2-1}$ is increasing from $1$ to infinity, and $z-\sqrt{z^2-1}$ is decreasing from $1$ to $0$.
Hence we see that  as $\theta$ goes from $0$ to $\pi/2$, $r_+(\theta)$ decreases from $\infty$ to $1$, and $r_-(\theta)$ increases from $0$ to $1$. 

Label the contour which is parametrised by $r_+(\theta)e^{i\theta}$, and its reflections in the real and imaginary axes, as $asc$. Label the contour which is parametrised by $r_-(\theta)e^{i\theta}$, and its reflections in the real and imaginary axes, by $desc$. Since $asc$ goes to infinity in $\HH^+$, it intersects the ball of radius $R>>1$, $B_R(0)$ at a point $z_R^*$. Consider the contour which is the union of $asc\cap\HH^+$ and the section of $\partial B_R(0)$ starting at $z_R^*$ and travelling down to the real line. Label this contour (and its reflections in real and imaginary axes) $asc_R$. Similarly, $desc$ travels down the origin in $\HH^+$ and intersects and ball of radius $1/R<<1$ at $z_{1/R}^{**}$. Label the contour $desc_{1/R}$ which is the union of $desc\cap\HH^+$ and the section of $\partial B_{1/R}(0)$ starting at $z_{1/R}^{**}$ and travelling down to the real line (and its reflections).

Because 
\begin{align}
f(w)-f(i)=-\frac{2i}{3}(w-i)^3+O(w-i)^4
\end{align}
we see that $desc$ corresponds to the path of steepest descent, and $asc$ is the path of steepest ascent.

For convenience, introduce
\begin{align}
\tilde{g}_{\alpha_j,\beta_j}(w)=f(w)-\beta_j\frac{q_n}{2n}(w^2+\frac{1}{w^2})+\alpha_j\frac{2p_n}{an}\log(w)
\end{align}
so that
\begin{align}
H_{x_1+1,x_2}(w_1)=&\ \Big(-\sqrt{\frac{2}{a}}\Big)^{-\beta_jq_n+1-\eps_1}w_1^{\eps_1+1}\exp\Big(\frac{an}{2}\tilde{g}_{\alpha_j,\beta_j}(w_1)+a^2nR_1(w_1)\Big),\label{Hsaddlessimpler}\\
H_{y_1,y_2+1}(w_2)=&\ \Big(-\sqrt{\frac{2}{a}}\Big)^{-\beta_iq_n+\eps_2-1}w_2^{\eps_2+1}\exp\Big(\frac{an}{2}\tilde{g}_{\alpha_i,\beta_i}(w_2)+a^2nR_2(w_2)\Big).
\end{align}

\begin{lemma}\label{saddlepointexpprop1}
Fix an $\eps>0$ small enough and consider $w_1',w_2'\in \overline{B(i,\eps)}$. Set $w_1'=i+w_1(na)^{-1/3}$, $w_2'=i+w_2(na)^{-1/3}$, then
\begin{align}
&\frac{H_{x_1+1,x_2}(w_1')}{H_{x_1+1,x_2}(i)}=(-iw_1')^{\eps_1+1}\exp\big(-iw_1\alpha_j-w_1^2\beta_j-\frac{i}{3}w_1^3+\frac{(|w_1|+|w_1|^4) }{(an)^{1/3}}R_3(w_1,a)+a^2nR_4(w_1)\big)\label{tewmprfvn45}
\end{align}
and
\begin{align}
&\frac{H_{y_1,y_2+1}(i)}{H_{y_1,y_2+1}(w_2')}=(-iw_2')^{\eps_2+1}\exp\big(iw_2\alpha_i+w_2^2\beta_i+\frac{i}{3}w_2^3+\frac{(|w_2|+|w_2|^4) }{(an)^{1/3}}R_5(w_2,a)+a^2nR_6(w_1)\big)
\end{align}
where there is a $C>0$ such that
\begin{align}
|R_k(w,a)|\leq C\label{Rem146}
\end{align} 
 uniformly in $\alpha_j,\beta_j,\alpha_i,\beta_i$ and $a$ small enough.   
 \begin{proof}
We focus on \eqref{tewmprfvn45}. There are complex numbers $\chi_1,\chi_2\in L(i,w_1')$ such that
\begin{align}
&\tilde{g}_{\alpha_j,\beta_j}(w_1')-\tilde{g}_{\alpha_j,\beta_j}(i)\label{gexp123}\\&=\tilde{g}_{\alpha_j,\beta_j}'(i)(w_1'-i)+\tilde{g}_{\alpha_j,\beta_j}''(i)(w_1'-i)^2/2+\tilde{g}_{\alpha_j,\beta_j}'''(i)(w_1'-i)^3/3!\nonumber\\
&\quad+\mcR[(w_1'-i)^4\tilde{g}_{\alpha_j,\beta_j}^{iv}(\chi_1)/4!]+\mcI[(w_1'-i)^4\tilde{g}_{\alpha_j,\beta_j}^{iv}(\chi_2)/4!]\nonumber
\end{align}
where $L(i,w)$ is the straight line connecting $i$ and $w$. We have
\begin{align}
\label{temp3gatis}\tilde{g}_{\alpha_j,\beta_j}'(i)&=\alpha_j\frac{2p_n}{an}(-i)\\\tilde{g}_{\alpha_j,\beta_j}''(i)&=-\beta_j\frac{q_n}{2n}8+\alpha_j\frac{2p_n}{an}\\\tilde{g}_{\alpha_j,\beta_j}'''(i)&=-4i-\beta_j\frac{q_n}{2n}24i+\alpha_j\frac{2p_n}{an}2i.
\end{align}
We can find a constant $C'>0$ depending only on $\eps$ such that
\begin{align}
\sup_{w\in B(i,\eps)}|\tilde{g}_{\alpha_j,\beta_j}^{iv}(w)|\leq C'.
\end{align}
We use this bound and substitute \eqref{temp3gatis} into \eqref{gexp123}. Collecting the remainders, we see they are all $O((an)^{-4/3})$, uniformly for $\beta_j,\alpha_j$. Now substitute $w_1'=i+w_1(an)^{-1/3}$, and simplify the sum of the remainders with
\begin{align}
c_1|w_1|+c_2|w_1|^2+c_3|w_1|^3+c_4|w_1|^4\leq 2 \max_i\{c_i\}(|w_1|+|w_1|^4)
\end{align}
where $c_i>0$.
 \end{proof}
\end{lemma}

We have a formula that simplifies the leading order factor of the asymptotics of the kernel.
\begin{lemma}\label{Veps1eps2formula}We have
\begin{align}
-2iV_{\eps_1,\eps_2}(i,i)=\mathrm{g}_{\eps_1,\eps_2}\label{lemma3.16CJ}
\end{align}
where 
\begin{align}
\mathrm{g}_{\eps_1,\eps_2}=i^{1+\eps_1(1-\eps_2)-\eps_2(1-\eps_1)}(a/2)^{(\eps_1+\eps_2)/2}(1+O(a)). \label{temp23fd}
\end{align}
We also have
\begin{align}
V_{0,0}(w_1,w_2)=-\frac{w_1}{2w_2}+a\frac{(w_1^2+w_2^2)(1+w_1^2w_2^2)}{4w_1w_2^3}+a^2R(w_1,w_2,a).\label{V00expans}
\end{align}
For a fixed $R>1$ the function defined by
\begin{align}
\tilde{V}_{\eps_1,\eps_2}(w_1,w_2)=c^{-(\eps_1+\eps_2)/2}V_{\eps_1,\eps_2}(w_1,w_2)
\label{Vuniformbdd}
\end{align}
and the function $R(w_1,w_2,a)$ above are uniformly bounded for all $a>0$ small enough, $w_1,w_2\in B_{R}(0)\setminus B_{1/R}(0)$. Furthermore, the functions
\begin{align}
\frac{d}{dw_1}\tilde{V}_{\eps_1,\eps_2}(w_1,w_2), && \frac{d}{dw_1}R(w_1,w_2,a)\label{derivsVtildeR}
\end{align}
are also uniformly bounded for all $a>0$ small enough, $w_1,w_2\in B_{R}(0)\setminus B_{1/R}(0)$.
\begin{proof}
The formula in \eqref{lemma3.16CJ} follows by a simple Taylor expansion of the formula for $V_{\eps_1,\eps_2}(i,i)$ given in Lemma 3.16 in \cite{C/J}. We focus on the expansion for $V_{0,0}(w_1,w_2)$.
The remainder functions $R_{\alpha}(w_1,w_2,a)$ in this proof are bounded uniformly for all $w_1,w_2\in B_R(0)\setminus B_{1/R}(0)$, all $0<a<1$ sufficiently small.
Observe that the \emph{polynomials} $\tilde{y}_{i,j}^{0,0}(u,v)$ in \eqref{ytildes} are written in the form 
\begin{align}
\frac{a}{4(1+a^2)}\sum_{k=0}^3a^{2k}p_{k}(u,v)
\end{align}
for some polynomials $p_k(u,v):=p_k^{i,j}(u,v)$. In each case $i,j$ we insert $u=G(w_1), v=G(w_2)$ with 
\begin{align}
G(w^{\pm1})=-\sqrt{\frac{a}{2}}\frac{1}{w^{\pm1}}+a^{3/2}R_1(w^{\pm1},a),&&\frac{1}{G(w^{\pm})}=-\sqrt{\frac{2}{a}}w^{\pm1}-\sqrt{\frac{a}{2}}\frac{1}{w^{\pm1}}+a^{3/2}R_2(w^{\pm 1},a)\label{tempGexps234}
\end{align}
into $p^{i,j}_0$, where the expansions come from \eqref{Gexpans}. 
We get
\begin{align}
p_0^{0,0}(u,v)&=-(1+v^2+u^2+3u^2v^2)\nonumber\\
&=-1-a\Big(\frac{1}{2w_1^2}+\frac{1}{2w_2^2}\Big)+a^2R_3(w_1,w_2,a),\nonumber\\
p_0^{1,0}(u,v)&=p_0^{0,1}(v,u)=(1+a^2)(1+v^2-u^2+u^2v^2)\nonumber\\
&=1+a\Big(\frac{-1}{2w_1^2}+\frac{1}{2w_2^2}\Big)+a^2R_4(w_1,w_2,a),\nonumber\\
p_0^{1,1}(u,v)&=(1+a^2)^2(-1+v^2+u^2+u^2v^2)\nonumber\\
&=-1+a\Big(\frac{1}{2w_1^2}+\frac{1}{2w_2^2}\Big)+a^2R_5(w_1,w_2,a).\nonumber
\end{align}
Hence we have
\begin{align}
\tilde{y}_{i,j}^{0,0}(G(w_1),G(w_2))=\frac{a}{4}(-1)^{i+j+1}+\frac{a^2}{4}\Big(\frac{(-1)^{i+1}}{2w_1^2}+\frac{(-1)^{j+1}}{2w_2^2}\Big)+a^3R_{i,j}(a,w_1,w_2).
\end{align}
We also have
\begin{align}
\sqrt{w^2+2c}=\frac{1}{w}+aw+a^2R_7(w,a),&&\sqrt{w^2+2c}=w+\frac{a}{w}+a^2R_7(1/w,a)
\end{align}
from which we deduce
\begin{align}
\frac{1}{4(1+a^2)^2\prod_{j=1}^2\sqrt{w_j^2+2c}\sqrt{w_j^{-2}+2c}}=\frac{1}{4}-\frac{a}{4}\Big(w_1^2+\frac{1}{w_1^2}+w_2^2+\frac{1}{w_2^2}\Big)+a^2R_8(w_1,w_2,a)\label{tempsexpans}
\end{align}
and
\begin{align}
s(w)&=1+aw^2+R_9(w,a)a^2.
\end{align}
We want to take the sum over $\gamma_1,\gamma_2$ in $Q_{\gamma_1,\gamma_2}^{0,0}$ in \eqref{Qours}, and  so we compute 
\begin{align}
&\sum_{\gamma_1,\gamma_2}(-1)^{\gamma_1+\gamma_2}s(w_1)^{\gamma_1}s(w_2^{-1})^{\gamma_2}\tilde{y}_{\gamma_1,\gamma_2}^{0,0}(G(w_1),G(w_2^{-1})\label{temp43refdmj}\\
&=\sum_{\gamma_1,\gamma_2}\frac{a}{4}(-1)^{\gamma_1+\gamma_2}(1+aw_1^2)^{\gamma_1}(1+a/w_2^2)^{\gamma_2}\Big((-1)^{\gamma_1+\gamma_2+1}+a\Big(\frac{(-1)^{\gamma_1+1}}{2w_1^2}+\frac{(-1)^{\gamma_2+1}w_2^2}{2}\Big)\Big)\nonumber\\&\quad+a^3R_{10}(w_1,w_2,a)\nonumber\\
&=-a-\frac{1}{2}(w_1^2+w_2^2)a^2+a^3R_{11}(w_1,w_2,a).\nonumber
\end{align}
We use the expansions \eqref{tempGexps234}, \eqref{tempsexpans}, \eqref{temp43refdmj} and expand brackets to obtain
\begin{align}
V_{0,0}(w_1,w_2)&=\sum_{\gamma_1,\gamma_2}Q_{\gamma_1,\gamma_2}^{0,0}(w_1,w_2)\nonumber\\
&=-\frac{w_1}{2w_2}+a\frac{(w_1^2+w_2^2)(1+w_1^2w_2^2)}{4w_1w_2^3}+a^2R(w_1,w_2,a).\nonumber
\end{align}

Now we check that $\tilde{V}$ is bounded. By using the expansions of $G(w), 1/G(w)$ above, the following bounds are straightforward to see by expanding brackets and using the triangle inequality,
\begin{align}
|\tilde{y}_{i,j}^{0,1}(u,v)|&=|\tilde{y}_{i,j}^{0,1}(v,u)|\leq C_1,\label{tempytildebounds}\\
|\tilde{y}_{i,j}^{1,1}(u,v)|&\leq \frac{C_2}{a},\nonumber
\end{align}
which hold for constants $C_1,C_2>0$, uniformly for all $w_1,w_2\in B_R(0)\setminus B_{1/R}(0)$, all $a>0$ small. Then, using the further expansions above it is easy to show $\tilde{V}$ is bounded. We omit the details.

Finally we want to prove the functions \eqref{derivsVtildeR} are bounded. Recalling \eqref{temp24dzb}, we begin by observing that the polynomials $\tilde{y}_{\gamma_1,\gamma_2}^{\eps_1,\eps_2}(u,v)$ in \eqref{ytildes} are Laurent polynomials in the variables $u^{\pm 2}, v^{\pm 2}$. Hence the term $\tilde{y}_{\gamma_1,\gamma_2}^{\eps_1,\eps_2}(G(w_1),G(w_2^{-1}))$ appearing in $Q_{\gamma_1,\gamma_2}^{\eps_1,\eps_2}(w_1,w_2)$ only depends on the functions
\begin{align}
G(w)^2=\frac{w^2}{2c}(1-\sqrt{1+\frac{2c}{w^2}})^2,&&\frac{1}{G(w)^2}=\frac{w^2}{2c}(1+\sqrt{1+\frac{2c}{w^2}})^2,
\end{align}
which are both analytic for 
\begin{align} 
(c,w)\in (B_\eps(0)\setminus \{0\})\times (B_R(0)\setminus B_{1/R}(0))\subset \C^2
\end{align}
(where we first fixed $R>1$ and then took $\eps>0$ sufficiently small). Hence $\tilde{y}_{\gamma_1,\gamma_2}^{\eps_1,\eps_2}(G(w_1),G(w_2^{-1}))$ can be considered as analytic functions of $(c,w_1,w_2)$ on 
\begin{align}
 (B_\eps(0)\setminus \{0\})\times (B_R(0)\setminus B_{1/R}(0))^2\subset \C^3. \label{analyticregion1}
\end{align}
From here we can see that $c^{(-\eps_1-\eps_2)/2}Q_{\eps_1,\eps_2}^{\gamma_1,\gamma_2}(w_1,w_2)$, \eqref{Qours}, is analytic on the set \eqref{analyticregion1}. In particular, we check the term in its numerator
\begin{align}
&c^{(-\eps_1-\eps_2)/2}G(w_1)^{3\eps_1-1}G(w_2^{-1})^{3\eps_2-1}\nonumber\\\nonumber&=\frac{2^{(\eps_1+\eps_2)/2}}{(2c)^{4\eps_1+4\eps_2-2)/2}}(w_1(1-\sqrt{1+2c/w_1^2}))^{3\eps_1-1}(\frac{1}{w_2}(1-\sqrt{1+2cw_2^2}))^{3\eps_2-1}
\end{align}
is indeed analytic on \eqref{analyticregion1}, the rest of the functions that make up $c^{(-\eps_1-\eps_2)/2}Q_{\eps_1,\eps_2}^{\gamma_1,\gamma_2}(w_1,w_2)$ are clearly analytic. Hence $\tilde{V}_{\eps_1,\eps_2}(w_1,w_2)$ is analytic on \eqref{analyticregion1}. Furthermore, the bounds in the arguments that we used to show $\tilde{V}_{\eps_1,\eps_2}(w_1,w_2)$ is uniformly bounded for $c>0$ small extend to uniform bounds on $B_{\eps}(0)\setminus \{0\}$ in a similar way. Hence the singularity at $c=0$ is removable, and so $\tilde{V}_{\eps_1,\eps_2}(w_1,w_2)$ has an analytic extension to 
\begin{align}
 B_\eps(0)\times (B_R(0)\setminus B_{1/R}(0))^2\subset \C^3. \label{analyticregion2}
\end{align}
This proves the functions \eqref{derivsVtildeR} are bounded as stated. 
\end{proof}
\end{lemma}

We have a lemma regarding the leading term in the asymptotics
\begin{lemma}\label{labelH(i)simple}
We have that
\begin{align}
\frac{H_{x_1+1,x_2}(i)}{H_{y_1,y_2+1}(i)}=i^{(y_1+y_2-x_1-x_2)/2}\abs{G(i)}^{(x_2-y_2+y_1-x_1-2)/2}.
\end{align}
\begin{proof}
We use $G(i)=i(1-\sqrt{1-2c})=i\abs{G(i)}$, $G(1/i)=i^{-1}\abs{G(i)}$ in \eqref{Hfunc}.
\end{proof}
\end{lemma}
For the following proposition we define the prefactor
\begin{align}
&P_{\eps_1,\eps_2}(\tilde{x}_j,\tilde{y}_i):=i^{(x_2-x_1+y_1-y_2)/2}\frac{H_{x_1+1,x_2}(i)}{H_{y_1,y_2+1}(i)}\Big(\frac{iV_{\eps_1,\eps_2}(i,i)}{(an)^{1/3}}\Big)\label{prefact}
\end{align}
and the term $\tilde{B}_{\eps_1,\eps_2}(a,\tilde{x}_j,\tilde{y}_i)$ by
\begin{align}\label{Bepstemp51}
B_{\eps_1,\eps_2}(a,\tilde{x}_j,\tilde{y}_i)=2P_{\eps_1,\eps_2}(\tilde{x}_j,\tilde{y}_i)\tilde{B}_{\eps_1,\eps_2}(a,\tilde{x}_j,\tilde{y}_i).
\end{align}
where $B_{\eps_1,\eps_2}(a,\tilde{x}_j,\tilde{y}_i):=B_{\eps_1,\eps_2}(a,x_1,x_2,y_1,y_2)$.
\begin{proposition}\label{Btildeairy}
\begin{align}
\lim_{n\rarrow \infty} \tilde{B}_{\eps_1,\eps_2}(a,\tilde{x}_j,\tilde{y}_i)=-e^{\beta_j\alpha_j-\beta_i\alpha_i+\frac{2}{3}(\beta_j^3-\beta_i^3)} i \tilde{A}(-\beta_j,\alpha_j+\beta_j^2;-\beta_i,\alpha_i+\beta_i^2)\label{blimtairy}
\end{align}
uniformly in  $-\beta\leq \alpha_i,\alpha_j\leq \beta\log(n)$, $|\beta_i|,|\beta_j|\leq \beta$, where \eqref{airypart}.
\begin{proof}
Recall \eqref{Beps1eps2}. Let $R>>1$ and deform the $w_1$ contour $\Gamma_{r}$ to $desc_{1/R}$ and the $w_2$ contour to $asc_R$.
For the moment, let $f$ be an arbitrary continuous function on an open region containing a path $\gamma\subset \C$, then \begin{align}
\overline{\int_\gamma f(z)dz}=\int_{\overline{\gamma}}\overline{f(\overline{z})}dz.\label{tempidentitywef}
\end{align}
Using the symmetries \eqref{Gsymmetries}, we see the integrand  maps to its complex conjugate under $(w_1,w_2)\rarrow (-\overline{w_1},-\overline{w_2})$. To see this, in particular use
\begin{align}
H_{y_1,y_2+1}(-w)=(-1)^{(-y_1+y_2+1)/2}H_{y_1,y_2+1}(w)=(-1)^{(y_1+y_2+1)/2}H_{y_1,y_2+1}(w)=(-1)^{\eps_2+1}H_{y_1,y_2+1}(w)
\end{align} similarly,
\begin{align}
 H_{x_1+1,x_2}(-w)&=(-1)^{\eps_1+1}H_{x_1+1,x_2}(w),
 \end{align}
 and
 \begin{align}
 Q_{\gamma_1,\gamma_2}^{\eps_1,\eps_2}(-w_1,-w_2)&=(-1)^{3\eps_1+3\eps_2+2} Q_{\gamma_1,\gamma_2}^{\eps_1,\eps_2}(w_1,w_2)
\end{align} 
which gives $V_{\eps_1,\eps_2}(-w_1,-w_2)=(-1)^{\eps_1+\eps_2}V_{\eps_1,\eps_2}(w_1,w_2)$. So the sections of the integral over $desc\cap \HH\times asc\cap \HH$ and  $desc\cap \overline{\HH}\times asc\cap \overline{\HH}$ have the same value by the identity \eqref{tempidentitywef}.  
We see the contributions in $\HH\times \overline{\HH}$, $\overline{\HH}\times\HH$ are also equal, these contributions give zero in the limit which will become clear during this proof. Thus, we focus on the section of the integral over  $desc\cap \HH\times asc\cap \HH$. 

We require one more contour deformation.
Recall the constant $C>0$ bounding the remainders in Lemma \ref{saddlepointexpprop1}, take $\eps<\frac{1}{12C}$ so that
\begin{align}
-\frac{r^3}{6}+\frac{Cr^4}{(an)^{1/3}}=-\frac{r^3}{6}(1-\frac{6Cr}{(an)^{1/3}})<-\frac{r^3}{12},&& -\frac{s^3}{6}+\frac{Cs^4}{(an)^{1/3}}<-\frac{s^3}{12}\label{epssmallcond}
\end{align}
where $r,s\in[0,\eps(an)^{1/3}]$. We will use these bounds later in the proof. Let $w_2^*=i+\eps e^{i\theta}\in asc\cap \partial B_\eps(i)$, since $asc$ leaves $i$ at angles $\pi/6$ and $5\pi/6$ we can take $\eps$ so small that $\theta\in (\pi/6-\delta,\pi/6+\delta)\cup(5\pi/6-\delta,5\pi/6+\delta)$ where $\delta=\pi/18$. Similarly, let $w_1^*=i+\eps e^{i\varphi}\in desc\cap \partial B_\eps(i)$. Since $desc$ leaves $i$ at angles $-\pi/6$ and $-5\pi/6$ we can take $\eps$ so small that $\varphi\in(-\pi/6-\delta,-\pi/6+\delta)\cup(-5\pi/6-\delta,-5\pi/6+\delta)$. For later use, we remark that
\begin{align}
\mcR[-\frac{i}{3}w_1^3]=\frac{r^3}{3}\sin(3\theta)<-\frac{r^3}{6},&& \mcR[\frac{i}{3}w_2^3]=-\frac{s^3}{3}\sin(3\varphi)<-\frac{r^3}{6}.\label{deltasmallcond}
\end{align}

 Create a new contour $asc_R^\eps$ consisting of $asc_R\cap B_\eps(i)^c\cap\HH$ and the two straight lines connecting $i$ to the  two points $w_2*\in asc\cap\partial B_\eps(i)$. Construct $desc_{1/R}^\eps$ similarly. Deform the contour $asc_R\cap \HH$ to $asc_R^\eps$, and deform $desc_{1/R}\cap \HH$ to $desc_{1/R}^\eps$.

We have
\begin{align}
B_{\eps_1,\eps_2}(a,\tilde{x}_j,\tilde{y}_i)=2\frac{i^{(x_2-x_1+y_1-y_2)/2}}{(2\pi i)^2}\int_{desc_{1/R}^\eps}\frac{dw_1'}{w_1'}\int_{asc_R^\eps}dw_2'\frac{V_{\eps_1,\eps_2}(w_1',w_2')}{w_2'-w_1'}\frac{H_{x_1+1,x_2}(w_1')}{H_{y_1,y_2+1}(w_2')}+o(1).
\end{align}
Now we develop bounds to deal with with the section of the contours on $\partial B_R(0)$ and $\partial B_{1/R}(0)$. Let $w_2=Re^{i\theta}\in asc_R^\eps$, $\theta\in (0,\arg(z_R^*))$,
\begin{align}
\abs{\frac{H_{y_1,y_2+1}(w_2)}{H_{y_1,y_2+1}(i)}}=R^{\eps_2+1}\exp\Big(\mcR \frac{an}{2}\tilde{g}_{\alpha_i,\beta_i}(w_2)-\mcR\frac{an}{2}\tilde{g}_{\alpha_i,\beta_i}(i)+a^2nR_2'(w_2)\Big).
\end{align}
Observe that
\begin{align}
\mcR \frac{an}{2}\tilde{g}_{\alpha_i,\beta_i}(w_2)-\mcR\frac{an}{2}\tilde{g}_{\alpha_i,\beta_i}(i)&=\frac{an}{2}\Big(\log R+\mcR\frac{1}{4}(R^2e^{2i\theta}-\frac{1}{R^2}e^{-2i\theta})\\
&-\frac{\beta_jq_n}{2n}\mcR(R^2e^{2i\theta}+\frac{1}{R^2}e^{-2i\theta}-2)+\alpha_j\frac{2p_n}{an}\log R\Big)\nonumber\\
&\geq \frac{an}{2}\Big(\log R+\mcR\frac{1}{4}R^2e^{2i\theta}-\frac{1}{4R^2}-\frac{\beta q_n}{2n}(R^2+\frac{1}{R^2}+1)-\frac{\beta2p_n}{an}\log R\Big)\nonumber
\end{align}
uniformly for $-\beta<\alpha_i<\beta\log(n)$, $|\beta_i|\leq \beta$. Since $r_+(\theta)\rarrow \infty$ as $\theta\rarrow 0$, we take $R$ so large that $2\arg z_R^*<\pi/6$, i.e. $\cos(2\arg z_R^*)>1/2$. So $\mcR\frac{1}{4}R^2e^{2i\theta}>\frac{1}{8}R^2$. Since $q_n/n,p_n/(an)\rarrow 0$, take $R$ large (depending only on $\beta$) so that
\begin{align}
\frac{R^2}{8}-\frac{1}{4R^2}-\frac{\beta q_n}{2n}(R^2+\frac{1}{R^2}+1)-\frac{\beta2p_n}{an}\log R\geq \frac{R^2}{12}
\end{align}
for all $n$ large enough. Finally, take $n$ so large that $\frac{an}{2}R^2/12+a^2nR_2'(w_2)\geq \frac{an}{2}R^2/16$. We have established
\begin{align}
\abs{\frac{H_{y_1,y_2+1}(w_2)}{H_{y_1,y_2+1}(i)}}\geq R^{1+\frac{an}{2}}\exp(an\frac{R^2}{32}).
\end{align} 
Let $w_1=\frac{1}{R}e^{i\theta}\in desc_R^\eps\cap \partial B_{1/R}(0)$, $\theta\in (0,\arg(z_{1/R}^{**}))$, a similar argument yields
\begin{align}
\abs{\frac{H_{x_1+1,x_2}(w_1)}{H_{x_1+1,x_2}(i)}}\leq e^{-anR^2/32}.
\end{align} 
The bounds extend by symmetry to the rest of $asc_R^\eps\cap \partial B_R(0)$ and $desc_{1/R}^\eps\cap \partial B_{1/R}(0)$.

Now we consider the section of the integral $B_{\eps_1,\eps_2}(\tilde{x}_j,\tilde{y}_i)$ contained in $desc_{1/R}^\eps\cap B_\eps(i)\times asc_R^\eps\cap B_\eps(i)$. We use Taylor's theorem to write this as
\begin{align}
\label{tmepf,ff}2\frac{i^{(x_2-x_1+y_1-y_2)/2}}{(2\pi i)^2}\int_{desc_{1/R}^\eps\cap B_\eps(i)}\frac{dw_1'}{w_1'}&\int_{asc_R^\eps\cap B_\eps(i)}dw_2'\Big(\frac{V_{\eps_1,\eps_2}(i,i)}{i}\big(1+R_3(a)(|w_1'-i|+|w_2'-i|)\big)\Big)\\&\quad\quad\times\frac{(-iw_1')^{-\eps_1-1}(-iw_2)^{-\eps_2-1}}{w_2'-w_1'}\frac{H_{x_1+1,x_2}(w_1')}{H_{y_1,y_2+1}(w_2')}+o(1).\nonumber
\end{align}
We set $w_1'=i+w_1(an)^{-1/3}, w_2'=i+w_2(an)^{-1/3}$ and use lemma \ref{saddlepointexpprop1}  to rewrite
\begin{align}
\label{firstpartintnorem}2\frac{i^{(x_2-x_1+y_1-y_2)/2}}{(2\pi i)^2}&\int_{desc_{1/R}^\eps\cap B_\eps(i)}dw_1'\int_{asc_R^\eps\cap B_\eps(i)}dw_2'\frac{V_{\eps_1,\eps_2}(i,i)}{i}\\&\quad\times\frac{(-iw_1')^{-\eps_1-1}(-iw_2)^{-\eps_2-1}}{w_2'-w_1'}\frac{H_{x_1+1,x_2}(w_1')}{H_{y_1,y_2+1}(w_2')}\nonumber
\end{align}
as $2P(\tilde{x}_j,\tilde{y}_i)$ multiplied by
\begin{align}
&\frac{1}{(2\pi i)^2}\int_{C_1}dw_1\int_{C_2}dw_2 \frac{1}{w_1-w_2}\exp\big(-iw_1\alpha_j+iw_2\alpha_i-w_1^2\beta_j+w_2^2\beta_i-\frac{i}{3}w_1^3+\frac{i}{3}w_2^3\label{temp5341}\\&\quad+\frac{(|w_1|+|w_1|^4) }{(an)^{1/3}}R_1(w_1,a)+\frac{(|w_2|+|w_2|^4) }{(an)^{1/3}}R_2(w_2,a)+a^2nR(w_1,w_2,a)\big)\nonumber
\end{align}
where $C_1$ and $C_2$ are the images of $desc_{1/R}^\eps\cap B_\eps(i)$ and $asc_R^\eps\cap B_\eps(i)$ under the map $w\mapsto (an)^{1/3}(w-i)$, respectively. Use the inequality $|e^t-1|\leq |t|e^{|t|}$ and \eqref{Rem146} to get that the absolute value of the difference between \eqref{temp5341} and
\begin{align}
\frac{1}{(2\pi i)^2}\int_{C_1}dw_1\int_{C_2}dw_2 \frac{1}{w_1-w_2}\exp\big(-iw_1\alpha_j+iw_2\alpha_i-w_1^2\beta_j+w_2^2\beta_i-\frac{i}{3}w_1^3+\frac{i}{3}w_2^3\big)\label{airycontribution}
\end{align}
is bounded above by
\begin{align}
&\label{temp5314}\frac{C'}{4\pi^2}\int_{C_1}|dw_1|\int_{C_2}|dw_2| \exp\big(\mcR[-iw_1\alpha_j+iw_2\alpha_i-w_1^2\beta_j+w_2^2\beta_i-\frac{i}{3}w_1^3+\frac{i}{3}w_2^3]\\&\quad+C\frac{(|w_1|+|w_2|+|w_1|^4+|w_2|^4) }{(an)^{1/3}}+C''a^2n\big)(\frac{(|w_1|+|w_2|+|w_1|^4+|w_2|^4)}{(an)^{1/3}|w_1-w_2|}+\frac{C''a^2n}{|w_1-w_2|}).\nonumber
\end{align}
We see that $C_1$ and $C_2$ are straight lines and they can be parametrised in the form $w_1=re^{i\theta}$ for $r\in [0,\eps (an)^{1/3}]$, $\theta$ fixed in $[-\pi/6-\delta,-\pi/6+\delta]\cup [-5\pi/6-\delta,-5\pi/6+\delta]$ and $w_2=se^{i\varphi}$ for $s\in [0,\eps (an)^{1/3}]$, $\varphi$ fixed in $[\pi/6-\delta,\pi/6+\delta]\cup [5\pi/6-\delta,5\pi/6+\delta]$. We use the inequalities \eqref{epssmallcond}, \eqref{deltasmallcond} to see that the integral in \eqref{temp5314} is bounded above by integrals of the form
\begin{align}
\label{tempmnnmnm}\int_0^\infty dr \int_0^\infty ds\exp\big(r(\alpha_j\sin\theta+\frac{C}{(an)^{1/3}})-s(\alpha_i\sin\varphi-\frac{C}{(an)^{1/3}})\\-r^2\cos(2\theta)\beta_j+s^2\cos(2\varphi)\beta_i-\frac{r^3}{12}-\frac{s^3}{12}\big)\frac{r+s+r^4+s^4}{|r-s|}\nonumber
\end{align}
which are bounded uniformly in $n$ and $\alpha_i,\alpha_j,\beta_i,\beta_j$ (indeed,  observe the bound becomes smaller as $\alpha_j,\alpha_i$ become large). Hence the term \eqref{temp5314} goes to faster than $C(an)^{-1/3}+C'a^2n$ as $n\rarrow \infty$. Looking at the integrand in \eqref{airycontribution}, for $|w_1|, |w_2|$ large, the exponent is  $\sim-iw_1^3/3+iw_2^3/3$. By analyticity of the integrand, we see that in extending the length of the straight lines in \eqref{airycontribution} to infinity, we can deform the angles of the contours slightly to obtain
\begin{align}
-e^{\beta_j\alpha_j-\beta_i\alpha_i+\frac{2}{3}(\beta_j^3-\beta_i^3)} i \tilde{A}(-\beta_j,\alpha_j+\beta_j^2;-\beta_i,\alpha_i+\beta_i^2)
\end{align}
where $\tilde{A}$ is the Airy part of the extended-Airy kernel \eqref{airypart}. 
We see this is the contribution to the limit in \eqref{blimtairy}.
From here it is straightforward to show that the remainder given by \eqref{tmepf,ff} minus \eqref{firstpartintnorem} is the prefactor $P_{\eps_1,\eps_2}(x,y)$ multiplied by a term bounded above by $C(an)^{-1/3}$ for some $C>0$, which goes to zero uniformly in $\alpha_i,\beta_i, \alpha_j,\beta_j$.

All that remains is show that the rest of the integral over $\big(desc_{1/R}^\eps\cap B_\eps(i)\times asc_R^\eps\cap B_\eps(i)\big)^c$ goes to zero as $n\rarrow \infty$. We divide this into three parts; $desc_{1/R}^\eps\cap B_\eps(i)^c\times asc_R^\eps\cap B_\eps(i)$, $desc_{1/R}^\eps\cap B_\eps(i)\times asc_R^\eps\cap B_\eps(i)^c$ and $desc_{1/R}^\eps\cap B_\eps(i)^c\times asc_R^\eps\cap B_\eps(i)^c$. Each of these three parts can be bounded in a similar fashion so we just focus on one. In particular, we consider
\begin{align}
\frac{i^{(x_2-x_1+y_1-y_2)/2}}{(2\pi i)^2}\int_{desc_{1/R}^\eps\cap B_\eps(i)^c}\frac{dw_1'}{w_1'}\int_{asc_R^\eps\cap B_\eps(i)}dw_2'\frac{V_{\eps_1,\eps_2}(w_1',w_2')}{w_2'-w_1'}\frac{H_{x_1+1,x_2}(w_1')}{H_{y_1,y_2+1}(w_2')}.\label{temp3498}
\end{align}
The contours in \eqref{temp3498} do not cross, so we can use \eqref{Vuniformbdd} to get a $C>0$ such that
\begin{align}
\sup_{(w_1,w_2)}|\frac{V_{\eps_1,\eps_2}(w_1,w_2)}{w_1-w_2}|\leq C a^{(\eps_1+\eps_2)/2}
\end{align}
where the supremum is over $desc_{1/R}^\eps\cap B_\eps(i)^c\times asc_R^\eps\cap B_\eps(i)$. So \eqref{temp3498} is bounded above by
\begin{align}
\label{temp4209} C|P_{\eps_1,\eps_2}(x,y)|\int_{desc_{1/R}^\eps\cap B_\eps(i)^c}|dw_1|\abs{\frac{H_{x_1+1,x_2}(w_1)}{H_{x_1+1,x_2}(i)}}\\\quad\quad\cdot\int_{asc_R^\eps\cap B_\eps(i)}|dw_2|\abs{\frac{H_{y_1,y_2+1}(i)}{H_{y_1,y_2+1}(w_2)}}\nonumber
\end{align}
where $P_{\eps_1,\eps_2}(x,y)$ was defined by \eqref{prefact}.
As before we map $asc_R^\eps\cap B_\eps(i)$ under $w_2\rarrow (an)^{1/3}(w_2-i)$ with image $C_2$ and parametrise $w_2$ in the $s$-variable. The bounds \eqref{deltasmallcond} and \eqref{epssmallcond} give that the second integral in \eqref{temp4209} is bounded above by integrals of the form
\begin{align}
\int_0^\infty \frac{ds}{(an)^{1/3}}\exp(-s(\alpha_i\sin(\varphi)-\frac{C}{(an)^{1/3}})+s^2\cos(2\varphi) \beta_i-\frac{s^3}{12})
\end{align} 
which is bounded uniformly, just as before. The curve $desc_{1/R}^\eps\cap B_\eps(i)^c\cap B_R(0)$ is a path of steepest descent, so
\begin{align}
\abs{\frac{H_{x_1+1,x_2}(w_1)}{H_{x_1+1,x_2}(i)}}\leq \abs{\frac{H_{x_1+1,x_2}(w_1^*)}{H_{x_1+1,x_2}(i)}}\label{temp5634}
\end{align}
where $w_1^*\in desc_{1/R}^\eps\cap \partial B_{\eps}(i)$. Under the map $w\rarrow (an)^{1/3}(w-i)$, $w_1^*$ is the one of the endpoints of the two straight lines $C_1$ originating at $i$. Hence similar to before we apply Lemma \ref{saddlepointexpprop1} and use \eqref{deltasmallcond}, \eqref{epssmallcond} to see \eqref{temp5634} is bounded above by 
\begin{align}
\exp(\eps(an)^{1/3}(\alpha_j \sin\theta+\frac{C}{(an)^{1/3}})-\eps^2(an)^{2/3}\cos(2\theta)\beta_j-\frac{\eps^3 an}{12})\leq C'\exp(-\frac{\eps^3an}{24})
\end{align} for some $C'>0$ uniformly in $n$ large enough and for $\alpha_j,\beta_j$. Putting these bounds together we get that the modulus of \eqref{temp3498} is bounded above by $C|P_{\eps_1,\eps_2}(\tilde{x}_j,\tilde{y}_i)|e^{-\eps^3an/24}$ for some $\eps,C>0$. This shows the contribution to $\tilde{B}_{\eps_1,\eps_2}(\tilde{x}_j,\tilde{y}_i)$ goes to zero exponentially in $an$.
\end{proof}
\end{proposition}

\begin{remark}[The Discrete Bessel Kernel]\label{BesselKernel}
In the case that the model parameter $a$ is such that
\begin{align}
an\rarrow 4\nu>0,
\end{align}
we see the Discrete Bessel Kernel. In particular, for $a$-dimers $e_j=(\tilde{x}_j,\tilde{y}_j)\in W_{\eps_1}\times B_{\eps_2}$ in the pre-squished graph $AD$, with $\tilde{x}_j=(x_1,x_2)$ and $\tilde{y}_j=(y_1,y_2)$ let
 \begin{align}
 \tilde{x}_j=(\frac{n}{2}+1+\tilde{\alpha}_j)\vec{e}_1+(0,2\eps_1-1),\\
 \tilde{y}_j=(\frac{n}{2}+1+\tilde{\alpha}_j)\vec{e}_1+(2\eps_2-1,0)
 \end{align}
 where $\tilde{\alpha}_j,\in 2\Z$ lies in a bounded set.
Recall the representation
\begin{align}
B_{\eps_1,\eps_2}(a,\tilde{x}_j,\tilde{y}_i)=\frac{i^{(x_2-x_1+y_1-y_2)/2}}{(2\pi i)^2}\int_{\Gamma_r}\frac{dw_1}{w_1}\int_{\Gamma_{1/r}}dw_2\frac{w_2}{w_2^2-w_1^2}\frac{H_{x_1+1,x_2}(w_1)}{H_{y_1,y_2+1}(w_2)}V_{\eps_1,\eps_2}(w_1,w_2)
\end{align}
in Corollary 2.4 in \cite{C/J}. Make the shift $\alpha_jp_n=\tilde{\alpha_j}-n(\xi_c+1/2)=\tilde{\alpha_j}-an/2+O(a^2n)$ in \eqref{Hsaddlessimpler} and use Lemma \ref{Veps1eps2formula} to get that when $\eps_1=\eps_2=0$,
\begin{align}
-aiB_{\eps_1,\eps_2}(a,\tilde{x}_j,\tilde{y}_i)\rarrow\frac{1}{(2\pi i)^2}\int_{\Gamma_r}\frac{dw_1}{w_1}\int_{\Gamma_{1/r}}dw_2\frac{w_2}{w_2^2-w_1^2}\frac{e^{\frac{\nu}{2}(w_1^2-1/w_1^2)+\tilde{\alpha}_j\log(w_1)}}{e^{\frac{\nu}{2}(w_2^2-1/w_2^2)+\tilde{\alpha}_i\log(w_2)}}\frac{w_1}{w_2}.
\end{align}
Making the change of variables $w_1^2=w, w_2^2=z$, one sees the above integral is
\begin{align}
\K_{Bessel}(\tilde{\alpha}_i,\tilde{\alpha}_j,(\nu/2)^2)=\frac{1}{(2\pi i)^2}\iint\limits_{|z|>|w|}dzdw\frac{e^{\frac{\nu}{2}(z-z^{-1}-w+w^{-1})}}{z-w}\frac{1}{z^{\tilde{\alpha}_i+1/2}w^{-\tilde{\alpha}_j+1/2}}.
\end{align} 
\end{remark}

\section{Asymptotics of $\K_{1,1}^{-1}$}\label{asympk11inv}
Let $e_i=(\tilde{x}_i,\tilde{y}_i)$, $e_j=(\tilde{x}_j,\tilde{y}_j)$ be two $a$-dimers as in \ref{cond-a-dimers}. Let $(x_1,x_2)=\tilde{x}_j\in W_{\eps_1}$, $(y_1,y_2)=\tilde{y}_i\in B_{\eps_2}$. In this section we obtain the asymptotics of $\K_{1,1}^{-1}(x_1,x_2,y_1,y_2)$  defined by \eqref{K_11}. We substitute $x_1,x_2,y_1,y_2$ into formulas \eqref{klh1}, \eqref{klh2} and obtain
\begin{align}
k_1&=\frac{1}{2}\big(p_n(\alpha_j-\alpha_i)+q_n(\beta_j-\beta_i)\big)-(2\eps_1-1)(\eps_2-1),\label{k1temp}\\
\ell_1&=\frac{1}{2}\big(p_n(\alpha_i-\alpha_j)+q_n(\beta_j-\beta_i)\big)+\eps_2-1.\label{l1temp}
\end{align}
Observe that if $\beta_j\neq \beta_i$ then $|k_1|,|l_1|=q_n|\beta_j-\beta_i|/2(1+O(\log(n)p_n/q_n)$ and if $\beta_j=\beta_i$ then $|k_1|,|l_1|=p_n|\alpha_j-\alpha_i|/2(1 +O(1/p_n))$. Hence we compute asymptotics of \eqref{EKL1} for 
\begin{align}
\ell=m(1+a'_m)>0,\quad k=m(1+a_m)>0
\label{ellk1}
\end{align} where $\ell+k$ is even, and 
\begin{align}
a_m,a'_m=o(m^{-1/2}), \quad ma^2=o(1) \text{ and } ma\rarrow \infty.\label{conditionsaergav}
\end{align} This will produce the asymptotics applicable to the regime in which $\beta_j\neq \beta_i$. To see this, set $m=\frac{1}{2}q_n|\beta_j-\beta_i|$, $ma_m=|\ell_1|-m$, $ma'_m=|k_1|-m$, (noting the absolute values in \eqref{ADPAD}) and observe the conditions \eqref{conditionsaergav}, $\log(n)p_n/q_n=o(q_n^{-1/2})$, $q_na^2=o(1)$ and $aq_n\rarrow \infty$, are indeed satisfied. The case $\beta_j=\beta_i$ is dealt with at the end of the proof of Proposition \ref{propgasker}. 
We have
\begin{align}
G(w)^\ell G(1/w)^k&=\exp\big( m(\ln(G(w))+\ln (G(1/w)))+ma_m\ln(G(w))+ma'_m\ln(G(1/w))\big)\label{Gexp}\\
&=\exp\big(mh(w)+ma_m h_1(w)+ma'_mh_1(1/w)+o(1)\big)\nonumber
\end{align}
where $h(w)=h_1(w)+h_1(1/w)$ and $h_1(w)=\ln(-\sqrt{a/2}/w)-a/(2w^2)$.
It is easy to see $\mcI[h(e^{i\theta})]=0$ and $h(e^{i\theta})=\overline{h(e^{-i\theta})}$ for $\theta\in [-\pi,\pi]$ hence the unit semi-circle (in the top half plane) is either a steepest ascent or descent path for $h$. Let $0<\eps<\pi/2$, Taylor's theorem applied to $\theta\mapsto h_1(ie^{i\theta})$  gives two numbers $\xi_1,\xi_2\in (-\eps,\eps)$ such that
\begin{align}
&h_1(ie^{i\theta})-h_1(i)=-i(1+a)\theta-a\theta^2+\mcR[\theta^3\frac{d^3}{d\theta^3}h_1(ie^{i\xi_1})]+i\mcI[\theta^3\frac{d^3}{d\theta^3}h_1(ie^{i\xi_2})]\label{h1taylor}
\end{align}
adding the above equation with its complex conjugate we obtain
\begin{align}
h(ie^{i\theta})-h(i)=-2a\theta^2+\mcR[\theta^3\frac{d^3}{d\theta^3}h(ie^{i\xi_1})].\label{htaylor}
\end{align}
From \eqref{htaylor} we see the two paths traversing along the semi-circle starting at $i$ and ending at $\pm 1$ are both paths of steepest descent.
We can calculate
\begin{align}
\frac{d^3}{d\theta^3}h_1(ie^{i\theta})]=4iae^{-2i\theta}=O(a), && \frac{d^3}{d\theta^3}h_1(-ie^{-i\theta})]=-4iae^{2i\theta}=O(a)
\label{Oaboundrem1}
\end{align}
so that the second term on the right hand side of \eqref{htaylor} is of the form $\theta^3\cdot O(a)$. 
\begin{proposition}\label{propeklasymp}
Consider $\ell, k$ as in \eqref{ellk1}. There is a fixed $\eps>0$ such that
\begin{align}
E_{k,\ell}&=\frac{i^{-k-\ell}e^{\ell h_1(i)+k h_1(-i)}}{\sqrt{4\pi} }\Big(\frac{e^{-\frac{m}{8a}(a_m'-a_m)^2(1+a)^2}}{\sqrt{2am}}\\
&+o(\frac{1}{\sqrt{am}})+O(a)+O(\frac{1}{ am})+O(\frac{a_m'+a_m}{\sqrt{am}})+O(e^{-2\eps^2am}))\Big).\nonumber
\end{align}
\begin{proof}
We use \eqref{squarerootsymmetries}, \eqref{Gsymmetries} and \eqref{Gexp} to write
\begin{align}
E_{k,\ell}=\frac{i^{-k-l}}{(1+a^2)2\pi}\mcI\Big[\int_{\Gamma_1\cap \HH}\frac{dw}{w}\frac{\exp(mh(w)+ma_mh_1(w)+m a'_m h_1(1/w))(1+o(1))}{\sqrt{w^2+2c}\sqrt{1/w^2+2c}}\Big].
\label{ekltemp}
\end{align}
Take $\eps>0$ small, we split the integral in \eqref{ekltemp} across the two regions $\Gamma_{1,\eps}$ and $(\Gamma_1\setminus \Gamma_{1,\eps})\cap \HH$ where $\Gamma_{1,\eps}=\{e^{i\theta}:\theta\in(\pi/2-\eps,\pi/2+\eps)\}$. First consider the integral over $\Gamma_{1,\eps}$, parametrise $w(\theta)=ie^{i\theta}$, $\theta\in(-\eps,\eps)$ so that it is
\begin{align}
&\int_{-\eps}^\eps\frac{\exp(mh(ie^{i\theta})+ma_mh_1(ie^{i\theta})+ma_m'h_1(-ie^{-i\theta}))(1+o(1))}{|e^{2i\theta}-2c|}id\theta\label{temp32}\\&=\int_{-\eps}^\eps \exp(mh(ie^{i\theta})+ma_mh_1(ie^{i\theta})+ma_m'h_1(-ie^{-i\theta}))(1+o(1))(1+O(a))id\theta\nonumber.
\end{align}
Note that in \eqref{temp32}, we first used \eqref{squarerootsymmetries} and then applied Taylor's theorem to $a\rarrow 1/|e^{2i\theta}-2c|(=1/|e^{2i\theta}|+O(a))$ then noted $|e^{2i\theta}|=1$.
Use \eqref{h1taylor}, \eqref{htaylor}, \eqref{Oaboundrem1} in  the exponent in \eqref{temp32},
 \begin{align}
 &mh(ie^{i\theta})+ma_mh_1(ie^{i\theta})+ma'_mh_1(-ie^{-i\theta})-C(m,a)\label{taylorsaddle1}\\
 &=-2am\theta^2-ma_mi(1+a)\theta+ma'_mi(1+a)\theta - am (a_m+a_m')\theta^2+m\theta^3\cdot O(a)\nonumber
 \end{align}
 where $C(m,a)=mh(i)+ma_mh_1(i)+ma'_mh_1(-i)$. 
 Take $\eps>0$ so small and $m$ large enough that
 \begin{align}
 -2am \theta^2+am(a_m+a'_m)\theta^2+m|\theta^3|\cdot O(a)<-am\theta^2
 \label{upperboundexp}
 \end{align}
 for $\theta\in (-\eps,\eps)$. Since $\Gamma_1$ is a descent path for $h$ with maximum real part at $i$, we apply \eqref{taylorsaddle1} then use  \eqref{upperboundexp} on the section of the integral in \eqref{ekltemp} over $\Gamma_1\cap \HH\setminus \Gamma_{1,\eps}$, this gives a constant $D>0$ such that
 \begin{align}
 \abs{\int_{\Gamma_1\setminus \Gamma_{1,\eps}\cap\HH}\frac{dw}{w}\frac{\exp(mh(w)+ma_mh_1(w)+m a'_m h_1(1/w))(1+o(1))}{\sqrt{w^2+2c}\sqrt{1/w^2+2c}}}
\leq  C(a,m)D \ e^{-am\eps^2}.
 \end{align}
 Now we have a lemma which gives the leading order term.
 \begin{lemma}
 We have
 \begin{align}
& \int_{-\eps}^\eps\exp(mh(ie^{i\theta})+ma_mh_1(ie^{i\theta})+ma'_mh_1(-ie^{-i\theta}))d\theta\label{inttemp1}\\&=e^{C(m,a)}\Big(\sqrt{\frac{\pi}{2am}}e^{\frac{m}{8a}(a_m'-a_m)^2(1+a)^2}+O(1/(am))+O\Big(\frac{a_m+a'_m}{\sqrt{am}}\Big)+O(e^{-2\eps^2am})\Big)\nonumber
 \end{align}
 \begin{proof}
 Use \eqref{taylorsaddle1} on the left hand side of \eqref{inttemp1} and perform a change of variables $\theta\rarrow \theta/\sqrt{2am}$ to see \eqref{inttemp1} is equal to 
 \begin{align}
 \frac{e^{C(a,m)}}{\sqrt{2am}}\int_{-\eps\sqrt{2am}}^{\eps\sqrt{2am}}\exp(-\theta^2+\frac{i\theta}{\sqrt{2am}}(ma'_m-ma_m)(1+a)+\theta^3\cdot O(\frac{1}{\sqrt{am}})-\frac{a_m+a'_m}{2}\theta^2).\label{temp343}
 \end{align}
 We can use \eqref{upperboundexp} (with $\theta$ substituted for $\theta/\sqrt{2am}$) and the bound $|e^t-1|\leq |t|e^{|t|}$ to get that the absolute value of the difference between \eqref{temp343} and
 \begin{align}
  \frac{e^{C(a,m)}}{\sqrt{2am}}\int_{-\eps\sqrt{2am}}^{\eps\sqrt{2am}}e^{-\theta^2+\sqrt{\frac{m}{2a}}i\theta(a'_m-a_m)(1+a)}d\theta\label{temp413}
 \end{align}
 is bounded above by
 \begin{align}
& \frac{e^{C(a,m)}}{\sqrt{2am}}\int_{-\eps\sqrt{2am}}^{\eps\sqrt{2am}}\big(O(\frac{1}{\sqrt{am}})|\theta|^3+|\frac{a_m+a'_m}{2}|\theta^2\big)e^{-\theta^2/2}d\theta\\&< \frac{e^{C(a,m)}}{\sqrt{2am}}\big(O(\frac{1}{\sqrt{am}})\int_{-\infty}^{\infty}|\theta|^3e^{-\theta^2/2}d\theta+|\frac{a_m+a'_m}{2}|\int_{-\infty}^\infty\theta^2e^{-\theta^2/2}d\theta\big).\nonumber\\
 \end{align}
Hence we have two of the remainders in \eqref{inttemp1}.
 Observe that 
 \begin{align}
 \frac{e^{C(a,m)}}{\sqrt{2am}}\int_{-\infty}^\infty e^{-\theta^2+\sqrt{\frac{m}{2a}}i\theta(a'_m-a_m)(1+a)}d\theta=e^{C(a,m)}\sqrt{\frac{\pi}{2am}}e^{-\frac{m}{8a}(a'_m-a_m)^2(1+a)^2}\label{temp2}
 \end{align}
 is the Fourier transform of a Gaussian function. Finally, we see that the absolute value of the difference between \eqref{temp413} and \eqref{temp2} is $O(e^{-2\eps^2am})$ by a standard bound on the complementary error function \cite{DLMF}.
 \end{proof}
  \end{lemma}
 Further bounds on the difference between \eqref{temp32} and \eqref{inttemp1} yield the result.
\end{proof}
\end{proposition}

Now we have enough to give the proposition concerning the asymptotics of $\K_{1,1}^{-1}$. 
\begin{proposition} \label{propgasker} Let $e_i=(\tilde{x}_i,\tilde{y}_i)$, $e_j=(\tilde{x}_j,\tilde{y}_j)$ be two $a$-dimers as in condition \ref{cond-a-dimers} and let $(x_1,x_2)=\tilde{x}_j\in W_{\eps_1}$, $(y_1,y_2)=\tilde{y}_i\in B_{\eps_2}$, then
\begin{align}
&\big(i^{x_1-y_1-1}\mathcal{G}^{\frac{2+x_1-x_2-y_1+y_2}{2}}\mathrm{g}_{\eps_1,\eps_2}^{-1}(an)^{1/3}\big)\K_{1,1}^{-1}(x_1,x_2,y_1,y_2)\label{limgasker}\\
&=e^{\beta_j\alpha_j-\beta_i\alpha_i+\frac{2}{3}(\beta_j^3-\beta_i^3)} \ind_{\beta_i<\beta_j}\Psi(-\beta_j,\alpha_j+\beta_j^2;-\beta_i,\alpha_i+\beta_i^2)(1+o(1))\nonumber
\end{align}
in the limit $n\rarrow \infty$.
\begin{proof}
Let $\beta_j\neq \beta_i$. From the formulas \eqref{klh1}, \eqref{klh2}, for $n$ large enough we have
\begin{align}
&|k_1|=\frac{1}{2}q_n|\beta_j-\beta_i|+\frac{\sigma}{2}p_n(\alpha_j-\alpha_i)-\sigma(2\eps_1-1)(\eps_2-1),\\
&|\ell_1|=\frac{1}{2}q_n|\beta_j-\beta_i|+\frac{\sigma}{2}p_n(\alpha_i-\alpha_j)+\sigma(\eps_2-1),\\
&|k_2|=\frac{1}{2}q_n|\beta_j-\beta_i|+\frac{\sigma}{2}p_n(\alpha_j-\alpha_i)+\sigma\eps_2(2\eps_1-1),\\
&|\ell_2|=\frac{1}{2}q_n|\beta_j-\beta_i|+\frac{\sigma}{2}p_n(\alpha_i-\alpha_j)+\sigma\eps_2.
\end{align}
where $\sigma:=\text{sign}(\beta_j-\beta_i)$. 
We use Proposition \ref{propeklasymp} to obtain formulas for the asymptotics of $E_{k_1,\ell_1}$ and  $E_{k_2,\ell_2}$.
For instance, take $m=\frac{1}{2}q_n|\beta_j-\beta_i|$, $ma_m=|\ell_1|-m$, $ma'_m=|k_1|-m$ in Proposition \ref{propeklasymp} to obtain the asymptotics of $E_{k_1,\ell_1}$. Since $p_n^2/(aq_n)=1$, a Taylor expansion gives
\begin{align}
\frac{e^{-\frac{m}{8a}(a'_m-a_m)^2(1+a)^2}}{\sqrt{2am}}=\frac{e^{-\frac{(\alpha_j-\alpha_i)^2}{4|\beta_j-\beta_i|}}\big(1+O(\frac{1}{(an)^{1/3}})\big)}{(an)^{1/3}\sqrt{|\beta_j-\beta_i|}}.
\end{align}
We also have that $e^{h(i)}=i\mathcal{G}$ and $e^{h(-i)}=\mathcal{G}/i$ where we recall \eqref{Gcurly}. So we obtain
\begin{align}
E_{k_1,\ell_1}=& \ i^{\sigma p_n(\alpha_i-\alpha_j)+q_n(\beta_i-\beta_j)}\mathcal{G}^{q_n|\beta_j-\beta_i|}\frac{e^{-\frac{(\alpha_j-\alpha_i)^2}{4|\beta_j-\beta_i|}}}{(an)^{1/3}\sqrt{4\pi |\beta_j-\beta_i|}}(1+o(1))\\&
 \Big(i^{2\sigma(\eps_2-1)(2\eps_1-1)}\big(\frac{a}{2}\big)^{\sigma(1-\eps_1)(\eps_2-1)}e^{\sigma a(1-\eps_1)(\eps_2-1)}\nonumber\Big).
\end{align}
Similarly,
\begin{align}
E_{k_2,\ell_2}=& \ i^{\sigma p_n(\alpha_i-\alpha_j)+q_n(\beta_i-\beta_j)}\mathcal{G}^{q_n|\beta_j-\beta_i|}\frac{e^{-\frac{(\alpha_j-\alpha_i)^2}{4|\beta_j-\beta_i|}}}{(an)^{1/3}\sqrt{4\pi |\beta_j-\beta_i|}}(1+o(1))\\&
 \Big(i^{2\sigma \eps_2(1-2\eps_1)}\big(\frac{a}{2}\big)^{\sigma\eps_1\eps_2}e^{\sigma a\eps_1\eps_2}\nonumber\Big).
\end{align}
These expressions go into the formula for $\K_{1,1}^{-1}$ given by \eqref{K_11}. We recall \eqref{K_11} as
\begin{align}
\label{temp0940}\K_{1,1}^{-1}(x_1,x_2,y_1,y_2)&=-i^{1+h(\eps_1,\eps_2)}(a^{\eps_2}E_{k_1,\ell_1}+a^{1-\eps_2}E_{k_2,\ell_2})\\
&=-i^{1+h(\eps_1,\eps_2)+\sigma p_n(\alpha_i-\alpha_j)+q_n(\beta_i-\beta_j)}\mathcal{G}^{q_n|\beta_j-\beta_i|}\frac{e^{-\frac{(\alpha_j-\alpha_i)^2}{4|\beta_j-\beta_i|}}W_{\eps_1,\eps_2}}{(an)^{1/3}\sqrt{4\pi |\beta_j-\beta_i|}}(1+o(1))\nonumber
\end{align}
where $W_{\eps_1,\eps_2}$ is the expression
\begin{align}
W_{\eps_1,\eps_2}= a^{\eps_2}i^{2\sigma(\eps_2-1)(2\eps_1-1)}\big(\frac{a}{2}\big)^{\sigma(1-\eps_1)(\eps_2-1)}e^{\sigma a(1-\eps_1)(\eps_2-1)}+a^{1-\eps_2}i^{2\sigma \eps_2(1-2\eps_1)}\big(\frac{a}{2}\big)^{\sigma\eps_1\eps_2}e^{\sigma a\eps_1\eps_2}.
\end{align}
Since
\begin{align}
\label{temp3901}\mathcal{G}^{\frac{2+x_1-x_2-y_1+y_2}{2}}=\mathcal{G}^{-q_n(\beta_j-\beta_i)}\big(\frac{a}{2}\big)^{\frac{2-\eps_1-\eps_2}{2}}e^{\frac{a}{2}(2-\eps_1-\eps_2)},
\end{align}
substituting \eqref{temp0940} into \eqref{limgasker} we see that the factor $\mathcal{G}^{q_n|\beta_j-\beta_i|-q_n(\beta_j-\beta_i)}$ appearing in the left-hand side of \eqref{limgasker} dominates the asymptotics. Indeed, if $\beta_j<\beta_i$ then $\mathcal{G}^{q_n|\beta_j-\beta_i|-q_n(\beta_j-\beta_i)}$ tends to zero and if $\beta_j>\beta_i$ it is is equal to one. So this factor converges to the indicator function on the right-hand side of \eqref{limgasker}. 

Now we assume that $\beta_j>\beta_i$ so that $\sigma=1$. We use \eqref{temp0940}, \eqref{temp3901} together with $i^{x_1-y_1-1}=i^{p_n(\alpha_j-\alpha_i)-q_n(\beta_j-\beta_i)-2\eps_2}$ to write the left-hand side of \eqref{limgasker} as
\begin{align}
\ind_{\beta_j>\beta_i}\frac{e^{-\frac{(\alpha_j-\alpha_i)^2}{4(\beta_j-\beta_i)}}}{\sqrt{4\pi(\beta_j-\beta_i)}}\tilde{W}_{\eps_1,\eps_2}+o(1)
\end{align}
where 
\begin{align}
\tilde{W}_{\eps_1,\eps_2}=-i^{1+h(\eps_1,\eps_2)+2q_n(\beta_i-\beta_j)-2\eps_2}\big(\frac{a}{2}\big)^{\frac{2-\eps_1-\eps_2}{2}}e^{\frac{a}{2}(2-\eps_1-\eps_2)}\mathrm{g}_{\eps_1,\eps_2}^{-1}W_{\eps_1,\eps_2}.
\end{align}
By substitution we get that
\begin{align}
\frac{e^{-\frac{(\alpha_j-\alpha_i)^2}{4(\beta_j-\beta_i)}}}{\sqrt{4\pi(\beta_j-\beta_i)}}=e^{\beta_j\alpha_j-\beta_i\alpha_i+\frac{2}{3}(\beta_j^3-\beta_i^3)} \Psi(-\beta_j,\alpha_j+\beta_j^2;-\beta_i,\alpha_i+\beta_i^2),
\end{align}
so we see that we will be done if we show $\tilde{W}_{\eps_1,\eps_2}\rarrow 1$ as $a\rarrow 0$ (i.e $n\rarrow \infty$).
Because of \eqref{alphabetarestric} we see that $i^{2q_n(\beta_i-\beta_j)}=1$ and from \eqref{temp23fd} and \eqref{HHH} we get
\begin{align}
\tilde{W}_{\eps_1,\eps_2}=-i^{2\eps_2\eps_1}\big(\frac{a}{2}\big)^{1-\eps_1-\eps_2}e^{\frac{a}{2}(2-\eps_1-\eps_2)}W_{\eps_1,\eps_2}.
\end{align} From here one can easily check the four cases $\eps_1,\eps_2\in \{0,1\}$ that $\tilde{W}_{\eps_1,\eps_2}\rarrow 1$. For example 
\begin{align}
\tilde{W}_{1,1}=-i^2\frac{2}{a}\big(W_{1,1}\big)=\frac{2}{a}\big(a(1+O(a))-\frac{a}{2}(1+O(a))\big)=(2(1+O(a))-(1+O(a)))\rarrow 1.
\end{align}
Hence we have proven \eqref{limgasker} in the case $\beta_j\neq \beta_i$. Now assume $\beta_i=\beta_j$ and $\alpha_j\neq \alpha_i$ so that
\begin{align}
|\ell_s|,|k_s|=\frac{1}{2}p_n|\alpha_j-\alpha_i|(1+O(an)^{-1/3})\leq d_1(an)^{1/3},\quad s=1,2
\end{align}
for some $d_1>0$.
The modulus of \eqref{limgasker} is bounded above by 
\begin{align}
&|\mathcal{G}|^{\frac{2+x_1-x_2-y_1+y_2}{2}}\big(\frac{a}{2}\big)^{-\frac{\eps_1+\eps_2}{2}}(1+O(a))(an)^{1/3}\big(a^{\eps_2}|E_{k_1,\ell_1}+a^{1-\eps_2}|E_{k_2,\ell_2}|\big)\label{tempbound1}\\
&\leq Ca^{2-\eps_1}(an)^{1/3}\big(|E_{k_1,\ell_1}|+a^{1-2\eps_2}|E_{k_2,\ell_2}|\big)\nonumber
\end{align} for some $C>0$, where we used \eqref{xlocs}, \eqref{ylocs}, \eqref{temp23fd} and \eqref{K_11}. From \eqref{EKL1} we have 
\begin{align}
|E_{k_s,\ell_s}|=\frac{1}{4\pi(1+a^2)}|\int_{\Gamma_1}\frac{dw}{w}\frac{G(w)^{\ell_s}G(1/w)^{k_s}}{\sqrt{w^2+2c}\sqrt{1/w^2+2c}}\leq \tilde{C}\sup|G(w)|^{\ell_s}|G(1/w)|^{k_s}.\label{temp2030}
\end{align}
for some $\tilde{C}>0$.
We also have
\begin{align}
|G(w)|^{\ell_s}|G(1/w)|^{k_s}\leq |G(w)G(1/w)|^{d_1(an)^{1/3}}\leq \big(\sqrt{\frac{a}{2}}\big)^{d_1(an)^{1/3}}(1+O(a))^{d_1(an)^{1/3}}.
\end{align}
Hence
\begin{align}
|E_{k_s,\ell_s}|\leq C\big(\frac{a}{2}\big)^{d_1(an)^{1/3}}.
\end{align} for some $C>0$. Inserting this into the right-hand side of \eqref{tempbound1} and taking the limit $n\rarrow \infty$, we see \eqref{limgasker} tends to zero.

Finally, we assume that $\beta_j=\beta_i$ and $\alpha_j=\alpha_i$. We get 
\begin{align}
|k_1|=|\ell_1|=1-\eps_2,\quad |k_2|=|\ell_2|=\eps_2,
\end{align}
similar to before
\begin{align}
|G(w)|^{\ell_1}|G(1/w)|^{k_1}\leq C|\sqrt{\frac{a}{2}}(1+O(a))|^{2(1-\eps_2)}\leq Ca^{1-\eps_2}
\end{align}
and 
\begin{align}
|G(w)|^{\ell_2}|G(1/w)|^{k_2}\leq Ca^{\eps_2}
\end{align}
for some $C>0$. Hence we get
\begin{align}
|E_{k_1,\ell_1}|\leq C'a^{1-\eps_2},\quad |E_{k_2,\ell_2}|\leq C'a^{\eps_2}
\end{align}
for some $C'>0$.
Inserting these two bounds into the right-hand side of \eqref{tempbound1} and taking the limit $n\rarrow \infty$, we see \eqref{limgasker} tends to zero once again. This completes the proof.
\end{proof}
\end{proposition}

Now we prove Theorem \ref{extendedAiryKernelLimit}.
\begin{proof}[Proof of Theorem \ref{extendedAiryKernelLimit}]\label{proofextendedairy}
By Lemma \ref{B*smallmaindiag}, $|B^*|\leq C_1e^{-C_2n}$.
From proposition \ref{propgasker} and \ref{ThmInvKast} we see we have already established convergence to the Gaussian part of the extended-Airy kernel. 
Hence we need to show
\begin{align}
&i^{x_1-y_1-1}\mathcal{G}^{\frac{2+x_1-x_2-y_1+y_2}{2}}\mathrm{g}_{\eps_1,\eps_2}^{-1}(an)^{1/3}B_{\eps_1,\eps_2}(a,x_1,x_2,y_1,y_2) \\
&\rarrow e^{\beta_j\alpha_j-\beta_i\alpha_i+\frac{2}{3}(\beta_j^3-\beta_i^3)}  \tilde{A}(-\beta_j,\alpha_j+\beta_j^2;-\beta_i,\alpha_i+\beta_i^2)\nonumber
\end{align}
in the limit $n\rarrow \infty$.

Recall that in \eqref{Bepstemp51} we wrote $B_{\eps_1,\eps_2}(a,\tilde{x}_j,\tilde{y}_i)=2P_{\eps_1,\eps_2}(\tilde{x}_j,\tilde{y}_i)\tilde{B}_{\eps_1,\eps_2}(a,\tilde{x}_j,\tilde{y}_i)$ where the prefactor was given by \eqref{prefact}.
So from Proposition \ref{Btildeairy} all we need to do is show that
\begin{align}
i^{x_1-y_1-1}\mathcal{G}^{\frac{2+x_1-x_2-y_1+y_2}{2}}\mathrm{g}_{\eps_1,\eps_2}^{-1}(an)^{1/3}2P_{\eps_1,\eps_2}(\tilde{x}_j,\tilde{y}_i)\label{temp135a}
\rarrow -i^{-1},\nonumber
\end{align}
however this follows from Lemmas \ref{labelH(i)simple}, \ref{Veps1eps2formula} and the bound 
\begin{align}
\frac{\mathcal{G}}{|G(i)|}\leq e^{Ca^2}
\end{align} which follows by \eqref{logGtaylor}.
\end{proof}

\newpage

\section{Backtracking dimers}\label{backtracksection}
In this section we show that the probability of seeing backtracking dimers along a line on the main diagonal in the lower left quadrant is small. For a fixed $t$, define $\gamma'_t$ as the straight line parallel to $\vec{e}_1$, starting at a boundary face in the lower left quadrant and which ends at the face with position $(n(1+\xi_c)+1+\alpha p_n)\vec{e}_1+t'q_n\vec{e}_2$ on $\partial S^T$, where $t'=(2\floor{\frac{tq_n-1}{2}}+2)/q_n$.
\begin{proposition}\label{backtrackingforheightfn}
Fix $t\in \R$, we have
\begin{align}
\P(\{\exists e\in \gamma_t' \text{ which is backtracking} \})\leq \sum_{\substack{e\in \gamma_t';\\e \text{  backtracking}}}\P(e\in \omega)\rarrow 0
\end{align}
as $n\rarrow \infty$.
\begin{proof}Fix $0<\eps<1/2$ and extend $\gamma_t'$ up to the point $n(1-\eps)\vec{e}_1+t'q_n\vec{e}_2$. We prove the statement for this longer path instead. Then at the end of this proof, we will extend the argument to showing there are no back tracknig
Index the backtracking edges which intersect $\gamma_t'$ as follows; they have vertices $(y,x)\in B_1\times W_1$ such that
\begin{align}
x=(x_1,x_2)=(i-t'q_n,i+t'q_n+1)\\
y=(y_1,y_2)=(i-t'q_n+1,i+t'q_n)
\end{align}
where $2\Z+1\ni i\leq n(1-\eps)$ and $i\geq |t'q_n|$. Since $K_{a,1}(b,w)=ai$ for $e\in \gamma'_{t}$, $|\K_{1,1}^{-1}(w,b)|\leq Ca$ by lemma \ref{boundK^-1} and $|B^*_{1,1}(w,b)|\leq C_1e^{-C_2n}$ by Lemma \ref{B*smallmaindiag}, we have that
\begin{align}
\P(e\in \omega)=-aiB_{1,1}(w,b)+O(a^2)+O(e^{-C_2n}).
\end{align}
By subadditivity,
\begin{align}
\P(\{\exists e\in \gamma_t' \text{ which is backtracking} \})\leq \sum_{\substack{e\in \gamma_t';\\e \text{  backtracking}}}\P(e\in \omega).\label{subadditbound}
\end{align}
Hence, we if prove
\begin{align}
\sum_{\substack{e\in \gamma_t';\\e \text{  backtracking}}}-iB_{1,1}(w,b)=O(na)\label{temp30f3l}
\end{align}
we will be done, as $O(na^2)\rarrow 0$ and there are at most $n$ edges intersected by $\gamma_t'$, so we would have the right hand side of \eqref{subadditbound} as $O(na^2)\rarrow 0$. So now we focus on proving \eqref{temp30f3l}.
We substitute $x,y$ into \eqref{Hfunc} (and then into \eqref{Beps1eps2}). We get
\begin{align}
\frac{H_{x_1+1,x_2}(w_1)}{H_{y_1.y_2+1}(w_2)}=\Big(\frac{w_1}{w_2}\Big)^{n/2}\Big(\frac{G(w_1)G(w_2^{-1})}{G(w_1^{-1})G(w_2)}\Big)^{(n-i-1)/2}\Big(\frac{G(w_1)G(w_1^{-1})}{G(w_2)G(w_2^{-1})}\Big)^{t'q_n/2}.
\end{align}
The sum in \eqref{temp30f3l} is then the sum over indices $2\Z+1\ni i\leq n(1-\eps)$ such that $i\geq |t'q_n|$. Make the change of indices $i'=i-|t'q_n|-1$ so that the sum is over $2\Z\ni i' \leq n(1-\eps)-|t'q_n|$ such that $i'\geq 0$. By the geometric sum formula,
\begin{align}
\sum_{i'}\Big(\frac{G(w_2)G(w_1^{-1})}{G(w_1)G(w_2^{-1})}\Big)^{i'/2}=\sum_{\substack{0\leq i \in \Z\\ i\leq (n(1-\eps)-|t'q_n|)/2}}\Big(\frac{G(w_2)G(w_1^{-1})}{G(w_1)G(w_2^{-1})}\Big)^{i}=\frac{1-\big(\frac{G(w_2)G(w_1^{-1})}{G(w_1)G(w_2^{-1})}\big)^{(n(1-\eps)-|t'q_n|)/2+1}}{1-\big(\frac{G(w_2)G(w_1^{-1})}{G(w_1)G(w_2^{-1})}\big)}.
\end{align}
Hence we get that
\begin{align}
\label{sumbacktrackB}&\sum_{\substack{e\in \gamma_t';\\e \text{  backtracking}}}-iB_{1,1}(w,b)\\ &=\frac{1}{(2\pi i)^2}\int_{\Gamma_r}\frac{dw_1}{w_1}\int_{\Gamma_{1/r}}dw_2\frac{V_{1,1}(w_1,w_2)}{w_2-w_1}\Big(\frac{w_1}{w_2}\Big)^{n/2}\Big(\frac{G(w_1)G(1/w_2)}{G(w_2)G(1/w_1)}\Big)^{n/2}\nonumber\\
&\quad\times\frac{G(w_1)G(1/w_2)}{G(w_1)G(1/w_2)-G(w_2)G(1/w_1)}\Big(\frac{G(1/w_2)}{G(1/w_1)}\Big)^{(-|t'q_n|-t'q_n)/2}\Big(\frac{G(w_1)}{G(w_2)}\Big)^{(-|t'q_n|+t'q_n)/2}\nonumber\\
&\quad- \frac{1}{(2\pi i)^2}\int_{\Gamma_r}\frac{dw_1}{w_1}\int_{\Gamma_{1/r}}dw_2\frac{V_{1,1}(w_1,w_2)}{w_2-w_1}\Big(\frac{w_1}{w_2}\Big)^{n/2}\Big(\frac{G(1/w_2)}{G(1/w_1)}\Big)^{(-|t'q_n|-t'q_n)/2}\Big(\frac{G(w_1)}{G(w_2)}\Big)^{(-|t'q_n|+t'q_n)/2}\nonumber\\&\quad\times\frac{G(w_1)G(1/w_2)}{G(w_1)G(1/w_2)-G(w_2)G(1/w_1)}\Big(\frac{G(w_2)G(1/w_1)}{G(w_1)G(1/w_2)}\Big)^{(-|t'q_n|-n\eps)/2+1}.\nonumber
\end{align}
To analyse the two previous integrals, we need a lemma.
\begin{lemma}\label{doublepolelemma}
Fix $0<r<1$. For fixed $w_2\in \Gamma_{1/r}$, the function 
\begin{align}
w_1\mapsto\frac{(w_1-w_2)(w_1+w_2)}{G(w_2)G(1/w_1)-G(w_1)G(1/w_2)}
\end{align} has an analytic extension to the annulus $\{w: r\leq |w|\leq 3/r\}$ for all $a$ small enough.
\begin{proof}
We have
\begin{align}
\label{temp23ewdsas}&G(w_1)G(1/w_2)-G(w_2)G(1/w_1)\\&=\frac{w_1}{2cw_2}(1-\sqrt{1+2c/w_1^2})(1-\sqrt{1+2cw_2^2})-\frac{w_2}{2cw_1}(1-\sqrt{1+2c/w_2^2})(1-\sqrt{1+2cw_1^2})\nonumber\\
&=(-\frac{w_1}{w_2}+\frac{w_2}{w_1})\frac{c}{2}+R(w_1,w_2)c^2.\nonumber
\end{align}
Clearly, \eqref{temp23ewdsas} has zeros at $w_1=w_2$, $w_1=-w_2$. Power series expanding in $c$, observe the first coefficient of the series is only zero at $w_1=w_2,-w_2$, so these are the only zeros. One can further expand \eqref{temp23ewdsas} at $w_1=w_2,-w_2$ to check both zeros are order one. Taking the reciprocal and multiplying by $(w_1+w_2)(w_1-w_2)$ we see the two singularities are removable.
\end{proof}
\end{lemma}
We first consider the second integral in \eqref{sumbacktrackB}. Using the expansions in \eqref{tempGexps234} we get
\begin{align}
\frac{G(1/w_1)}{G(1/w_2)}=\frac{w_1}{w_2}+aR_1(w_1,w_2,a)\label{tempGoverGexpans}.
\end{align}
Note that $|w_1/w_2|=r^2<1$, so
\begin{align}
\Big(\frac{|G(w_2)G(1/w_1)|}{|G(w_1)G(1/w_2)|}\Big)^{(-|t'q_n|-n\eps)/2+1}&=(\frac{|w_1^2|}{|w_2^2|}+aR_2(w_1,w_2,a))^{(-|t'q_n|-n\eps)/2+1}\\&\leq \exp({\frac{1}{2}(n\eps+|t'q_n|-2)(\ln(1/r^4)+Ca)})\nonumber
\end{align}
for some $C>0$.
Similarly we have
\begin{align}
\Big(\frac{|G(1/w_2)|}{|G(1/w_1)|}\Big)^{(-|t'q_n|-t'q_n)/2}\Big(\frac{|G(w_1)|}{|G(w_2)|}\Big)^{(-|t'q_n|+t'q_n)/2}\leq e^{-|t'q_n|\ln (1/r^4)+C'a q_n}.\label{temp34rloope}
\end{align} 
Recall \eqref{definitionofa} and \eqref{definitionofpnqn}.  We see
\begin{align}
\Big (\frac{|w_1|}{|w_2|}\Big)^{n/2}=\exp(\frac{n}{2}\ln r^2).
\end{align}
Now observe that
\begin{align}
\exp(\frac{n}{2}\ln r^2+\frac{n\eps}{2}\ln(1/r^4))\leq e^{-C_4n}
\end{align}
since $\eps<1/2$, and that this bound is extremely small compared to the other bounds. Recall Lemma \ref{Veps1eps2formula}, which gives $|V_{1,1}(w_1,w_2)|\leq Ca$. Using similar expansions on the remaining terms in the integrand, from here it easy to see that that second integral in \eqref{sumbacktrackB} is $O(ae^{-C n})$. 

We now focus on the first integral in \eqref{sumbacktrackB}. Using Lemma \ref{doublepolelemma}, we deform the $w_1$ contour from $\Gamma_r$ to $\Gamma_{2/r}$. We pick up a simple pole at $w_1=-w_2$ and a second order pole at $w_1=w_2$. So the first integral in \eqref{sumbacktrackB} becomes the sum of three integrals which are 
\begin{align}
A_n=& \ -\frac{1}{2\pi i}\int_{\Gamma_{1/r}}dw_2\lim_{w_1\rarrow w_2}\frac{d}{dw_1}\frac{V_{1,1}(w_1,w_2)}{-w_1}\Big(\frac{w_1}{w_2}\Big)^{n/2}\Big(\frac{G(w_1)G(1/w_2)}{G(w_2)G(1/w_1)}\Big)^{n/2}\nonumber\\
&\quad\times\frac{G(w_1)G(1/w_2)(w_1-w_2)}{G(w_1)G(1/w_2)-G(w_2)G(1/w_1)}\Big(\frac{G(1/w_2)}{G(1/w_1)}\Big)^{(-|t'q_n|-t'q_n)/2}\Big(\frac{G(w_1)}{G(w_2)}\Big)^{(-|t'q_n|+t'q_n)/2}\nonumber\\
B_n=& \ -\frac{1}{2\pi i}\int_{\Gamma_{1/r}}dw_2\frac{V_{1,1}(-w_2,w_2)}{2w_2^2}(-1)^{|t'q_n|}\nonumber\\
&\quad\times\lim_{w_1\rarrow -w_2}\frac{G(w_1)G(1/w_2)(w_1+w_2)}{G(w_1)G(1/w_2)-G(w_2)G(1/w_1)},\nonumber
\end{align}
and
\begin{align}
R_n=& \ \frac{1}{(2\pi i)^2}\int_{\Gamma_{2/r}}\frac{dw_1}{w_1}\int_{\Gamma_{1/r}}dw_2\frac{V_{1,1}(w_1,w_2)}{w_2-w_1}\Big(\frac{w_1}{w_2}\Big)^{n/2}\Big(\frac{G(w_1)G(1/w_2)}{G(w_2)G(1/w_1)}\Big)^{n/2}\nonumber\\
&\quad\times\frac{G(w_1)G(1/w_2)}{G(w_1)G(1/w_2)-G(w_2)G(1/w_1)}\Big(\frac{G(1/w_2)}{G(1/w_1)}\Big)^{(-|t'q_n|-t'q_n)/2}\Big(\frac{G(w_1)}{G(w_2)}\Big)^{(-|t'q_n|+t'q_n)/2}\nonumber
\end{align}
$A_n$ comes from the double pole, $B_n$ comes from the single pole, and $R_n$ is the remaining integral.

We first bound $R_n$.  Recall $|V_{1,1}(w_1,w_2)|\leq Ca$ and the expansion given by \eqref{tempGoverGexpans}. We have $|w_1|/|w_2|=2$ so
\begin{align}
\Big(\frac{|w_1|}{|w_2|}\Big)^{n/2}=2^{n/2}, &&\Big(\frac{|G(w_1)G(1/w_2)|}{|G(w_2)G(1/w_1)|}\Big)^{n/2}=(\frac{|w_2^2|}{|w_1^2|}+aR_2(1/w_1,1/w_2,a))^{n/2}\leq \frac{1}{3^{n/2}}
\end{align}
and recall \eqref{temp34rloope}.
Putting these bounds together, we see the product of terms with an exponential dependence on $n$ in the integrand of $R_n$ is exponentially decaying in $n$. The remaining terms in the integrand of $R_n$ are bounded by expansions \eqref{tempGoverGexpans} and \eqref{temp23ewdsas}. Hence we have established that $R_n=O(ae^{-Cn})$ for some $C>0$. 

Next we bound $B_n$. This is simple, by Lemma \ref{doublepolelemma}, the expansions \eqref{temp23ewdsas} and \eqref{Gexpans} we have 
\begin{align}
\abs{\lim_{w_1\rarrow -w_2}\frac{G(w_1)G(1/w_2)(w_1+w_2)}{G(w_1)G(1/w_2)-G(w_2)G(1/w_1)}}\leq C.
\end{align}
Recalling $|V_{1,1}|\leq Ca$, we get $B_n = O(a)$.

All that remains is to show that $A_n=O(na)$, since then we will have proven \eqref{temp30f3l}.
Rewrite 
\begin{align}
A_n=\frac{1}{2\pi i}\int_{\Gamma_{1/r}}dw_2 \lim_{w_1\rarrow w_2}\frac{d}{dw_1}\frac{1}{w_1}V_{1,1}(w_1,w_2) f_1(w_1,w_2) f_2(w_1,w_2)
\end{align}
where 
\begin{align}
f_1(w_1,w_2)=\Big(\frac{w_1}{w_2}\frac{G(w_1)G(1/w_2)}{G(w_2)G(1/w_1)}\Big)^{n/2}\Big(\frac{G(1/w_2)}{G(1/w_1)}\Big)^{(-|t'q_n|-t'q_n)/2}\Big(\frac{G(w_1)}{G(w_2)}\Big)^{(-|t'q_n|+t'q_n)/2}\label{temp23bvcbned1}
\end{align}
contains the terms with an exponential dependence on $n$, and
\begin{align}
f_2(w_1,w_2)=\frac{G(w_1)G(1/w_2)(w_1-w_2)}{G(w_1)G(1/w_2)-G(w_2)G(1/w_1)}.\label{temp23bvcbned2}
\end{align}
From Lemma \ref{Veps1eps2formula}, we have $|V_{1,1}|\leq Ca$ and also $|\frac{d}{dw_1}V_{1,1}(w_1,w_2)|\leq Ca$. Recall that in the course of proving these two facts, we showed that $c^{-1}V_{1,1}(w_1,w_2)$ extends to an analytic function for $(a,w_1,w_2)\in B_{\eps}(0)\times (B_R(0)\setminus B_{1/R}(0))^2\subset \C^3$ (for $R>1$ fixed, large, and $\eps>0$ sufficiently small only depending on $R$). A similar argument shows that $(w_1+w_2)f_2(w_1,w_2)$ is analytic on the same set. This gives
\begin{align}
\abs{\lim_{w_1\rarrow w_2}V_{1,1}(w_1,w_2)f_2(w_1,w_2)}\leq Ca\label{tmep3rmhg|}
\end{align} 
and
\begin{align}
\abs{\lim_{w_1\rarrow w_2}\frac{d}{dw_1}\frac{1}{w_1}V_{1,1}(w_1,w_2)f_2(w_1,w_2)}\leq Ca.
\end{align} 
So we have
\begin{align}
A_n=\frac{1}{2\pi i}\int_{\Gamma_{1/r}}\frac{dw_2}{w_2} V_{1,1}(w_2,w_2)\Big(\lim_{w_1\rarrow w_2}\frac{d}{dw_1}f_1(w_1,w_2)\Big)f_2(w_2,w_2)+O(a).
\end{align}
We compute
\begin{align}
\lim_{w_1\rarrow w_2}\frac{d}{dw_1}f_1(w_1,w_2)=&\lim_{w_1\rarrow w_2}\frac{n}{2}\Big(\frac{w_1}{w_2}\frac{G(w_1)G(1/w_2)}{G(w_2)G(1/w_1)}\Big)^{n/2-1}\frac{\frac{d}{dw_1}w_1G(w_1)/G(1/w_1)}{w_2G(w_2)/G(1/w_2)}\label{f1deriv}\\
\times &\Big(\frac{G(1/w_2)}{G(1/w_1)}\Big)^{(-|t'q_n|-t'q_n)/2}\Big(\frac{G(w_1)}{G(w_2)}\Big)^{(-|t'q_n|+t'q_n)/2}\nonumber\\\nonumber+&\frac{t'q_n+|t'q_n|}{2}\Big(\frac{G(1/w_1)}{G(1/w_2)}\Big)^{(t'q_n+|t'q_n|)/2-1}\frac{\frac{d}{dw_1}G(1/w_1)}{G(1/w_2)}\\
\times&\Big(\frac{G(w_1)}{G(w_2)}\Big)^{(-|t'q_n|+t'q_n)/2}\Big(\frac{w_1}{w_2}\frac{G(w_1)G(1/w_2)}{G(w_2)G(1/w_1)}\Big)^{n/2}\nonumber\\
\nonumber+&\frac{-|t'q_n|+t'q_n}{2}\Big(\frac{G(w_1)}{G(w_2)}\Big)^{(-|t'q_n|+t'q_n)/2-1}\frac{\frac{d}{dw_1}G(w_1)}{G(w_2)}\\
\times&\Big(\frac{G(1/w_1)}{G(1/w_2)}\Big)^{(t'q_n+|t'q_n|)/2}\Big(\frac{w_1}{w_2}\frac{G(w_1)G(1/w_2)}{G(w_2)G(1/w_1)}\Big)^{n/2}\nonumber\\
=&  \ \frac{n}{2}\lim_{w_1\rarrow w_2}\frac{\frac{d}{dw_1}w_1G(w_1)/G(1/w_1)}{w_2G(w_2)/G(1/w_2)}+\frac{t'q_n+|t'q_n|}{2}\lim_{w_1\rarrow w_2}\frac{\frac{d}{dw_1}G(1/w_1)}{G(1/w_2)}\nonumber\\
&\quad \quad+\frac{-|t'q_n|+t'q_n}{2}\lim_{w_1\rarrow w_2}\frac{\frac{d}{dw_1}G(w_1)}{G(w_2)}.\nonumber
\end{align}
By defining the value of
\begin{align}
wG(w)/G(1/w)=-\frac{w}{2c}(1-\sqrt{1+2c/w^2})(1+\sqrt{1+2cw^2})
\end{align}
to be $-1/(2w)$ when $a=0$, we see it is has an analytic extension to $(a,w)\in B_{\eps}(0)\times B_R(0)\setminus B_{1/R}(0)\subset \C^2$. Hence it is bounded on this set and its derivative in $w$ is bounded on this set. We get
\begin{align}
\abs{\lim_{w_1\rarrow w_2}\frac{d}{dw_1}f_1(w_1,w_2)}\leq Cn.
\end{align}
Using this bound together with \eqref{tmep3rmhg|} we see that 
\begin{align}
A_n=O(na)
\end{align}
 and the proof is complete.
\end{proof}
\end{proposition}

\begin{proof}[Proof Proposition \ref{backtrackpropmain}]
Extend $\gamma_t'$ to the boundary of the lower left quadrant. Fix $t\in \R$, we have
\begin{align}
\P(\{\exists e\in \gamma_t' \text{ which is backtracking} \})\leq \sum_{\substack{e\in \gamma_t';\\e \text{  backtracking}}}\P(e\in \omega)\rarrow 0
\end{align}
 which follows from the proof above and Lemma \ref{B*smallmaindiag}.
Via an analogous argument in the other quadrants, one can get that the probability of seeing a backtracking dimer along the paths \eqref{antidiagpaths} goes to zero.
\end{proof}
\newpage
\section{The $a$ height function along $\partial S^T$}
\label{aheightsection}
In this section we establish asymptotics of the $a$-height function along the top line $\partial S^T$ of the box $S$ in order to help gain control over the whereabouts of the path $\Gamma_{m-1,m}$ for large $n$. For a fixed $t$ define the curve $\gamma_t$ which is the straight line parallel to $\vec{e}_1$, starting at a boundary face in the lower left quadrant and which ends at the face with position $(n(1+\xi_c)+1+\alpha p_n)\vec{e}_1+t'q_n\vec{e}_2$ on $\partial S^T$, where $t'=(2\floor{\frac{tq_n-1}{2}}+2)/q_n$. Label the height at the face $(n(1+\xi_c)+1+\alpha p_n)\vec{e}_1+t'q_n\vec{e}_2$ by $h(t')$, it is given by
\begin{align}
h(t')-(2|t'q_n|+1)=\sum_{e\in \gamma_{t'}}\sigma_e(4\ind_{e\in \omega}-1)
\end{align}
where $2|t'q_n|+1$ is the value of the height function at the end of $\gamma_t$ on the boundary and we defined the height as 1 at the face at $(0,0)$. The sum is over $a$-edges crossed by $\gamma_{t'}$ and $\sigma_e$ is $-1$ if the edge $e$ connects vertices of type $W_1,B_1$ (i.e. it is a backtracking edge in the lower left quadrant) and $1$ if $e$ connects vertices of type $W_0, B_0$.
\begin{proposition}\label{ExpectationProp}
For any fixed real $t$, 
\begin{align}
\E_{Az}[h(t')]= n+o(1)
\end{align}
as $n\rarrow \infty$.
\begin{proof}
Call the $a$-edges (or dimers) in the lower left quadrant that connect vertices of type $W_0,B_0$ forward edges (or dimers). We have
\begin{align}
\E_{Az}[h(t')-(2|t'q_n|+1)]&=\sum_{\substack{e\in \gamma_{t'};\\ e \text{ forward}}}(4\P(e\in \omega)-1)-\sum_{\substack{e\in \gamma_{t'};\\ e \text{ backtracking}}}(4\P(e\in \omega)-1)\\
&=4\sum_{\substack{e\in \gamma_{t'};\\ e \text{ forward}}}\P(e\in \omega)-1+o(1)\label{sumtemp12}
\end{align}
where in the second line we used the fact that the path $\gamma_{t'}$ crosses precisely one more forward edge than backtracking edges, and that the sum over backtracking edges along such a path goes to zero by Proposition \ref{backtrackingforheightfn}. From \eqref{corrfuncs} and Theorem \ref{ThmInvKast} the one point correlation function is given by
\begin{align}
\P(e\in \omega)=K_{a,1}(b,w)(\K_{1,1}^{-1}(w,b)-B_{\eps_1,\eps_2}(w,b)+B_{\eps_1,\eps_2}^*(w,b)).
\end{align}
Since $K_{a,1}(b,w)=ai$ for $e\in \gamma_{t'}$, $|\K_{1,1}^{-1}(w,b)|\leq Ca$ by Lemma \ref{boundK^-1} and $|B^*_{\eps_1,\eps_2}(w,b)|\leq C_1e^{-C_2n}$ by Lemma \ref{B*smallmaindiag}, we have that
\begin{align}
\P(e\in \omega)=-aiB_{\eps_1,\eps_2}(w,b)+O(a^2)+O(ae^{-c_1n}).
\end{align}
We index the forward edges $e\in \gamma_{t'}$ as follows; they have vertices $(y,x)$ in $B_0\times W_0$, such that
\begin{align}
x=(x_1,x_2)=(i-t'q_n+1,i+t'q_n)\\
y=(y_1,y_2)=(i-t'q_n,i+t'q_n+1)
\end{align}
where $2\Z \ni i\leq n(1+\xi_c)+\alpha p_n$ and $i\geq |t'q_n|$. The sum appearing in \eqref{sumtemp12} can be written
\begin{align}
\sum_{\substack{e\in \gamma_{t'};\\ e \text{ forward}}}\P(e\in \omega)=-ai\sum_{\substack{|t'q_n|\leq i\in 2\Z;\\i \leq n(1+\xi_c)+\alpha p_n}}B_{0,0}(a,x_1,x_2,y_1,y_2)+o(1)\label{temp133}
\end{align}
We substitute $x,y$ into \eqref{Beps1eps2} making use of  \eqref{Hfunc}, this yields \eqref{temp133} as
\begin{align}
-ai\sum_{\substack{|t'q_n|\leq i\in 2\Z;\\i \leq n(1+\xi_c)+\alpha p_n}}\frac{i^{-1}}{(2\pi i)^2}\int_{\Gamma_r}\frac{dw_1}{w_1}\int_{\Gamma_{1/r}}dw_2\frac{V_{0,0}(w_1,w_2)}{w_2-w_1}\Big(\frac{w_1}{w_2}\Big)^{n/2}\label{temp324}\\
\times\Big(\frac{G(w_1)G(1/w_1)}{G(w_2)G(1/w_2)}\Big)^{t'q_n/2}\Big(\frac{G(w_1)G(1/w_2)}{G(w_2)G(1/w_1)}\Big)^{(n-i)/2}\frac{1}{G(w_1)G(1/w_2)}.\nonumber
\end{align}
Change variables $i'=i-|t'q_n|$
\begin{align}
&-a\sum_{\substack{0\leq i'\in \Z;\\i' \leq (n(1+\xi_c)+\alpha p_n-|t'q_n|)/2}}\frac{1}{(2\pi i)^2}\int_{\Gamma_r}\frac{dw_1}{w_1}\int_{\Gamma_{1/r}}dw_2\frac{V_{0,0}(w_1,w_2)}{w_2-w_1}\Big(\frac{w_1}{w_2}\Big)^{n/2}\frac{1}{G(w_1)G(1/w_2)}\label{temp540}\\
&\quad\times\Big(\frac{G(w_1)G(1/w_2)}{G(w_2)G(1/w_1)}\Big)^{n/2}\Big(\frac{G(w_2)G(1/w_1)}{G(w_1)G(1/w_2)}\Big)^{i}\Big(\frac{G(1/w_2)}{G(1/w_1)}\Big)^{(-|t'q_n|-t'q_n)/2}\Big(\frac{G(w_1)}{G(w_2)}\Big)^{(-|t'q_n|+t'q_n)/2}\nonumber\\
 & \ =\frac{a}{(2\pi i)^2}\int_{\Gamma_r}\frac{dw_1}{w_1}\int_{\Gamma_{1/r}}dw_2\frac{V_{0,0}(w_1,w_2)}{w_2-w_1}\Big(\frac{w_1}{w_2}\Big)^{n/2}\Big(\frac{G(w_1)G(1/w_2)}{G(w_2)G(1/w_1)}\Big)^{n/2}\nonumber\\
&\quad\times\Big(1-\Big(\frac{G(w_2)G(1/w_1)}{G(w_1)G(1/w_2)}\Big)^{(n(1+\xi_c)+\alpha p_n-|t'q_n|)/2+1}\Big)\nonumber\\&\quad\times\frac{1}{G(w_2)G(1/w_1)-G(w_1)G(1/w_2)}\Big(\frac{G(1/w_2)}{G(1/w_1)}\Big)^{(-|t'q_n|-t'q_n)/2}\Big(\frac{G(w_1)}{G(w_2)}\Big)^{(-|t'q_n|+t'q_n)/2}\nonumber\\
& \ =\frac{a}{(2\pi i)^2}\int_{\Gamma_r}\frac{dw_1}{w_1}\int_{\Gamma_{1/r}}dw_2\frac{V_{0,0}(w_1,w_2)}{w_2-w_1}\Big(\frac{w_1}{w_2}\Big)^{n/2}\Big(\frac{G(w_1)G(1/w_2)}{G(w_2)G(1/w_1)}\Big)^{n/2}\nonumber\\
&\quad\times\frac{1}{G(w_2)G(1/w_1)-G(w_1)G(1/w_2)}\Big(\frac{G(1/w_2)}{G(1/w_1)}\Big)^{(-|t'q_n|-t'q_n)/2}\Big(\frac{G(w_1)}{G(w_2)}\Big)^{(-|t'q_n|+t'q_n)/2}\nonumber\\
& \quad - \frac{a}{(2\pi i)^2}\int_{\Gamma_r}\frac{dw_1}{w_1}\int_{\Gamma_{1/r}}dw_2\frac{V_{0,0}(w_1,w_2)}{w_2-w_1}\Big(\frac{w_1}{w_2}\Big)^{n/2}\Big(\frac{G(w_2)G(1/w_1)}{G(w_1)G(1/w_2)}\Big)^{(n\xi_c+\alpha p_n-|t'q_n|)/2+1}\nonumber\\&\quad\times\frac{1}{G(w_2)G(1/w_1)-G(w_1)G(1/w_2)}\Big(\frac{G(1/w_2)}{G(1/w_1)}\Big)^{(-|t'q_n|-t'q_n)/2}\Big(\frac{G(w_1)}{G(w_2)}\Big)^{(-|t'q_n|+t'q_n)/2}\nonumber
\end{align}
where we used the geometric sum formula in the first equality. Denote the last double integral by $X_n$, we will bound this at the end of the proof. We focus on the first double integral instead (which appears in the last line of \eqref{temp540}).
Lemma \ref{doublepolelemma} allows us to deform the $w_1$-contour from $\Gamma_{r}$ to $\Gamma_{2/r}$. Just as in Proposition \ref{backtrackingforheightfn}, we cross over two poles, a single pole at $w_1=-w_2$ and a double pole at the point $w_1=w_2$. The residue theorem then gives this double integral as the sum of
\begin{align}
&A_n'= -\frac{a}{2\pi i}\int_{\Gamma_{1/r}}dw_2\lim_{w_1\rarrow w_2}\frac{d}{dw_1}\frac{V_{0,0}(w_1,w_2)}{-w_1}\Big(\frac{w_1}{w_2}\Big)^{n/2}\Big(\frac{G(w_1)G(1/w_2)}{G(w_2)G(1/w_1)}\Big)^{n/2}\label{temp34ew}\\
&\quad\times\frac{w_1-w_2}{G(w_2)G(1/w_1)-G(w_1)G(1/w_2)}\Big(\frac{G(1/w_2)}{G(1/w_1)}\Big)^{(-|t'q_n|-t'q_n)/2}\Big(\frac{G(w_1)}{G(w_2)}\Big)^{(-|t'q_n|+t'q_n)/2}\nonumber\\
&B_n'=-\frac{a}{2\pi i}\int_{\Gamma_{1/r}}dw_2\frac{V_{0,0}(-w_2,w_2)}{2w_2^2}(-1)^{-|t'q_n|}\lim_{w_1\rarrow -w_2}\frac{w_1+w_2}{G(w_2)G(1/w_1)-G(w_1)G(1/w_2)}\nonumber\\
&R_n'=\frac{a}{(2\pi i)^2}\int_{\Gamma_{2/r}}\frac{dw_1}{w_1}\int_{\Gamma_{1/r}}dw_2\frac{V_{0,0}(w_1,w_2)}{w_2-w_1}\Big(\frac{w_1}{w_2}\Big)^{n/2}\Big(\frac{G(w_1)G(1/w_2)}{G(w_2)G(1/w_1)}\Big)^{n/2}\nonumber\\
&\quad\times\frac{1}{G(w_2)G(1/w_1)-G(w_1)G(1/w_2)}\Big(\frac{G(1/w_2)}{G(1/w_1)}\Big)^{(-|t'q_n|-t'q_n)/2}\Big(\frac{G(w_1)}{G(w_2)}\Big)^{(-|t'q_n|+t'q_n)/2}.\nonumber
\end{align}
Hence we have 
\begin{align}
\sum_{\substack{e\in \gamma_{t'};\\ e \text{ forward}}}\P(e\in \omega)=A_n'+B_n'+R_n'+X_n+o(1).\label{tempsummary5ref}
\end{align}
Observe that via a similar argument that we used to get $R_n=O(e^{-Cn})$ in the proof of Proposition \ref{backtrackingforheightfn}, we have that $R_n'=O(e^{-Cn})$. The first and second integral in \eqref{temp34ew} (which came from the double pole at $w_1=w_2$ and simple pole at $w_1=-w_2$) contains the main contribution to the asymptotics of the sum \eqref{temp540}. The following arguments contain expansions of functions of a similar type to expansions already computed in this article. We compute that asymptotics of $B_n'$ first.
From Lemma \ref{Veps1eps2formula}, we have 
\begin{align}
V_{0,0}(-w_2,w_2)=\frac{1}{2}+aR_1(w_2,a).
\end{align}
Using \eqref{temp23ewdsas} we compute
\begin{align}
\lim_{w_1\rarrow -w_2}\frac{w_1+w_2}{G(w_2)G(1/w_1)-G(w_1)G(1/w_2)}=\frac{w_2}{a}+R_2(w_2,a).
\end{align}
We see 
\begin{align}
B_n'=-\frac{1}{2\pi i}\int_{\Gamma_{1/r}}\frac{1}{4w_2}+O(a)=-\frac{1}{4}+O(a).\label{Bn'asymp}
\end{align}
Now we turn to $A_n'$, we rewrite it as
\begin{align}
-\frac{a}{2\pi i}\int_{\Gamma_{1/r}}dw_2 \lim_{w_1\rarrow w_2}\frac{d}{dw_1}\frac{V_{0,0}(w_1,w_2)}{w_1}f_1(w_1,w_2)\tilde{f}_2(w_1,w_2)\label{tempprodr}
\end{align}
where we recall $f_1,f_2$ as given in \eqref{temp23bvcbned1}, \eqref{temp23bvcbned2} and
\begin{align}
\tilde{f}_2(w_1,w_2)=\frac{f_2(w_1,w_2)}{G(w_1)G(1/w_2)}.
\end{align}
Note that a minus sign in the integrand in $A_n'$ was absorbed by $f_2(w_1,w_2)$.
We use the product rule in the integrand of \eqref{tempprodr}. Since $f_1(w_2,w_2)=1$, we get
\begin{align}
\label{tempAn''}A_n'=-\frac{a}{2\pi i}\int_{\Gamma_{1/r}}\frac{dw_2}{w_2} V_{0,0}(w_2,w_2)\Big(\lim_{w_1\rarrow w_2}\frac{d}{dw_1}f_1(w_1,w_2)\Big)\tilde{f}_2(w_2,w_2)\\-\frac{a}{2\pi i}\int_{\Gamma_{1/r}}dw_2 \lim_{w_1\rarrow w_2}\frac{d}{dw_1}\frac{V_{0,0}(w_1,w_2)}{w_1}\tilde{f}_2(w_1,w_2).\nonumber
\end{align}
From Lemma \ref{Veps1eps2formula},
\begin{align}\label{Vscsexpans}
V_{0,0}(w_2,w_2)=-\frac{1}{2}+a\frac{(2w_2^2)(1+w_2^4)}{4w_2^4}+a^2R(w_2,w_2,a),
\end{align}
and 
\begin{align}
\abs{\frac{d}{dw_1}\frac{V_{0,0}(w_1,w_2)}{w_1}}\leq Ca.
\end{align}
Using \eqref{temp23ewdsas} one can show
\begin{align}\label{f2expans}
\tilde{f}_2(w_1,w_2)=-2\frac{w_1w_2}{(w_1+w_2)a}-\frac{(w_1^2+w_2^2)(1+w_1^2w_2^2)}{w_1w_2(w_1+w_2)}+a R_3(w_1,w_2,a)
\end{align}
and so
\begin{align}
\lim_{w_1\rarrow w_2}\frac{d}{dw_1}\tilde{f}_2(w_1,w_2)=\frac{1}{2a}+R_1(w_2,a).
\end{align}
We can also use this to get
\begin{align}
\abs{\tilde{f}_2(w_1,w_2)\frac{d}{dw_1}\frac{V_{0,0}(w_1,w_2)}{w_1}}\leq Ca.
\end{align}
Hence the second term in \eqref{tempAn''} is
\begin{align}
&-\frac{a}{2\pi i}\int_{\Gamma_{1/r}}dw_2 \frac{V_{0,0}(w_2,w_2)}{w_2}\lim_{w_1\rarrow w_2}\frac{d}{dw_1}\tilde{f}_2(w_1,w_2) +O(a)\\
\nonumber & \ = \frac{1}{4}+O(a).
\end{align}
Recall we computed the derivative of $f_1$ in \eqref{f1deriv}. We expand the functions appearing in the derivative of $f_1$ up to second order in $a$,
\begin{align}
\lim_{w_1\rarrow w_2}\frac{\frac{d}{dw_1}w_1G(w_1)/G(1/w_1)}{w_2G(w_2)/G(1/w_2)}&=\lim_{w_1\rarrow w_2} -\frac{w_2}{w_1^2}+a\frac{-3w_1^2+w_2^2}{2w_2}+a^2R(w_1,w_2,a)\\
&=-\frac{1}{w_2}+a\frac{1+w_2^4}{w_2^3}+a^2R(w_2,w_2,a),\nonumber\\
\lim_{w_1\rarrow w_2}\frac{\frac{d}{dw_1}G(1/w_1)}{G(1/w_2)}&=\lim_{w_1\rarrow w_2}\frac{1}{w_2}+a\frac{-3w_1^2+w_2^2}{2w_2}+a^2R_4(w_1,w_2,a)\nonumber\\
\nonumber &= \frac{1}{w_2}-aw_2+a^2R_4(w_2,w_2,a),\\
\lim_{w_1\rarrow w_2}\frac{\frac{d}{dw_1}G(w_1)}{G(w_2)}&=\lim_{w_1\rarrow w_2} -\frac{w_2}{w_1^2}-a\frac{w_1^2-3w_2^2}{2w_1^4 w_2}+a^2R_5(w_1,w_2,a)\nonumber\\
&=-\frac{1}{w_2}+a\frac{1}{w_2^3}+a^2R_5(w_1,w_2,a).\nonumber
\end{align}
Substituting these expansions, we have
\begin{align}
\frac{d}{dw_1}f_1(w_1,w_2)&=\frac{n}{2}\Big(-\frac{1}{w_2}+\Big(\frac{1}{w_2^3}+w_2\Big)a\Big)+\frac{t'q_n+|t'q_n|}{2}\Big(\frac{1}{w_2}-w_2a\Big)\label{f1expans}\\
&+\frac{t'q_n-|t'q_n|}{2}\Big(-\frac{1}{w_2}+\frac{a}{w_2^3}\Big)+na^2R_5(w_2,a).\nonumber
\end{align}
Substitute the expansions \eqref{f1expans}, \eqref{f2expans} and \eqref{Vscsexpans} into the first integral of $A_n'$ in \eqref{tempAn''}. Expanding brackets we get
\begin{align}
A_n'&=-\frac{a}{2\pi i}\int_{\Gamma_{1/r}}dw_2\Big(\frac{-n+2|t'q_n|}{4w_2 a}+\frac{n(1+w_2^4)+t'q_n+|t'q_n|+w_2^4(-t'q_n+|t'q_n|)}{4w_2^3}\Big)\\&\quad\quad+\frac{1}{4}+O(na^2)+O(a)\nonumber\\
&=\frac{n-2|t'q_n|}{4} +\frac{1}{4}+O(na^2)+O(a).\nonumber
\end{align}
Hence with this result, \eqref{Bn'asymp}, \eqref{tempsummary5ref} and \eqref{sumtemp12} we have established
\begin{align}
\E_{Az}[h(t')]&=4\sum_{\substack{e\in \gamma_{t'};\\ e \text{ forward}}}\P(e\in \omega)-1+2|t'q_n|+1+o(1)\\
&=4(A_n'+B_n'+R_n'+X_n+o(1))+2|t'q_n|+o(1)\nonumber\\
&=4\Big(\frac{n-2|t'q_n|}{4} +\frac{1}{4}-\frac{1}{4}+(na^2)+O(a)+O(e^{-Cn})+X_n+o(1)\Big)+2|t'q_n|\nonumber\\
&=n+4X_n+O(e^{-Cn})+o(1)+O(na^2)+O(a)\nonumber\\
&=n+4X_n+o(1).\nonumber
\end{align}
It remains to show that $X_n=o(1)$. We rewrite $X_n$ as
\begin{align}
-\frac{a}{(2\pi i)^2}\int_{\Gamma_r}dw_1\int_{\Gamma_{1/r}}dw_2\frac{1}{(w_2-w_1)^2}\frac{V_{0,0}(w_1,w_2)}{w_1}\frac{H_{x_1'+1,x_2'}(w_1)}{H_{y_1',y_2'+1}(w_2)}\\\times\frac{(w_2-w_1)G(w_1^{-1})G(w_2)}{G(w_2)G(1/w_1)-G(w_1)G(1/w_2)}\label{Xnasymp}\nonumber
\end{align}
where $(x_1',x_2')=\tilde{x_j}$, $(y_1',y_2')=\tilde{y}_i$ are as in assumption \ref{assumptionairynotationcoords} with $\alpha_j=\alpha_i=\alpha $, $\beta_j=\beta_i=t'$, $\eps_1=\eps_2=0$. Note that
\begin{align}
\frac{(w_2-w_1)(w_1+w_2)G(w_1^{-1})G(w_2)}{G(w_2)G(1/w_1)-G(w_1)G(1/w_2)}
\end{align}
is analytic on $r\leq |w_1|,|w_2|\leq 3/r$, by applying lemma \ref{doublepolelemma} in the $w_1$, $w_2$ variables separately, it is also bounded for $a>0$ small.
Observe the structural similarity between $X_n$ and $B_{\eps_1,\eps_2}(\tilde{x}_j,\tilde{y}_i)$ in \eqref{Beps1eps2}, in particular observe that $X_n$ has a double pole coming from $(w_1-w_2)^2$, and single pole $(w_1+w_2)$ instead of just a single pole coming from $(w_1-w_2)$ in $B_{\eps_1,\eps_2}(\tilde{x}_j,\tilde{y}_i)$. Hence we expect a bound which exponentially decays in $\alpha $, since the Airy kernel exponentially decays for large, positive values. We focus on the $X_n$ at a neighbourhood of $(w_1,w_2)=(i,i)$, since by similar arguments to the proof of Proposition \ref{Btildeairy}, the rest of the contour has a lower order contribution. We recall the definition of the contours $desc_{1/R}^\eps$ and $asc_R^\eps$. These are straight lines in $B_\eps(i)$. The two straight lines $asc_R^\eps\cap B_\eps(i)$ stem from the two points $w_2^*=i+\eps e^{i\theta}$ down to $i$, hence they intersect the smaller circle $B_{\eps/(an)^{1/3}}(i)$ at two points $w_2^{**}$. Define a new contour $asc_{R}^{\eps,n}$ which is $asc_R^{\eps,n}$ outside of $B_{\eps/(an)^{1/3}}(i)$ and consists of the smaller section of $\partial B_{\eps/(an)^{1/3}}(i)$ between the points $w_2^{**}$. Define the contour $desc_{1/R}^{\eps,n}$ in an equivalent fashion. Deform the $w_1,w_2$ contours in $X_n$ to $desc_{1/R}^{\eps,n}$, $asc_{R}^{\eps,n}$ (and their reflections). Observe that now $|w_1-w_2|\geq C/(an)^{1/3}$. We set $w_1'=i+w_1(an)^{-1/3}, w_2'=i+w_2(an)^{-1/3}$ and use Lemmas \ref{labelH(i)simple}, \ref{saddlepointexpprop1} to get an upper bound on the section of the integral in $X_n$ over $desc_{1/R}^{\eps,n}\cap B_{\eps}(i)\times asc_{R}^{\eps,n}\cap B_{\eps}(i)$ given by 
\begin{align}
&\frac{C'}{4\pi^2}\int_{C_1'}|dw_1|\int_{C_2'}|dw_2| \exp\big(\mcR[-iw_1\alpha +iw_2\alpha -w_1^2t'+w_2^2t'-\frac{i}{3}w_1^3+\frac{i}{3}w_2^3]\\&\quad+C\frac{(|w_1|+|w_2|+|w_1|^4+|w_2|^4) }{(an)^{1/3}}+C''a^2n\big)\nonumber
\end{align}
where $C_1'$ and $C_2'$ are the images of $desc_{1/R}^{\eps,n}\cap B_\eps(i)$ and $asc_R^{\eps,n}\cap B_\eps(i)$ under the map $w\mapsto (an)^{1/3}(w-i)$, respectively. The sections of $C_1'$ and $C_2'$ that consist of straight lines heading out to infinity can be parametrised in the form $w_1=re^{i\theta}$ for $r\in [\eps,\eps (an)^{1/3}]$, $\theta$ fixed in $[-\pi/6-\delta,-\pi/6+\delta]\cup [-5\pi/6-\delta,-5\pi/6+\delta]$ and $w_2=se^{i\varphi}$ for $s\in [\eps,\eps (an)^{1/3}]$, $\varphi$ fixed in $[\pi/6-\delta,\pi/6+\delta]\cup [5\pi/6-\delta,5\pi/6+\delta]$ where $\delta=\pi/18$. As in \eqref{tempmnnmnm}, we use the inequalities \eqref{epssmallcond}, \eqref{deltasmallcond} to see that these sections are bounded above by integrals of the form
\begin{align}
\int_\eps^\infty dr \int_\eps^\infty ds\exp\big(r(\alpha \sin\theta+\frac{C}{(an)^{1/3}})-s(\alpha \sin\varphi-\frac{C}{(an)^{1/3}})\\-r^2\cos(2\theta)t'+s^2\cos(2\varphi)t'-\frac{r^3}{12}-\frac{s^3}{12}\big).\nonumber
\end{align}
Because $r,s>\eps$ and $\sin\theta<\sin(-\pi/6+\pi/18)<0$ and $\sin\varphi>\sin(\pi/6-\pi/18)>0$ in the above integral is $O(e^{-\sin(\pi/6-\pi/18)\eps\alpha })$. The sections of $C_1'$ and $C_2'$ over $\partial B_\eps(0)$ (the image of $\partial B_{\eps/(an)^{1/3}}(i)$) can be parametrised in the form $w_1=re^{i\theta}$ for $r=\eps$, $\theta$ in a subset of $[-5\pi/6-\delta,-\pi/6+\delta]$ and $w_2=se^{i\varphi}$ for $s=\eps$, $\varphi$ in a subset of $[\pi/6,5\pi/6+\delta]$. Simply observe that again, $\sin\theta<\sin(-\pi/6+\pi/18)<0$ and $\sin\varphi>\sin(\pi/6-\pi/18)>0$ to achieve the same bound. We conclude that 
\begin{align}
|X_n|\leq C\exp(-\sin(\pi/6-\pi/18)\eps\alpha )=C\exp(-c_1\alpha )\rarrow 0.\label{temp607oy}
\end{align}
This concludes the proof.
\end{proof}
\end{proposition}
Now we look to show that the variance of the height function along the top boundary $\partial S^T$ of $S$ tends to zero. 
\begin{proposition}\label{vargoestozero}
For any fixed real $t$, 
\begin{align}
\var_{Az}[h(t')]\rarrow 0
\end{align}
as $n\rarrow \infty$.
\begin{proof}
Recall that the path $\gamma_t$ starts at the lower left boundary and ends on $\partial S^T$. We extend $\gamma_t$ by a line segment $x\vec{e}_1, x\in[0,1]$, so that the path ends at $\partial S^T+\vec{e}_1$ and call this extension $\gamma_t^*$. Observe that this extension crosses an equal number of forward and backtracking dimers, this helps with the coming expressions. Define the height at the end of this extended path 
\begin{align}
h^*(t')&=\sum_{e\in \gamma_{t'}^*}\sigma_e(4\ind_{e\in \omega}-1)+(2|t'q_n|+1)\\
&=h(t')-4(\ind_e^*-1)\nonumber
\end{align}
where $e^*$ is the backtracking edge crossed by the short line segment $\gamma_t^*\setminus \gamma_t$.
We prove that $\var_{Az}(h^*(t'))\rarrow 0$, from which a similar argument to the following then shows that $\var_{Az}(h(t'))\rarrow 0$. We do this as it makes some expressions neater. Take two copies $\gamma_t^*(1),$ $\gamma_t^*(2)$ of $\gamma_t^*$ and introduce the notation $e=(x_e,y_e)$. Then
\begin{align}
\var_{Az}(h^*(t'))&=16\sum_{e\in \gamma_t^*(1)}\sum_{e'\in \gamma_t^*(2)}\sigma_e\sigma_{e'}(\P(e,e'\in\omega)-\P(e\in \omega)\P(e'\in\omega))\\
&=16\sum_{e\in \gamma_t^*(1)}\sum_{e'\in \gamma_t^*(2)}\sigma_e\sigma_{e'}K_{a,1}(e)K_{a,1}(e')K_{a,1}^{-1}(x_e,y_{e'})K_{a,1}^{-1}(x_{e'},y_{e}).\nonumber
\end{align}
The path $\gamma_t^*$ crosses edges $e\in W_{\eps}\times B_{\eps}$ which alternate between type $\eps=0$ and $\eps=1$ edges, i.e. forward and backtracking edges, respectively. We index these edges as follows
\begin{align}
x_{e_\eps}&=(1+i)\vec{e}_1+t'q_n\vec{e}_2+(0,2\eps-1),\\
y_{e_\eps}&=(1+i)\vec{e}_1+t'q_n\vec{e}_2+(2\eps-1,0)
\end{align}
where $i \in 2\Z$ runs from $|t'q_n|$ to $n(1+\xi_c)+\alpha p_n$, and $\eps\in \{0,1\}$. Similarly,
\begin{align}
x_{e_{\eps'}'}&=(1+j)\vec{e}_1+t'q_n\vec{e}_2+(0,2\eps'-1),\\
y_{e_{\eps'}'}&=(1+j)\vec{e}_1+t'q_n\vec{e}_2+(2\eps'-1,0)
\end{align}
where $j \in 2\Z$ runs from $|t'q_n|$ to $n(1+\xi_c)+\alpha p_n$, and $\eps'\in \{0,1\}$. 

Introduce the notation $x_{e_{\eps'}}:=x_{e_{\eps'}'}$, $y_{e_{\eps'}}:=y_{e_{\eps'}'}$, i.e. we consider $e_{\eps'}$ as labelling the edge $e_{\eps'}'$.
We have $K_{a,1}(e)=ai$, recall Lemmas \ref{boundK^-1}, \ref{B*smallmaindiag}, we also have $\sigma_{e_\eps}\sigma_{e_{\eps'}}=(-1)^{\eps+\eps'}$ so we get
\begin{align}
\var_{Az}(h^*(t'))=-16a^2\sum_{\eps,\eps'\in\{0,1\}}(-1)^{\eps+\eps'}\sum_{i,j}\Big(\K_{1,1}(x_{e_{\eps}},y_{e_{\eps'}})\K_{1,1}(x_{e_{\eps'}},y_{e_{\eps}})\\-B_{\eps,\eps'}(x_{e_\eps},y_{e_{\eps'}})\K_{1,1}(x_{e_{\eps'}},y_{e_{\eps}})-B_{\eps',\eps}(x_{e_{\eps'}},y_{e_{\eps}})\K_{1,1}(x_{e_{\eps}},y_{e_{\eps'}})\\
B_{\eps,\eps'}(x_{e_\eps},y_{e_{\eps'}})B_{\eps',\eps}(x_{e_{\eps'}},y_{e_{\eps}})\Big)+o(1).\label{temphorithj1}
\end{align}
Label the coordinates of the vertices 
\begin{align}
&x_{e_{\eps}}=(x_{e_{\eps}}(1),x_{e_{\eps}}(2)),&&y_{e_{\eps}}=(y_{e_{\eps}}(1),y_{e_{\eps}}(2)),\\
&x_{e_{\eps'}}=(x_{e_{\eps'}}(1),x_{e_{\eps'}'}(2)), &&y_{e_{\eps'}}=(y_{e_{\eps'}}(1),y_{e_{\eps'}}(2)).
\end{align}
We first focus on showing 
\begin{align}
a^2\sum_{i,j}B_{\eps,\eps'}(x_{e_\eps},y_{e_{\eps'}})B_{\eps',\eps}(x_{e_{\eps'}},y_{e_{\eps}})\rarrow 0\label{temphorithj}
\end{align}
for $\eps,\eps'\in\{0,1\}$.
We use \eqref{Beps1eps2} in \eqref{temphorithj}, we see the following factor simplifies upon substitution
\begin{align}
i^{(x_{e_{\eps}}(2)-x_{e_{\eps}}(1)+x_{e_{\eps'}}(2)-x_{e_{\eps'}}(1)+y_{e_{\eps}}(1)-y_{e_{\eps}}(2)+y_{e_{\eps'}}(1)-y_{e_{\eps'}}(2))/2}=(-1)^{\eps+\eps'-1}.
\end{align}
So we have
\begin{align}
\label{quadbeps}B_{\eps,\eps'}(x_{e_\eps},y_{e_{\eps'}})B_{\eps',\eps}(x_{e_{\eps'}},y_{e_{\eps}})=\frac{(-1)^{\eps+\eps'+1}}{(2\pi i)^4}\int_{\Gamma_r}\frac{dw_1}{w_1}\int_{\Gamma_{r'}}\frac{dw_1'}{w_1'}\int_{\Gamma_{1/r}}dw_2\int_{\Gamma_{1/{r'}}}dw_2'\\\times\frac{V_{\eps,\eps'}(w_1,w_2)}{w_2-w_1}\frac{V_{\eps',\eps}(w_1',w_2')}{w_2'-w_1'}\frac{H_{x_{e_{\eps}}(1)+1,x_{e_{\eps}}(2)}(w_1)}{H_{y_{e_{\eps'}}(1),y_{e_{\eps'}}(2)+1}(w_2)}\frac{H_{x_{e_{\eps'}}(1)+1,x_{e_{\eps'}}(2)}(w_1')}{H_{y_{e_{\eps}}(1),y_{e_{\eps}}(2)+1}(w_2')}\nonumber
\end{align}
where we deformed the $w_1'$ contour from $\Gamma_r$ to $\Gamma_{r'}$, for an $r'$ such that $r'/2<r<r'<1$ and the $w_2'$ contour from $\Gamma_{1/r}$ to $\Gamma_{1/r'}$, where we note $1<1/r'<1/r<2/r'$.
By substitution in \eqref{Hfunc} we have 
\begin{align}
&\frac{H_{x_{e_{\eps}}(1)+1,x_{e_{\eps}}(2)}(w_1)}{H_{y_{e_{\eps}}(1),y_{e_{\eps}}(2)+1}(w_2')}\\&=\Big(\frac{w_1G(w_1)G(w_2'^{-1})}{w_2'G(w_1^{-1})G(w_2')}\Big)^{n/2}\Big(\frac{G(w_1^{-1})G(w_2')}{G(w_1)G(w_2'^{-1})}\Big)^{i/2}\Big(\frac{G(w_1)G(w_1^{-1})}{G(w_2')G(w_2'^{-1})}\Big)^{t'q_n/2}\frac{\big(G(w_1^{-1})G(w_2')\big)^{\eps}}{G(w_1)G(w_2'^{-1})}\nonumber
\end{align}
and a similar formula for the remaining quotient of $H$ functions. Hence we see we want to use the geometric sum formula on
\begin{align}
\sum_{i,j}\Big(\frac{G(w_1^{-1})G(w_2')}{G(w_1)G(w_2'^{-1})}\Big)^{i/2}\Big(\frac{G(w_1'^{-1})G(w_2)}{G(w_1')G(w_2^{-1})}\Big)^{j/2}.
\end{align}
Make the change of variables $i'=\frac{i-|t'q_n|}{2}$, $j'=\frac{j-|t'q_n|}{2}$. We then use the geometric sum formula in $i'$ and $j'$ separately, we see the appearance of 
\begin{align}
\frac{1-\Big(\frac{G(w_1^{-1})G(w_2')}{G(w_1)G(w_2'^{-1})}\Big)^{L}}{1-\Big(\frac{G(w_1^{-1})G(w_2')}{G(w_1)G(w_2'^{-1})}\Big)}\times\frac{1-\Big(\frac{G(w_1'^{-1})G(w_2)}{G(w_1')G(w_2^{-1})}\Big)^{L}}{1-\Big(\frac{G(w_1'^{-1})G(w_2)}{G(w_1')G(w_2^{-1})}\Big)}\label{geomsumresult}
\end{align}
where $L:=(n(1+\xi_c)+\alpha p_n-|t'q_n|)/2+1$. Expanding \eqref{geomsumresult} and distributing the quadruple contour integral across the four resulting factors, we see that contour integral coming from the term
\begin{align}
\frac{\Big(\frac{G(w_1^{-1})G(w_2')}{G(w_1)G(w_2'^{-1})}\Big)^{L}}{1-\Big(\frac{G(w_1^{-1})G(w_2')}{G(w_1)G(w_2'^{-1})}\Big)}\times\frac{\Big(\frac{G(w_1'^{-1})G(w_2)}{G(w_1')G(w_2^{-1})}\Big)^{L}}{1-\Big(\frac{G(w_1'^{-1})G(w_2)}{G(w_1')G(w_2^{-1})}\Big)}
\end{align}
is $O(e^{-c_1\alpha })$ via an application of the argument we used to show $X_n=O(e^{-c_1\alpha })$ in the proof of Proposition \ref{ExpectationProp}, in both sets of variables $(w_1,w_2)$, $(w_1',w_2')$. We focus on the quadruple contour integral coming from 
\begin{align}
\frac{1}{1-\Big(\frac{G(w_1^{-1})G(w_2')}{G(w_1)G(w_2'^{-1})}\Big)}\times\frac{1}{1-\Big(\frac{G(w_1'^{-1})G(w_2)}{G(w_1')G(w_2^{-1})}\Big)}.
\end{align}
Explicitly, this contour integral is 
\begin{align}\frac{(-1)^{\eps+\eps'+1}}{(2\pi i)^4}\int_{\Gamma_r}\frac{dw_1}{w_1}\int_{\Gamma_{1/r}}dw_2\int_{\Gamma_{r'}}\frac{dw_1'}{w_1'}\int_{\Gamma_{1/r'}}dw_2' \frac{\mathcal{V}\times\mathcal{A}\times\mathcal{B}}{(w_2-w_1)(w_2'-w_1)(w_2'-w_1')(w_2'+w_1)(w_2+w_1')}\label{1term6poles}
\end{align}
where 
\begin{align}
\mathcal{V}:=V_{\eps,\eps'}(w_1,w_2)V_{\eps',\eps}(w_1',w_2'),
\end{align}
\begin{align}
\mathcal{A}:=&\Big(\frac{w_1w_1'G(w_1)G(w_1')G(w_2^{-1})G(w_2'^{-1})}{w_2w_2'G(w_1^{-1})G(w_1'^{-1})G(w_2)G(w_2')}\Big)^{n/2}\\&\times\Big(\frac{G(w_1)G(w_1')G(w_1^{-1})G(w_1'^{-1})}{G(w_2)G(w_2')G(w_2^{-1})G(w_2'^{-1})}\Big)^{t'q_n/2}\nonumber\\
&\times\Big(\frac{G(w_1^{-1})G(w_1'^{-1})G(w_2)G(w_2')}{G(w_1)G(w_1')G(w_2^{-1})G(w_2'^{-1})}\Big)^{|t'q_n|/2}\nonumber
\end{align}
and
\begin{align}
\mathcal{B}:=\frac{(G(w_1^{-1})G(w_2'))^\eps(w_2'-w_1)(w_2'+w_1)}{G(w_1)G(w_2'^{-1})-G(w_1^{-1})G(w_2')}\times\frac{(G((w_1'^{-1})G(w_2))^{\eps'}(w_2-w_1')(w_2+w_1')}{G(w_1')G(w_2^{-1})-G(w_1'^{-1})G(w_2)}.
\end{align}
Observe that by Lemma \ref{doublepolelemma}, $\mathcal{B}$ is analytic in an annulus $B_R(0)\setminus B_{1/R}(0)$ in each variable $w_1,w_1',w_2,w_2'$, for each $a>0$ small. Hence we see six simple poles in the integrand of \eqref{1term6poles}.  Similar to the proofs of Propositions \ref{backtrackingforheightfn} and \ref{ExpectationProp}, we want to deform the $w_1,w_1'$ contours over these poles. To do so, we introduce the following notation, let
\begin{align}
\mathcal{G}:=\frac{1}{w_1w_1'}\frac{\mathcal{V}\times\mathcal{A}\times\mathcal{B}}{(w_2-w_1)(w_2'-w_1)(w_2'-w_1')(w_2'+w_1)(w_2+w_1')},
\end{align}
and define
\begin{align}
&\mathcal{G}\substack{w\rarrow z\\\cdot}:=(2\pi i )\lim_{w\rarrow z}-(w-z) \ \mathcal{G}, \label{temp3rgfdann}\\&\mathcal{G}\substack{w\rarrow z\\w'\rarrow z'}:=(2\pi i )^2\lim_{\substack{w\rarrow z\\w'\rarrow z'}}(w-z)(w'-z') \ \mathcal{G}
\end{align}
where in general we view the arrows $\substack{w\rarrow z\\\cdot}$ as operators applied to some function $\mathcal{G}$.
Consider the integral in \eqref{1term6poles}, deforming the $w_1'$ contour from $\Gamma{r'}$ to $\Gamma_{2/{r'}}$, and then deforming the $w_1$ contour from $\Gamma_r$ to $\Gamma_{2/{r'}}$ in all of the resulting integrals, the integral in \eqref{1term6poles} is
\begin{align}
&\int_{\Gamma_r}dw_1\int_{\Gamma_{1/r}}dw_2\int_{\Gamma_{r'}}dw_1'\int_{\Gamma_{1/r'}}dw_2' \ \mathcal{G}\nonumber\\
&=\int_{\Gamma_r}dw_1\int_{\Gamma_{1/r}}dw_2\int_{\Gamma_{2/{r'}}}dw_1'\int_{\Gamma_{1/r'}}dw_2' \ \mathcal{G}\nonumber\\
&\quad +\int_{\Gamma_r}dw_1\int_{\Gamma_{1/r}}dw_2\int_{\Gamma_{1/r'}}dw_2' \ \mathcal{G}\substack{w_1'\rarrow -w_2\\\cdot}+\mathcal{G}\substack{w_1'\rarrow w_2\\\cdot}+\mathcal{G}\substack{w_1'\rarrow w_2'\\\cdot}\nonumber\\
&\label{tempG1}=\int_{\Gamma_{2/r'}}dw_1\int_{\Gamma_{1/r}}dw_2\int_{\Gamma_{2/{r'}}}dw_1'\int_{\Gamma_{1/r'}}dw_2' \ \mathcal{G}\\\label{tempG2}&\quad+\int_{\Gamma_{2/r'}}dw_1\int_{\Gamma_{1/r}}dw_2\int_{\Gamma_{1/r'}}dw_2'\ \mathcal{G}\substack{w_1'\rarrow -w_2\\\cdot}+\mathcal{G}\substack{w_1'\rarrow w_2\\\cdot}+\mathcal{G}\substack{w_1'\rarrow w_2'\\\cdot}\\&
\quad\label{tempG3}+\int_{\Gamma_{1/r}}dw_2\int_{\Gamma_{2/r'}}dw_1'\int_{\Gamma_{1/r'}}dw_2' \ \mathcal{G}\substack{w_1\rarrow -w_2'\\\cdot}+\mathcal{G}\substack{w_1\rarrow w_2'\\\cdot}+\mathcal{G}\substack{w_1\rarrow w_2\\\cdot}\\
&\quad\label{tempG4}+\int_{\Gamma_{1/r}}dw_2\int_{\Gamma_{1/r'}}dw_2' \ \mathcal{G}\substack{w_1'\rarrow -w_2\\w_1\rarrow -w_2'}+\mathcal{G}\substack{w_1'\rarrow -w_2\\w_1\rarrow w_2'}+\mathcal{G}\substack{w_1'\rarrow -w_2\\w_1\rarrow w_2}\\
&\qquad\qquad\qquad\qquad\qquad+\mathcal{G}\substack{w_1'\rarrow w_2\\w_1\rarrow -w_2'}+\mathcal{G}\substack{w_1'\rarrow w_2\\w_1\rarrow w_2'}+\mathcal{G}\substack{w_1'\rarrow w_2\\w_1\rarrow w_2}\nonumber\\
&\qquad\qquad\qquad\qquad\qquad+\mathcal{G}\substack{w_1'\rarrow w_2'\\w_1\rarrow -w_2'}+\mathcal{G}\substack{w_1'\rarrow w_2'\\w_1\rarrow w_2'}+\mathcal{G}\substack{w_1'\rarrow w_2'\\w_1\rarrow w_2}.\nonumber
\end{align}
Now, the integral in \eqref{tempG1} is $O(e^{-Cn})$ by a similar argument to showing $R_n=O(e^{-cn})$ in the proof of Proposition \ref{backtrackingforheightfn} applied in both sets of variables $(w_1,w_2), (w_1',w_2')$. Next, the integrals in lines \eqref{tempG2}, \eqref{tempG3} are also $O(e^{-C'n})$, we show this for the first term in \eqref{tempG2}. Taking the limit in $\mathcal{G}\substack{w_1'\rarrow -w_2\\\cdot}$  we see 
\begin{align}
\lim_{w_1'\rarrow -w_2}\mathcal{A}=\Big(\frac{w_1G(w_1)G(w_2'^{-1})}{w_2'G(w_1'^{-1})G(w_2')}\Big)^{n/2}\Big(\frac{G(w_1')G(w_1'^{-1})}{G(w_2')G(w_2'^{-1})}\Big)^{t'q_n/2}\nonumber\Big(\frac{G(w_1^{-1})G(w_2')}{G(w_1)G(w_2'^{-1})}\Big)^{|t'q_n|/2}\nonumber
\end{align} 
via \eqref{Gsymmetries}, this is $O(e^{-C'n})$ again by the same argument we used to show $R_n=O(e^{-cn})$  in Proposition \ref{backtrackingforheightfn} applied to the variables $(w_1,w_2')$. Now consider the double integral in \eqref{tempG4}. In each of the terms $\mathcal{G}\substack{\cdot\rarrow \cdot\\\cdot\rarrow \cdot}$ appearing in \eqref{tempG4}, taking the limit in the corresponding $\mathcal{A}$ gives $\lim_{\substack{\cdot\rarrow \cdot\\\cdot\rarrow \cdot}}\mathcal{A}=1$ (recall that $n/2$ is even, as $n=4m$). From \eqref{relQ}, $V_{\eps_1,\eps_2}(w_1,-w_2)=(-1)^{-\eps_2-1}V_{\eps_1,\eps_2}(w_1,w_2)$, and considered as a function of $w_2'$, we see $\mathcal{B}(-w_2')=(-1)^{\eps+1}\mathcal{B}(w_2')$, hence, as a function of $w_2'$, $\mathcal{V}(-w_2')\mathcal{B}(-w_2')=\mathcal{V}(w_2')\mathcal{B}(w_2')$. Recall that in the proof of lemma \ref{Veps1eps2formula} we proved that $c^{-{\eps_1+\eps_2}/2}V_{\eps_1,\eps_2}(w_1,w_2)$ is analytic in $(a,w_1,w_2)\in  (B_\eps(0)\setminus \{0\})\times (B_R(0)\setminus B_{1/R}(0))^2\subset \C^3$, the same holds for $\mathcal{B}$, hence it holds for $\mathcal{V}\mathcal{B}$. Also observe that after taking limits in the nine terms $\mathcal{G}\substack{\cdot\rarrow \cdot\\\cdot\rarrow \cdot}$ in \eqref{tempG4}, the only poles in front of $\lim_{\substack{\cdot\rarrow \cdot\\\cdot\rarrow \cdot}}\mathcal{V}\mathcal{B}$ are of the form $1/(w_2'\pm w_2)$, $1/w_2$, $1/w_2'$. Since $|w_2'|\leq |w_2|$, of these poles the only ones inside the $w_2'$ integral are at zero. So we see we can Laurent expand all nine terms $\mathcal{G}\substack{\cdot\rarrow \cdot\\\cdot\rarrow \cdot}$ in \eqref{tempG4} in just the $a$ variable up to order $O(a)$ and that the coefficients of the expansion may only have $w_2'$-poles at $w_2'=0$ inside the $w_2'$ contour, (since we had that for any $R>1$, the coefficients are analytic for $w_2'\in B_R(0)\setminus B_{1/R}(0)$ when $\eps$ is small enough). Hence, up to $O(a)$, the value of $w_2'$ integral in \eqref{tempG4} is given by $2\pi i$ times the residue at $w_2'=0$ of the lower order coefficients in $a$. Observe that as a function of $w_2'$, 
\begin{align}
\mathcal{G}(-w_2')=-\frac{w_2'-w_1'}{w_2'+w_1'}\mathcal{G}(w_2').
\end{align}
Consider the first 3 terms in \eqref{tempG4} as a function of $w_2'$, and then take $w_2'\rarrow -w_2'$, the sum of these three terms is then
\begin{align}
&\Big(\mathcal{G}\substack{w_1'\rarrow -w_2\\w_1\rarrow -w_2'}+\mathcal{G}\substack{w_1'\rarrow -w_2\\w_1\rarrow w_2'}+\mathcal{G}\substack{w_1'\rarrow -w_2\\w_1\rarrow w_2}\Big)(-w_2')\nonumber\\
&=\Big(-\frac{w_2'-w_1'}{w_2'+w_1'}\mathcal{G}\Big)\substack{w_1'\rarrow -w_2\\w_1\rarrow w_2'}+\Big(-\frac{w_2'-w_1'}{w_2'+w_1'}\mathcal{G}\Big)\substack{w_1'\rarrow -w_2\\w_1\rarrow -w_2'}+\Big(-\frac{w_2'-w_1'}{w_2'+w_1'}\mathcal{G}\Big)\substack{w_1'\rarrow -w_2\\w_1\rarrow w_2}\nonumber\\
&=-\frac{w_2'+w_2}{w_2'-w_2}\Big(\mathcal{G}\substack{w_1'\rarrow -w_2\\w_1\rarrow w_2'}+\mathcal{G}\substack{w_1'\rarrow -w_2\\w_1\rarrow -w_2'}+\mathcal{G}\substack{w_1'\rarrow -w_2\\w_1\rarrow w_2}\Big)(w_2')\nonumber\\
&=-\frac{w_2'+w_2}{w_2'-w_2}\Big(\mathcal{G}\substack{w_1'\rarrow -w_2\\w_1\rarrow -w_2'}+\mathcal{G}\substack{w_1'\rarrow -w_2\\w_1\rarrow w_2'}+\mathcal{G}\substack{w_1'\rarrow -w_2\\w_1\rarrow w_2}\Big)(w_2'). \nonumber
\end{align}
Note that after the first equality, we view the arrows in \eqref{temp3rgfdann} as operators on a function $\mathcal{G}$, which in this equality is replaced by the function $-(w_2'-w_1')/(w_2'+w_1')\mathcal{G}$.
So we see that, as a function of $w_2'$,
\begin{align} 
(w_2+w_2')\Big(\mathcal{G}\substack{w_1'\rarrow -w_2\\w_1\rarrow -w_2'}+\mathcal{G}\substack{w_1'\rarrow -w_2\\w_1\rarrow w_2'}+\mathcal{G}\substack{w_1'\rarrow -w_2\\w_1\rarrow w_2}\Big)
\end{align}
is invariant under $w_2'\rarrow -w_2'$, so its coefficients residues at $w_2'=0$ are equal to zero. Hence the $w_2'$ integral of these first three terms is $O(a)$. By a similar argument, the next three terms transform as
\begin{align}
&\Big(\mathcal{G}\substack{w_1'\rarrow w_2\\w_1\rarrow -w_2'}+\mathcal{G}\substack{w_1'\rarrow w_2\\w_1\rarrow w_2'}+\mathcal{G}\substack{w_1'\rarrow w_2\\w_1\rarrow w_2}\Big)(-w_2')\nonumber\\
&=-\frac{w_2'-w_2}{w_2'+w_2}\Big(\mathcal{G}\substack{w_1'\rarrow w_2\\w_1\rarrow -w_2'}+\mathcal{G}\substack{w_1'\rarrow w_2\\w_1\rarrow w_2'}+\mathcal{G}\substack{w_1'\rarrow w_2\\w_1\rarrow w_2}\Big)(w_2').\nonumber
\end{align}
So
\begin{align}
(w_2-w_2')\Big(\mathcal{G}\substack{w_1'\rarrow w_2\\w_1\rarrow -w_2'}+\mathcal{G}\substack{w_1'\rarrow w_2\\w_1\rarrow w_2'}+\mathcal{G}\substack{w_1'\rarrow w_2\\w_1\rarrow w_2}\Big)
\end{align}
is also invariant under $w_2'\rarrow -w_2'$, which shows that the $w_2'$ integral of the next three terms also equals $O(a)$. Then finally, one checks that each term in
\begin{align}
\mathcal{G}\substack{w_1'\rarrow w_2'\\w_1\rarrow -w_2'}+\mathcal{G}\substack{w_1'\rarrow w_2'\\w_1\rarrow w_2'}+\mathcal{G}\substack{w_1'\rarrow w_2'\\w_1\rarrow w_2}
\end{align}
is invariant under $w_2'\rarrow -w_2'$, so the $w_2'$ integral of them is again $O(a)$.

Now we consider the integrals arising from the cross terms when expanding the brackets in \eqref{geomsumresult}, specifically we consider the term arising from
\begin{align}
\frac{-\Big(\frac{G(w_1^{-1})G(w_2')}{G(w_1)G(w_2'^{-1})}\Big)^{L}}{1-\Big(\frac{G(w_1^{-1})G(w_2')}{G(w_1)G(w_2'^{-1})}\Big)}\times\frac{1}{1-\Big(\frac{G(w_1'^{-1})G(w_2)}{G(w_1')G(w_2^{-1})}\Big)},
\end{align}
the other cross term goes to zero by a similar argument.
Define
\begin{align}
\mathcal{A}^*:=& \ \mathcal{A}\times (-1)\Big(\frac{G(w_1^{-1})G(w_2')}{G(w_1)G(w_2'^{-1})}\Big)^L\\
=& \ \Big(\frac{w_1'}{w_2}\Big)^{n/2}\Big(\frac{G(w_1')G(w_2^{-1})}{G(w_1'^{-1})G(w_2)}\Big)^{n/2}\Big(\frac{G(w_1')G(w_1'^{-1})}{G(w_2)G(w_2^{-1})}^{t'q_n/2}\Big(\frac{G(w_1'^{-1})G(w_2)}{G(w_1')G(w_2^{-1})}\Big)^{|t'q_n|/2}\nonumber\\
&\times(-1)\Big(\frac{w_1}{w_2'}\Big)^{n/2}\Big(\frac{G(w_1)G(w_2'^{-1})}{G(w_1^{-1})G(w_2')}\Big)^{-n\xi_c/2-\alpha p_n/2-1}\Big(\frac{G(w_1')G(w_1'^{-1}}{G(w_2)G(w_2^{-1})}\Big)^{t'q_n/2}\nonumber
\end{align}
and analogously
\begin{align}
\mathcal{G}^*=\frac{1}{w_1w_1'}\frac{\mathcal{V}\times\mathcal{A}^*\times\mathcal{B}}{(w_2-w_1)(w_2'-w_1)(w_2'-w_1')(w_2'+w_1)(w_2+w_1')}.
\end{align}
Explicitly this cross term gives rise to
\begin{align}
\frac{(-1)^{\eps+\eps'+1}}{(2\pi i)^4}\int_{\Gamma_r}dw_1\int_{\Gamma_{1/r}}dw_2\int_{\Gamma_{r'}}dw_1'\int_{\Gamma_{1/r'}}dw_2  \ \mathcal{G}^*.\label{temp34fdmnmn}
\end{align}
Deform the $w_2$ contour to $\Gamma_1$, this crosses the $w_2'$ contour, but no poles. Then deform the $w_1'$ contour to $\Gamma_{R+1}$, this crosses the simple poles at $\pm w_2$ and $w_2'$. Then, in all resulting integrals, deform the $w_1$, $w_2'$  contours to the contours $desc^{\eps,n}_{1/R}$ and $asc^{\eps,n}_R$ from the proof that $X_n=O(e^{-c\alpha })$ in Proposition \eqref{ExpectationProp}. The integral in \eqref{temp34fdmnmn} becomes
\begin{align}
\label{tempqudin1}&\int_{desc^{\eps,n}_{1/R}}dw_1\int_{\Gamma_1}dw_2\int_{\Gamma_{R+1}}dw_1'\int_{asc^{\eps,n}_{R}}dw_2'  \ \mathcal{G}^*\\&
+\int_{desc^{\eps,n}_{1/R}}dw_1\int_{\Gamma_1}dw_2\int_{asc^{\eps,n}_{R}}dw_2'  \ \mathcal{G}^*\substack{w_1'\rarrow w_2\\\cdot}+\mathcal{G}^*\substack{w_1'\rarrow -w_2\\\cdot}.\label{tempqudin3}\\
&+\int_{desc^{\eps,n}_{1/R}}dw_1\int_{\Gamma_1}dw_2\int_{asc^{\eps,n}_{R}}dw_2'  \ \mathcal{G}^*\substack{w_1'\rarrow w_2'\\\cdot}\label{tempqudin2}
\end{align}
For large enough $n$, the contour $\Gamma_1$ is at minimum distance $C/(an)^{1/3}$ from both $desc^{\eps,n}_{1/R}$ and $asc^{\eps,n}_R$ by basic trigonometry.  The integral in \eqref{tempqudin1} is $O(e^{-Cn})$ by combining the argument we used to show $R_n=O(e^{-cn})$ in the proof of Proposition \ref{backtrackingforheightfn} applied in the variables $(w_1',w_2)$ and the argument  we used to show $X_n=O(e^{-c_1\alpha })$ in the proof of Proposition \ref{ExpectationProp}, applied in the variables $(w_1,w_2')$. The integral in \eqref{tempqudin3} is $O(e^{-c\alpha })$ again by the argument  we used to show $X_n=O(e^{-c_1\alpha })$, applied in the variables $(w_1,w_2')$. In \eqref{tempqudin2}, deform the $w_1$ contour to $\Gamma_{1/(R+2)}$, then the $w_2$ contour to $\Gamma_{1/(R+1)}$, then deform the $w_1$ contour back to  $desc^{\eps,n}_{1/R}$ crossing over the pole at $w_2$. The integral in \eqref{tempqudin2} becomes
\begin{align}
\int_{desc^{\eps,n}_{1/R}}dw_1\int_{\Gamma_{1/(R+1)}}dw_2\int_{asc^{\eps,n}_{R}}dw_2'  \ \mathcal{G}^*\substack{w_1'\rarrow w_2'\\\cdot}+\int_{\Gamma_{1/(R+1)}}dw_2\int_{asc_R^{\eps,n}}dw_2'\mathcal{G}^*\substack{w_1'\rarrow w_2'\\ w_1\rarrow w_2}\label{temp23r;adfla|}.
\end{align}
Looking at \eqref{temp23r;adfla|}, we see $|w_2|\leq |w_2'|$, so the first integral is $O(e^{-cn})$ by combining the argument we used to show $R_n=O(e^{-cn})$ in the proof of Proposition \ref{backtrackingforheightfn} applied in the variables $(w_2,w_2')$ and the argument  we used to show $X_n=O(e^{-c_1\alpha })$ in the proof of Proposition \ref{ExpectationProp}, applied in the variables $(w_1,w_2')$. For the second integral in \eqref{temp23r;adfla|}, deform the $w_2'$ contour back to $\Gamma_{1/r'}$, observe that
\begin{align}
\mathcal{G}^*\substack{w_1'\rarrow w_2'\\ w_1\rarrow w_2}=\mathcal{G}\substack{w_1'\rarrow w_2'\\ w_1\rarrow w_2}
\end{align}
and by an analogous argument (this time applied in $w_2$, instead of $w_2'$, since $|w_2|<|w_2'|$) we compute the $w_2$ integral to be $O(a)$ as we did for the integral in  \eqref{tempG4} (see the last term of the integral in \eqref{tempG4}).
Hence we have established
\begin{align}
\sum_{i,j}B_{\eps,\eps'}(x_{e_\eps},y_{e_{\eps'}})B_{\eps',\eps}(x_{e_{\eps'}},y_{e_{\eps}})\rarrow 0
\end{align}
from which \eqref{temphorithj} then follows trivially.

Next we show 
\begin{align}
\abs{a^2\sum_{i,j}\K_{1,1}^{-1}(x_{e_\eps},y_{e_{\eps'}})\K_{1,1}^{-1}(x_{\eps'},y_{\eps})}\rarrow 0.\label{temp34rdskjkj}
\end{align}
This is rather simple, since by Lemma
\ref{boundK^-1} the left hand side of \eqref{temp34rdskjkj} is less than or equal to 
\begin{align}
&a^2\sum_{\substack{i,j\\ \abs{i-j}\neq 2}}C\min(a,a^{\abs{j-i}/2-1})\min(a,a^{\abs{j-i}/2-1})+a^2\sum_{\substack{i,j\\\abs{i-j}=2}}C\nonumber\\&
\leq C'a^4n^2+C'a^2n\rarrow 0.\nonumber
\end{align}
\end{proof}
\end{proposition}

Now we have the information we need to prove Proposition \ref{expectheightprop}.
\begin{proof}[Proof of Proposition \ref{expectheightprop}]
In the notation of the current section, given a fixed collection $(t_1,...,t_j)\in \R^j$,
\begin{align}
\{\text{$a$-height along $\partial S^T_{\{t\}_j}=4m$}\}=\{h(t_1')=n,..., h(t_j')=n\}.
\end{align}
We prove that 
\begin{align}
\P_{Az}(\{h(t_1')=n,..., h(t_j')=n\}^c)\rarrow 0.
\end{align}
Observe that
\begin{align}
\P(\{h(t_1')=n,..., h(t_j')=n\}^c)\leq \sum_{i=1}^j\P_{Az}(\{h(t_i')=n\}^c),
\end{align}
so it suffices to show 
\begin{align}
\P_{Az}(\{h(t')=n\}^c)\rarrow 0
\end{align} for $t\in \{t_1,...,t_j\}$. Recall the $a$-height is a multiple of four, by Chebyshev's/Markov's inequality we have
\begin{align}
\P_{Az}(\{h(t')=n\}^c)&=\P_{Az}(\{\abs{h(t')-n}>2\})\\
&\leq \P_{Az}(\{\abs{h(t')-\E_{Az}[h(t')]}>1\})+\P_{Az}(\{\abs{\E_{Az}[h(t')]-n}>1\})\nonumber\\
&\leq \text{Var}_{Az}(h(t'))+\P_{Az}(\{\abs{\E_{Az}[h(t')]-n}>1\}).\nonumber
\end{align}
By Proposition \ref{ExpectationProp}, for all $n$ large enough we have $\{\abs{\E_{Az}[h(t')]-n}>1\}=\emptyset$, so by Proposition \ref{vargoestozero} the last upper bound tends to zero as $n$ tends to infinity.
\end{proof}
\newpage

\section{Miscellaneous lemmas}\label{Misclemmas}
Our first lemma allows us to rewrite the definition of $V_{\eps_1,\eps_2}(w_1,w_2)$ which appears in \cite{C/J} in a way which is easier to deal with.
Recall that $V$ is defined in \cite{C/J} as
\begin{align}
V_{\eps_1,\eps_2}(w_1,w_2)=\frac{1}{2}(Z_{\eps_1,\eps_2}(w_1,w_2)+(-1)^{\eps_2+1}Z_{\eps_1,\eps_2}(w_1,-w_2))\label{VinCJ}
\end{align}
where
\begin{align}
Z_{\eps_1,\eps_2}(w_1,w_2)=\sum_{\gamma_1,\gamma_2=0}^1Q_{\gamma_1,\gamma_2}^{\eps_1,\eps_2}(w_1,w_2)
\end{align}
and $Q$ is defined as follows.
For the moment,  we consider general weights $a,b>0$. Let 
\begin{align}
f_{a,b}(u,v)=(2a^2uv+2b^2uv-ab(-1+u^2)(-1+v^2))(2a^2uv+2b^2uv+ab(-1+u^2)(-1+v^2))\label{fabCJ}
\end{align}
and define the following rational functions
\begin{align}
&y_{0,0}^{0,0}(a,b,u,v)=\frac{1}{4(a^2+b^2)^2f_{a,b}(u,v)}(2a^7u^2v^2-a^5b^2(1+u^4+u^2v^2-u^4v^2+v^4-u^2v^4)\nonumber\\
&\quad\qquad-a^3b^4(1+3u^2+3v^2+2u^2v^2+u^4v^2+u^2v^4-u^4v^4)-ab^6(1+v^2+u^2+3u^2v^2)),\nonumber\\
&y_{0,1}^{0,0}(a,b,u,v)=\frac{a(b^2+a^2u^2)(2a^2v^2+b^2(1+v^2-u^2+u^2v^2))}{4(a^2+b^2)f_{a,b}(u,v)},\nonumber\\
&y_{1,0}^{0,0}(a,b,u,v)=\frac{a(b^2+a^2v^2)(2a^2u^2+b^2(1-v^2+u^2+u^2v^2))}{4(a^2+b^2)f_{a,b}(u,v)},\nonumber\\
&y_{1,1}^{0,0}(a,b,u,v)=\frac{a(2a^2u^2v^2+b^2(-1+v^2+u^2+u^2v^2))}{4f_{a,b}(u,v)}.\nonumber
\end{align}
For $i,j\in\{0,1\}$ define 
\begin{align}
&y_{i,j}^{0,1}(a,b,u,v)=\frac{y_{i,j}^{0,0}(b,a,u,v^{-1})}{v^2},\nonumber \\
&y_{i,j}^{1,0}(a,b,u,v)=\frac{y_{i,j}^{0,0}(b,a,u^{-1},v)}{u^2},\nonumber\\
&y_{i,j}^{1,1}(a,b,u,v)=\frac{y_{i,j}^{0,0}(a,b,u^{-1},v^{-1})}{u^2v^2}.\nonumber
\end{align}
When $b=1$, write $y_{i,j}^{\eps_1,\eps_2}(a,1,u,v)=y_{i,j}^{\eps_1,\eps_2}(u,v)$. Set
\begin{align}
x_{\gamma_1,\gamma_2}^{\eps_1,\eps_2}(w_1,w_2)=\frac{G(w_1)G(w_2)(1-w_1^2w_2^2)}{\prod_{j=1}^2\sqrt{w_j^2+2c}\sqrt{w_j^{-2}+2c}}y_{\gamma_1,\gamma_2}^{\eps_1,\eps_2}(G(w_1),G(w_2)),
\end{align} 
and so finally,
\begin{align}
Q_{\gamma_1,\gamma_2}^{\eps_1,\eps_2}(w_1,w_2)=(-1)^{\eps_1+\eps_2+\eps_1\eps_2+\gamma_1(1+\eps_2)+\gamma_2(1+\eps_1)}s(w_1)^{\gamma_1}s(w_2^{-1})^{\gamma_2}G(w_1)^{\eps_1}G(w_2^{-1})^{\eps_2}x_{\gamma_1,\gamma_2}^{\eps_1,\eps_2}(w_1,w_2^{-1}).\label{QCJ}
\end{align}
We first give a lemma regarding $f_{a,1}$.
\begin{lemma}\label{fabsimp}
Let $\eps_1,\eps_2\in\{0,1\}$. We have
\begin{align}
f_{a,1}(G(w_1)^{1-2\eps_1},G(w_2)^{1-2\eps_2})=4(1+a^2)^2(1-w_1^2w_2^2)G(w_1)^{2-4\eps_1}G(w_2)^{2-4\eps_2}.
\end{align}
\begin{proof}
Observe that \eqref{fabCJ} is written in the form $(A-B)(A+B)$, where $A=2a^2uv+2b^2uv$, $B=ab(-1+u^2)(-1+v^2)$. Expand these brackets and then rewrite it to get
\begin{align}
f_{a,1}(u,v)=4(a^2+1)^2u^2v^2(1-\frac{c^2}{4}(u-1/u)^2(v-1/v)^2).
\end{align}
By the invariance of $u\mapsto (u-1/u)^2$ under $u\rarrow 1/u$ we have
\begin{align}
f_{a,1}(u^{1-2\eps_1},v^{1-2\eps_2})=4(1+a^2)^2(1-\frac{c^2}{4}(u-1/u)^2(v-1/v)^2)u^{2-4\eps_1}v^{2-4\eps_2}.
\end{align}
Now let $u=G(w_1)$, $v=G(w_2)$ and note that $G$ is the inverse of $u\mapsto \sqrt{c/2}(u-1/u)$.
\end{proof}
\end{lemma}
By their definitions we have the relation
\begin{align}
\tilde{y}_{i,j}^{0,0}(a,b,u,v)=f_{a,b}(u,v)y_{i,j}^{0,0}(a,b,u,v),\nonumber
\end{align}
and since $f_{a,b}(u,v)=f_{b,a}(u,v)$,
\begin{align}
\tilde{y}_{i,j}^{0,0}(a,b,u,v)&=v^2f_{a,b}(u,v^{-1})y_{i,j}^{0,1}(a,b,u,v)=\tilde{y}_{i,j}^{0,0}(b,a,u,v^{-1})\nonumber\\
\tilde{y}_{i,j}^{1,0}(a,b,u,v)&=u^2f_{a,b}(u^{-1},v)y_{i,j}^{1,0}(a,b,u,v)=\tilde{y}_{i,j}^{0,0}(b,a,u^{-1},v)\nonumber\\
\tilde{y}_{i,j}^{1,1}(a,b,u,v)&=u^2v^2f_{a,b}(u^{-1},v^{-1})y_{i,j}^{1,0}(a,b,u,v)=\tilde{y}_{i,j}^{0,0}(a,b,u^{-1},v^{-1}).\nonumber
\end{align}
Hence, 
\begin{align}
y_{\gamma_1,\gamma_2}^{\eps_1,\eps_2}(G(w_1),G(w_2))=\frac{\tilde{y}_{\gamma_1,\gamma_2}^{\eps_1,\eps_2}(a,1,G(w_1),G(w_2))}{G(w_1)^{2\eps_1}G(w_2)^{2\eps_2}f_{a,1}(G(w_1)^{1-2\eps_1},G(w_2)^{1-2\eps_2})}.
\end{align}
Now it follows by lemma \ref{fabsimp} that
\begin{align}
x_{\gamma_1,\gamma_2}^{\eps_1,\eps_2}(w_1,w_2)=\frac{G(w_1)^{2\eps_1-1}G(w_2)^{2\eps_2-1}}{4(1+a^2)^2\prod_{j=1}^2\sqrt{w_j^2+2c}\sqrt{w_j^{-2}+2c}}\tilde{y}_{\gamma_1,\gamma_2}^{\eps_1,\eps_2}(G(w_1),G(w_2)).
\end{align}
By substituting this into \eqref{QCJ} we see we have our definition of $Q$ given by \eqref{Qours}.
Looking at \eqref{Qours}, observe that we have
\begin{align}
Q_{\gamma_1,\gamma_2}^{\eps_1,\eps_2}(w_1,-w_2)=(-1)^{3\eps_2-1}Q_{\gamma_1,\gamma_2}^{\eps_1,\eps_2}(w_1,w_2)=(-1)^{-\eps_2-1}Q_{\gamma_1,\gamma_2}^{\eps_1,\eps_2}(w_1,w_2)\label{relQ}
\end{align}
which follows by $\tilde{y}_{i,j}^{\eps_1,\eps_2}(u,-v)=\tilde{y}_{i,j}^{\eps_1,\eps_2}(u,v)$, the symmetries \eqref{squarerootsymmetries}, \eqref{Gsymmetries} and the fact that $(-1)^{3\eps_2-1}=(-1)^{-\eps_2-1}$. Using the relation \eqref{relQ} in \eqref{VinCJ} we arrive at the following lemma
\begin{lemma}\label{tempVequivlemma}
The formula for $V_{\eps_1,\eps_2}(w_1,w_2)$ given by \eqref{VinCJ} is equal to the formula given by \eqref{Vours}.
\end{lemma}
We give the following formula for $B^*_{\eps_1,\eps_2}(a,x_1,x_2,y_1,y_2)$ from \cite{C/J}.
\begin{align}
B^*_{\eps_1,\eps_2}(a,x_1,x_2,y_1,y_2)=& \ \frac{i}{a}(-1)^{\eps_1+\eps_2}(B_{1-\eps_1,\eps_2}(1/a,2n-x_1,x_2,2n-y_1,y_2)\\
& +B_{\eps_1,1-\eps_2}(1/a,x_1,2n-x_2,y_1,2n-y_2))\nonumber\\&-B_{1-\eps_1,1-\eps_2}(a,2n-x_1,2n-x_2,2n-y_1,2n-y_2).\nonumber
\end{align}
These three integrals have the following representations (Corollary 2.4 in \cite{C/J}),
\begin{align}
\label{firstBstar}&\frac{1}{a}B_{1-\eps_1,\eps_2}(1/a,2n-x_1,x_2,2n-y_1,y_2)=-\frac{i^{\frac{x_1-x_2-y_1-y_2}{2}}}{(2\pi i)^2}\\&\int_{\Gamma_r}\frac{dw_1}{w_1}\int_{\Gamma_{1/r}}dw_2\frac{w_2}{w_2^2-w_1^2}\frac{H_{x_1+1,x_2}(w_1)}{H_{2n-y_1,y_2+1}(w_2)}\sum_{\gamma_1,\gamma_2=0}^1(-1)^{\eps_2+\gamma_2}Q_{\gamma_1,\gamma_2}^{\eps_1,\eps_2}(w_1,w_2),\nonumber\\
&\frac{1}{a}B_{\eps_1,1-\eps_2}(1/a,x_1,2n-x_2,y_1,2n-y_2)=-\frac{i^{\frac{y_2-y_1-x_2-x_1}{2}}}{(2\pi i)^2}\\&\int_{\Gamma_r}\frac{dw_1}{w_1}\int_{\Gamma_{1/r}}dw_2\frac{w_2}{w_2^2-w_1^2}\frac{H_{x_1+1,2n-x_2}(w_1)}{H_{y_1,y_2+1}(w_2)}\sum_{\gamma_1,\gamma_2=0}^1(-1)^{\eps_1+\gamma_1}Q_{\gamma_1,\gamma_2}^{\eps_1,\eps_2}(w_1,w_2)\nonumber
\end{align}
and
\begin{align}
&B_{1-\eps_1,1-\eps_2}(a,2n-x_1,2n-x_2,2n-y_1,2n-y_2)=-\frac{i^{\frac{y_2+y_1+x_2+x_1}{2}}}{(2\pi i)^2}\\&\int_{\Gamma_r}\frac{dw_1}{w_1}\int_{\Gamma_{1/r}}dw_2\frac{w_2}{w_2^2-w_1^2}\frac{H_{x_1+1,2n-x_2}(w_1)}{H_{2n-y_1,y_2+1}(w_2)}\sum_{\gamma_1,\gamma_2=0}^1(-1)^{\eps_1+\gamma_1+\eps_2+\gamma_2}Q_{\gamma_1,\gamma_2}^{\eps_1,\eps_2}(w_1,w_2).\nonumber
\end{align}
We simplify these three integrals for further asymptotic analysis. We follow the procedure in \cite{C/J} for this simplification. We see that the factor in the integrand of \eqref{firstBstar} is
\begin{align}
\frac{w_2}{w_2^2-w_1^2}=\frac{1}{2}\Big(\frac{1}{w_2-w_1}+\frac{1}{w_2+w_1}\Big).
\end{align}
Write the integral as a sum over these two terms, and then make the change of variables $w_2\rarrow -w_2$. We see that
\begin{align}
H_{2n-y_1,y_2+1}(-w_2)=(-1)^{(y_1+y_2+1)/2}H_{2n-y_1,y_2+1}(w_2)=(-1)^{\eps_2+1}H_{2n-y_1,y_2+1}(w_2)
\end{align}
where the last equality holds because $y_1+y_2 \text{ mod }4=2\eps_2+1$ when $(y_1,y_2)\in B_{\eps_2}$. We see this factor of $(-1)^{\eps_2+1}$ multiplies with same factor in \eqref{relQ} to be $1$, so we get \eqref{firstBstar} as
\begin{align}
\label{B2}&\frac{1}{a}B_{1-\eps_1,\eps_2}(1/a,2n-x_1,x_2,2n-y_1,y_2)=-\frac{i^{\frac{x_1-x_2-y_1-y_2}{2}}}{(2\pi i)^2}\\&\int_{\Gamma_r}\frac{dw_1}{w_1}\int_{\Gamma_{1/r}}dw_2\frac{1}{w_2-w_1}\frac{H_{x_1+1,x_2}(w_1)}{H_{2n-y_1,y_2+1}(w_2)}\sum_{\gamma_1,\gamma_2=0}^1(-1)^{\eps_2+\gamma_2}Q_{\gamma_1,\gamma_2}^{\eps_1,\eps_2}(w_1,w_2).\nonumber
\end{align}
A similar argument gives
\begin{align}
\label{B3}&\frac{1}{a}B_{\eps_1,1-\eps_2}(1/a,x_1,2n-x_2,y_1,2n-y_2)=-\frac{i^{\frac{y_2-y_1-x_2-x_1}{2}}}{(2\pi i)^2}\\&\int_{\Gamma_r}\frac{dw_1}{w_1}\int_{\Gamma_{1/r}}dw_2\frac{1}{w_2-w_1}\frac{H_{x_1+1,2n-x_2}(w_1)}{H_{y_1,y_2+1}(w_2)}\sum_{\gamma_1,\gamma_2=0}^1(-1)^{\eps_1+\gamma_1}Q_{\gamma_1,\gamma_2}^{\eps_1,\eps_2}(w_1,w_2)\nonumber
\end{align}
and
\begin{align}
\label{B4}&B_{1-\eps_1,1-\eps_2}(a,2n-x_1,2n-x_2,2n-y_1,2n-y_2)=-\frac{i^{\frac{y_2+y_1+x_2+x_1}{2}}}{(2\pi i)^2}\\&\int_{\Gamma_r}\frac{dw_1}{w_1}\int_{\Gamma_{1/r}}dw_2\frac{1}{w_2-w_1}\frac{H_{x_1+1,2n-x_2}(w_1)}{H_{2n-y_1,y_2+1}(w_2)}\sum_{\gamma_1,\gamma_2=0}^1(-1)^{\eps_1+\gamma_1+\eps_2+\gamma_2}Q_{\gamma_1,\gamma_2}^{\eps_1,\eps_2}(w_1,w_2).\nonumber
\end{align}
These are the integrals we perform our asymptotic analysis on.
We have a simple lemma on the fact that $B^*$ is small near the main diagonal of the lower left quadrant.
\begin{lemma} \label{B*smallmaindiag} Fix $\eps>0$ small, e.g. $\eps<1/4$.
Let $(x,y)\in W_{\eps_1}\times B_{\eps_2}$ be such that
\begin{align}
x=(x_1,x_2)=n(1+\xi)\vec{e}_1+\tilde{x}_1\vec{e}_1+\tilde{x}_2\vec{e}_2,\nonumber\\
y=(y_1,y_2)=n(1+\xi)\vec{e}_1+\tilde{y}_1\vec{e}_1+\tilde{y}_2\vec{e}_2\nonumber
\end{align}
where $|\tilde{x}_1|,|\tilde{x}_2|,|\tilde{y}_1|,|\tilde{y}_2|\leq n^{1-\eps}$ and $-1\leq \xi<0$. There are positive constants $C_1,C_2$ such that
\begin{align}
|B^*_{\eps_1,\eps_2}(a,x_1,x_2,y_1,y_2)|\leq C_1e^{-C_2n}
\end{align}
uniformly for $\xi, \tilde{x}_i,\tilde{y_i}$ for $a>0$ small enough. Moreover, there are positive constants $C_3,C_4$ such that
\begin{align}
\abs{B_{\eps_1,\eps_2}(a,x_1,x_2,y_1,y_2)}\leq C_3e^{-C_4n}.\label{temp23efgfgf}
\end{align}
uniformly for $\tilde{x}_i,\tilde{y_i}$, $-1/2+\eps<\xi<0$, and $a>0$ small enough
\begin{proof}
A similar proposition is proved for $a$ fixed in lemma 3.6 in \cite{C/J}. In our setting the argument is simpler, so we will be rather brief. Consider just the  terms with an exponential dependence on $n$ in the three integrals \eqref{B2}, \eqref{B3} and \eqref{B4} that make up $B^*$. They are
\begin{align}
\frac{H_{x_1+1,x_2}(w_1)}{H_{2n-y_1,y_2+1}(w_2)},&& \frac{H_{x_1+1,2n-x_2}(w_1)}{H_{y_1,y_2+1}(w_2)},&& \frac{H_{x_1+1,2n-x_2}(w_1)}{H_{2n-y_1,y_2+1}(w_2)}\label{405iktporefld}
\end{align}
where $|w_1|=r, |w_2|=1/r$ for a fixed $r\in (\sqrt{2c},1)$.
By substitution
\begin{align}
&\label{tmptads}\frac{H_{x_1+1,x_2}(w_1)}{H_{y_1,y_2+1}(w_2)}\\&=\Big(\frac{w_1}{w_2}\Big)^{n/2}\Big(\frac{G(w_1)G(1/w_2)}{G(1/w_1)G(w_2)}\Big)^{n/2}\frac{G(1/w_1)^{x_2/2}}{G(w_1)^{x_1/2}}\frac{G(w_2)^{y_1/2}}{G(1/w_2)^{y_2/2}}\frac{1}{G(w_1)^{1/2}G(1/w_2)^{1/2}}.\nonumber
\end{align}
Observe the factor of $G(1/w_1)^{x_2/2}G(w_2)^{y_1/2}$ in the numerator. The idea is that this factor is also small in the cases \eqref{405iktporefld}. We just check the first of the cases in \eqref{405iktporefld}. We can ignore the terms depending on $\tilde{x}_1,\tilde{x}_2,\tilde{y}_1,\tilde{y}_2$ and the terms with no exponential dependence since they will not impact the final bound. We have $x_i,y_i\sim n(1+\xi)$, and $2n-y_1\sim n(1-\xi)=n(1+\xi)-2n\xi$. Hence
\begin{align}
&\frac{H_{x_1+1,x_2}(w_1)}{H_{2n-y_1,y_2+1}(w_2)}\\
&\sim\Big(\frac{w_1}{w_2}\Big)^{n/2}\Big(\frac{G(w_1)G(1/w_2)}{G(1/w_1)G(w_2)}\Big)^{n/2}\frac{G(1/w_1)^{x_2/2}}{G(w_1)^{x_1/2}}\frac{G(w_2)^{(2n-y_1)/2}}{G(1/w_2)^{y_2/2}}\nonumber\\
&=\Big(\frac{w_1}{w_2}\Big)^{n/2}\Big(\frac{G(w_1)G(1/w_2)}{G(1/w_1)G(w_2)}\Big)^{n\xi/2}G(w_2)^{-n\xi}.\nonumber
\end{align}
We have $|G(w_2)^{-n\xi}|\leq Ca^{-n\xi/2}\leq C$ uniformly in $\xi,w_2$. We also have $|w_1|/|w_2|=r^2<1$, and by the expansion \eqref{tempGoverGexpans}, 
\begin{align}
\abs{\frac{G(w_1)G(1/w_2)}{G(1/w_1)G(w_2)}}^{-1}\leq \frac{|w_1|^2}{|w_2|^2}+aR(w_1,w_2,a)< 1
\end{align}
for $a>0$ small enough.
Hence we see
\begin{align}
&\abs{\frac{H_{x_1+1,x_2}(w_1)}{H_{2n-y_1,y_2+1}(w_2)}}\leq C_1'e^{-C_2'n}.
\end{align}
The remaining cases in \eqref{405iktporefld} follow similarly. Likewise, the case \eqref{temp23efgfgf} follows by applying a similar, simple argument to \eqref{tmptads}.
\end{proof}
\end{lemma}

Next, we have the following bound, this bound in particular tells us that the probability of seeing an $a$-dimer in the smooth phase is less than $Ca^2$.

\begin{lemma}
\label{boundK^-1}
Assume 
\begin{align}
e_j=(\tilde{x}_j,\tilde{y}_j)\in W_{\eps_1(j)}\times B_{\eps_2(j)}
\end{align}
are two edges $j\in \{0,1\}$ in $AD$ with coordinates
\begin{align}
\tilde{x}_j=(1+\tilde{\alpha}_j)\vec{e}_1+\tilde{\beta}_j\vec{e}_2+(0,2\eps_1-1)\\
\tilde{y}_j=(1+\tilde{\alpha}_j)\vec{e}_1+\tilde{\beta}_j\vec{e}_2+(2\eps_2-1,0)
\end{align}
where $(\tilde{\alpha}_j,\tilde{\beta}_j)\in (2\Z)^2$ and $\eps_i(j)=:\eps_i$.
Let 
\begin{align}
\tilde{\beta}=\tilde{\beta}_j-\tilde{\beta}_i,&&\tilde{\alpha}=\tilde{\alpha}_j-\tilde{\alpha}_i.
\end{align}
Then 
\begin{align}
|\K_{1,1}^{-1}(\tilde{x_j},\tilde{y}_i)|\leq \begin{cases}C a^{(|\tilde{\beta}+\tilde{\alpha}|+|\tilde{\beta}-\tilde{\alpha}|)/4-1}&(\tilde{\alpha},\tilde{\beta})\neq (0,0)\\
Ca& \tilde{\alpha},\tilde{\beta}=0
\end{cases}
\end{align}
when $a>0$ is small.
\begin{proof}
\begin{align}
|\K_{1,1}^{-1}(\tilde{x_j},\tilde{y}_i)|&\leq C_1\abs{\int_{\Gamma_1}\frac{dw}{w}\frac{a^{\eps_2}G(w)^{|\ell_1|}G(1/w)^{|k_1|}}{\sqrt{w^2+2c}\sqrt{1/w^2+2c}}} +C_2\abs{\int_{\Gamma_1}\frac{dw}{w}\frac{a^{1-\eps_2}G(w)^{|\ell_2|}G(1/w)^{|k_2|}}{\sqrt{w^2+2c}\sqrt{1/w^2+2c}}}\\
&\leq C(a^{\eps_2}a^{\frac{1}{2}(|\ell_1|+|k_1|)}+a^{1-\eps_2}a^{\frac{1}{2}(|\ell_2|+|k_2|)})
\end{align}
when $a$ is small, by \eqref{Gexpans}
Define 
\begin{align}\label{sig1sig2}
\sig_1=
\begin{cases}
\text{sign}(\tilde{\alpha}+\tilde{\beta}), & \text{if }\tilde{\alpha}\neq -\tilde{\beta}\\
-1,& \text{if }\tilde{\alpha}=-\tilde{\beta} \text{ and } \eps_1=0\\
1,& \text{if }\tilde{\alpha}=-\tilde{\beta} \text{ and } \eps_1=1
\end{cases},
&&
\sig_2=
\begin{cases}
\text{sign}(\tilde{\alpha}-\tilde{\beta}), & \text{if }\tilde{\alpha}\neq \tilde{\beta}\\
1,& \text{if }\tilde{\alpha}=\tilde{\beta} \text{ and } \eps_2=0\\
-1,& \text{if }\tilde{\alpha}=\tilde{\beta} \text{ and } \eps_2=1
\end{cases}
\end{align}
so that 
\begin{align}
&|k_1|=\frac{|\tilde{\alpha}+\tilde{\beta}|}{2}+\sig_1(2\eps_1-1)(1-\eps_2), &&|\ell_1|=\frac{|\tilde{\alpha}-\tilde{\beta}|}{2}+\sig_2(1-\eps_2),\\
&|k_2|=\frac{|\tilde{\alpha}+\tilde{\beta}|}{2}+\sig_1\eps_2(2\eps_1-1), &&|\ell_2|=\frac{|\tilde{\alpha}-\tilde{\beta}|}{2}-\sig_2\eps_2
\end{align}
(for a proof see lemma 10 in \cite{Mas}). So
\begin{align}
|\K_{1,1}^{-1}(\tilde{x_j},\tilde{y}_i)|\leq Ca^{(|\tilde{\beta}+\tilde{\alpha}|+|\tilde{\beta}-\tilde{\alpha}|)/4}(a^{\eps_2+(1-\eps_2)(\sigma_1(2\eps_1-1)+\sigma_2)/2}+a^{1-\eps_2+(\eps_1(\sigma_1(2\eps_2-1)-\sigma_2))/2}).
\end{align}
One can then check
\begin{align}
\eps_2+(1-\eps_2)(\sigma_1(2\eps_1-1)+\sigma_2)/2\geq -1\\
1-\eps_2+(\eps_1(\sigma_1(2\eps_2-1)-\sigma_2))/2\geq -1
\end{align}
and the above two expressions equal one when $\tilde{\alpha}=\tilde{\beta}=0$.
\end{proof}
\end{lemma}


\end{document}